\documentclass{amsart}[12 pt]

\usepackage{latexsym}

\usepackage{amsmath,amssymb,dsfont,amsfonts}
\usepackage{mathrsfs}
\usepackage{verbatim}
\usepackage{graphicx}
\usepackage{graphics}
\usepackage{geometry}
\usepackage{booktabs}
\usepackage{enumitem}
\usepackage{xcolor}

{\theoremstyle{plain}}

\newtheorem{proposition}{Proposition}[section]
\newtheorem{theorem}{Theorem}[section]
\newtheorem{definition}{Definition}[section]
\newtheorem{corollary}{Corollary}[section]
\newtheorem{lemma}{Lemma}[section]
\newtheorem{remark}{Remark}[section]
\newtheorem{example}{Example}[section]

\numberwithin{equation}{section}

\allowdisplaybreaks

\begin{document}
\markboth{}{}

\title[Existence of densities for stochastic evolution equations driven by FBM] {Existence of densities for stochastic evolution equations driven by fractional Brownian motion}

\author{Jorge A. de Nascimento}

\address{Departamento de Matem\'atica, Universidade Federal da Para\'iba, 13560-970, Jo\~ao Pessoa - Para\'iba, Brazil}\email{jorgeacnas@gmail.com}

\author{Alberto Ohashi}

\address{Departamento de Matem\'atica, Universidade de Bras\'ilia, 70910-900, Bras\'ilia - Distrito Federal, Brazil}\email{amfohashi@gmail.com}

\thanks{}
\date{\today}


\begin{center}

\end{center}

\begin{abstract}
In this work, we prove a version of H\"{o}rmander's theorem for a stochastic evolution equation driven by a trace-class fractional Brownian motion with Hurst exponent $\frac{1}{2} < H < 1$ and an analytic semigroup on a given separable Hilbert space. In contrast to the classical finite-dimensional case, the Jacobian operator in typical solutions of parabolic stochastic PDEs is not invertible which causes a severe  difficulty in expressing the Malliavin matrix in terms of an adapted process. Under a H\"{o}rmander's bracket condition and some algebraic constraints on the vector fields combined with the range of the semigroup, we prove the law of finite-dimensional projections of such solutions has a density w.r.t Lebesgue measure. The argument is based on rough path techniques and a suitable analysis on the Gaussian space of the fractional Brownian motion.
\end{abstract}

\maketitle

\section{Introduction}

H\"{o}rmander's theorem is one of the central aspects of Probability theory with many applications to the theory of partial differential equations, ergodic theory, stochastic filtering and numerical analysis of stochastic processes. Let $X$ be a $d$-dimensional SDE written in Stratonovich form

\begin{equation}\label{fdsdeintrthesis}
dX_t = V_0(X_t)dt + \sum_{j=1}^n V_j(X_t)\circ dW^j_t,
\end{equation}
where $V_0, \ldots, V_n$ are smooth vector fields and $(W_i)_{i=1}^n$ is a standard $n$-dimensional Brownian motion. It is well known that if (\ref{fdsdeintrthesis}) is elliptic namely if, for every point $x\in \mathbb{R}^n$, the linear span
of $\{V_j(x); j=1,\ldots, n\}$ is $\mathbb{R}^d$, then the law of the solution of (\ref{fdsdeintrthesis}) (at a given time $t$) has a smooth density w.r.t Lebesgue measure. Based on the fundamental work of H\"{o}rmander, we know that much weaker conditions on the vector fields, the so-called parabolic H\"{o}rmander's bracket condition, also produce smoothness of the law of $X_t$. This phenomena is called hypoellipticity.

The main tool in proving hypoellipticity for finite-dimensional SDEs is based on Malliavin calculus. More precisely, let $\mathscr{M}_t$ be the Malliavin matrix

$$
\mathscr{M}_t = \big( \langle \mathbf{D}X^i_t, \mathbf{D} X^j_t\rangle_{L^2([0,T];\mathbb{R}^n)} \big)_{1\le i,j\le d},
$$
at a time $t>0$, where $\mathbf{D}X^i_t$ is the Gross-Sobolev-Malliavin derivative of $X^i_t$ w.r.t the Brownian motion. In order to get suitable integrability of the the Malliavin matrix associated with $X_t$, the key idea is to connect $\mathscr{M}_t$ with the Jacobian $J_{s,t}; s\le t$ of the SDE constructed as follows. Denote by $\Phi_t$ the (random) solution map to (\ref{fdsdeintrthesis}) so that $X_t = \Phi_t(x_0)$. It is well-known that under mild integrability assumption, we do have a flow of smooth maps, namely a two parameter family of maps $\Phi_{s.t}$ such that $X_t = \Phi_{s,t}(X_s)$ for every $s\le t$ and such that $\Phi_{t,u}\circ \Phi_{s,t} = \Phi_{s,u}$ and $\Phi_t = \Phi_{0,t}$. For a given initial condition $x_0$, we then denote by $J_{s,t}$ the derivative of $\Phi_{s,t}$ evaluated at $X_s$.

Under rather weak assumptions, the Jacobian is invertible and this fact allows us to write
\begin{equation}\label{repmthesis}
\mathscr{M}_t = J_{0,t}\mathcal{C}_{t}J^*_{0,t},
\end{equation}
where
$$\mathcal{C}_t =\int_0^t J^{-1}_{0,s}V(X_s)V^*(X_s)\big(J^{-1}_{0,s}\big)^*ds,$$
and $V$ is the $d \times n$-matrix-valued function obtained by concatenating the vector fields $V_j$ for $j = 1, \ldots, n$. By representation (\ref{repmthesis}), the invertibility of $\mathscr{M}_t$ is equivalent to the invertibility of the so-called \textit{reduced} Malliavin matrix $\mathcal{C}_t$ given by the following quadratic form

$$\langle \mathcal{C}_t \xi, \xi\rangle = \sum_{j=1}^n \int_0^t\langle \xi, J^{-1}_{0,s}V_j (X_s)\rangle^2ds; \xi \in \mathbb{R}^d. $$
Then, It\^o's formula and Norris's lemma (\cite{norris}) combined with the parabolic H\"{o}rmander's bracket condition allow us to conclude hypoellipticity for finite-dimensional SDEs of the form (\ref{fdsdeintrthesis}).

The analysis of the hypoellipticity phenomena for stochastic partial differential equations (henceforth abbreviated by SPDE) is much harder. The main technical problem with the generalization of H\"{o}rmander's theorem to parabolic SPDEs is the fact that the Jacobian $J_{0,t}$ is typically \textit{not} invertible regardless the type of noise. The existence of densities for finite-dimensional projections of SPDEs driven by Brownian motion was firstly tackled by Baudoin and Teichmann \cite{baudoinTeich} where the linear part of the SPDE generates a group of bounded linear operators on a Hilbert space. In this case, the Jacobian becomes invertible. Shamarova \cite{shamarova} studies the existence of densities for a stochastic evolution equation driven by Brownian motion in 2-smooth Banach spaces. Recently, based on a pathwise Fubini theorem for rough
path integrals, Gerasimovics and Hairer \cite{gera} overcome the lack of invertibility of the Jacobian for SPDEs driven by Brownian
motion. They show that the Malliavin matrix is invertible on every finite-dimensional subspace and
jointly with a purely pathwise Norri’s type lemma, they prove that laws of finite-dimensional projections of SPDE solutions driven by Brownian motion admit smooth densities w.r.t Lebesgue measure. In contrast to \cite{baudoinTeich}, the authors are able to prove existence and smoothness of densities for truly parabolic systems generated by semigroups and SPDEs driven by Brownian motion under a priory integrability conditions on the Jacobian.

The goal of this paper is to prove the existence of densities for finite-dimensional projections for a SPDE driven by a trace-class fractional Brownian motion (henceforth abbreviated by FBM) with Hurst exponent $\frac{1}{2} < H < 1$. The novelty of our work is to handle the infinite-dimensional case jointly with the fractional case which requires a new set of ideas. For FBM driving noise with $H > 1/2$ and under ellipticity assumptions on the vector fields $\{V_i; 0\le i \le n\}$, the existence and smoothness of the density for SDEs are shown by Hu and Nualart \cite{hu} and Nualart and Saussereau \cite{nualart2}. The hypoelliptic case for $H > 1/2$ is treated by Baudoin and Hairer \cite{baudoin} based on previous papers of Nualart and Saussereau \cite{nualart3} and Hu and Nualart \cite{hu}. When $\frac{1}{4} < H < \frac{1}{2}$, the integrability of the Jacobian given by Cass, Litterer and Lyons \cite{cass2} yields smoothness of densities in the elliptic case. The hypoelliptic case was treated in a series of works by Cass and Friz \cite{cass3}, Cass, Friz and Victoir \cite{cass4} and culminating with Cass, Hairer, Litterer and Tindel \cite{cass2} who provide smoothness of densities for a wide class of Gaussian noises including FBM with $\frac{1}{4} < H < \frac{1}{2}$.

\subsection{Main result}
In this article, we aim to provide a version of H\"ormander's theorem for a SPDE of the form
\begin{equation}\label{SPDEthesis}
dX_t = \big(A(X_t) + F(X_t)\big)dt + G(X_t)dB_t,
\end{equation}
where $(A,\text{dom}(A)\big)$ is the infinitesimal generator of an analytic semigroup $\{S(t); t\ge 0\}$ on a separable Hilbert space $E$, $B$ is a trace-class FBM taking values on a separable Hilbert space $U$ with Hurst parameter $\frac{1}{2} < H < 1$ and $F,G$ are smooth coefficients. Let $\mathcal{T}:E\rightarrow \mathbb{R}^d$ be a bounded and surjective linear operator. The goal is to prove, under H\"{o}rmander's bracket conditions, that the law of $\mathcal{T}(X_t)$ has a density w.r.t Lebesgue for every $t>0$. In this article, we obtain the proof of this result under the additional assumption that the analytic semigroup has a dense range in $E$ at a given time $t>0$ which is satisfied in many concrete examples (see Remark \ref{oneexample}). Moreover, in order to overcome the lack of invertibility of the Jacobian operator, some algebraic constraints on the vector fields combined with the range of the semigroup are imposed (Assumption B). To the best of our knowledge, this is the first result of hypoellipticity (existence of densities) for SPDEs driven by FBM. The result is build on a carefully analysis of the It\^o map (solution map)

$$B\mapsto X(B)$$
defined on a suitable abstract Wiener space associated with a trace-class FBM $B$ with parameter $\frac{1}{2} < H < 1$ and taking values on suitable space of increments. By means of rough path techniques, it is shown that $B\mapsto X(B)$ is Fr\'echet differentiable and hence differentiable in the sense of Malliavin calculus. Even though the noise $B$ is more regular than Brownian motion (in the sense of H\"{o}lder regularity), the rough path formalism in the sense of Gubinelli \cite{gubinelli,gubinelli1} allows us to obtain better estimates for the It\^o map compared to the classical approach \cite{young} or other frameworks based on fractional calculus \cite{nualart1}.

Let us define

$$
G_0(x):= A(x) + F(x); x\in \text{dom}(A^\infty),
$$
where $\text{dom}(A^\infty) = \cap_{n\ge 1} \text{dom} (A^n)$ is equipped with the projective limit topology associated with the graph norm of $\text{dom}(A)$. Given the SPDE (\ref{SPDEthesis}), let $\mathcal{V}_k$ be a collection of vector fields given by

$$
\mathcal{V}_0 := \{G_i; i\ge 1\},\quad \mathcal{V}_{k+1}:=\mathcal{V}_k \cup \big\{ [G_j, V]; V\in \mathcal{V}_k~\text{and}~j\ge 0 \big\},
$$
where $G_i(x) := G(x)(e_i)$ for some orthonormal basis $(e_i)_{i=1}^\infty$ of $U$ and $[\cdot, \cdot]$ denotes the Lie bracket (see (\ref{liebracket})) between smooth vector fields on $\text{dom}(A^\infty)$. We also define the vector spaces $\mathcal{V}_k(x_0):=\text{span}\{V(x_0); V\in \mathcal{V}_k\}$ and we set

$$\mathcal{D}(x_0):=\cup_{k\ge 1}\mathcal{V}_k(x_0),$$
for each $x_0\in \text{dom}(A^\infty)$. Let us now state the main result of this work.

\begin{theorem}\label{ExistenceResultTHESIS}
Let $X^{x_0}$ be the SPDE solution of (\ref{SPDEthesis}) with a given initial condition $x_0\in \text{dom}(A^\infty)$. For a given $t\in (0,T]$, assume that $\mathcal{D}(x_0)$ and $S(t)E$ are dense subsets of $E$. Under assumptions H1-A1-A2-A3-B1-B2-C1-C2-C3, if $\mathcal{T}:E\rightarrow \mathbb{R}^d$ is a bounded linear surjective operator, then the law of $\mathcal{T}(X^{x_0}_t)$ has a density w.r.t Lebesgue measure in $\mathbb{R}^d$.
\end{theorem}


The remainder of this paper is organized as follows. In Section \ref{preliminariesSECTION}, we establish some preliminary results on the Gaussian space of trace-class FBM and the associated Malliavin calculus. Section \ref{MalliavinDiffSection} and \ref{jacobianSECTION} present the main technical results concerning regularity of the It\^o map in the sense of Malliavin calculus and the existence of the right-inverse of the Jacobian, respectively. Section \ref{ExistenceSection} presents the proof of Theorem \ref{ExistenceResultTHESIS}.


\section{Preliminaries}\label{preliminariesSECTION}

\subsection{Fractional powers of infinitesimal generators}
In this work, we make extensive use of the regularizing effects of an analytic semigroup. Throughout this article, $E$ is a given separable Hilbert space and $(A,\text{dom}(A)\big)$ is the infinitesimal generator of an analytic semigroup $\{S(t); t\ge 0\}$ on $E$ satisfying the following property: there exist constants $\lambda,M>0$ such that

$$\|S(t)\|\le M e^{-\lambda t}~\text{for all}~t\ge 0.$$

In this case, we can define the fractional power $\big((-A)^\alpha,\text{Dom}((-A)^\alpha)\big)$ for any $\alpha\in \mathbb{R}$ (see Sections 2.5 and 2.6 in \cite{pazy}). To keep notation simple, we denote $E_\alpha:=\text{Dom}((-A)^\alpha)$ for $\alpha >0$ equipped with the norm $|x|_{\alpha}:=\|(-A)^\alpha x\|_E$ which is equivalent to the graph norm of $(-A)^\alpha$. If $\alpha < 0$, let $E_\alpha$ be the completion of $E$ w.r.t to the norm $|x|_{\alpha}:=\| (-A)^\alpha x\|_E$. If $\alpha=0$, we set $E_\alpha = E$. Then, $(E_\alpha)_{\alpha\in \mathbb{R}}$ is a family of separable Hilbert spaces such that $E_\delta \hookrightarrow E_\alpha$ whenever $\delta \ge \alpha$. Moreover, $S(t)$ may be extended to $E_\alpha$ as bounded linear operators for $\alpha<0$ and $t\ge 0$. Moreover, $S(t)$ maps $E_\alpha$ to $E_\delta$ for every $\alpha\in \mathbb{R}$ and $\delta \ge 0$. We also denote $\|\cdot\|_{\beta\rightarrow \alpha}$ as the norm operator of the space of bounded linear operators $\mathcal{L}(E_\beta,E_\alpha)$ from $E_\beta$ to $E_\alpha$ and, with a slight abuse of notation, we set $\|\cdot\| = \|\cdot\|_{0\rightarrow 0}$. The space of bounded multilinear operators from the $n$-fold space $E^n_\alpha$ to $E_\alpha$ is equipped with the usual norm $\|\cdot\|_{(n),\alpha\rightarrow \alpha}$ for $\alpha \ge 0$.

\subsection{Preliminaries on the gaussian space of fractional Brownian motion}\label{gaussianspace}
Let us start our discussion by recalling some elementary facts on the fractional Brownian motion (FBM). The FBM with Hurst parameter $0 < H < 1$ is a centered Gaussian process with covariance

$$R_H(t,s):=\frac{1}{2}\big(s^{2H} + t^{2H} - |t-s|^{2H}\big).$$
Throughout this paper, we fix $1/2 < H < 1$. Let $\beta = \{\beta_t; 0\le t \le T\}$ be a FBM defined on a complete probability space $(\Omega,\mathcal{F},\mathbb{P})$. Let $\mathcal{E}$ be the set of all step functions on $[0,T]$ equipped with the inner product

$$\langle \mathds{1}_{[0,t]},\mathds{1}_{[0,s]}\rangle_{\mathcal{H}}:=R_H(t,s).$$
One can check (see e.g Chapter 5 in \cite{nualart} or Chapter 1 in \cite{mishura}) for every $\varphi,\psi\in \mathcal{E}$, we have

\begin{equation}\label{innerproduct}
\langle \varphi,\psi\rangle_{\mathcal{H}}= \alpha_H\int_0^T\int_0^T |r-u|^{2H-2}\varphi(r)\psi(u)dudr,
\end{equation}
where $\alpha_H := H(2H-1)$. Let $\mathcal{H}$ be the reproducing kernel Hilbert space associated with FBM, i.e., the closure of $\mathcal{E}$ w.r.t (\ref{innerproduct}). The mapping $\mathds{1}_{[0,t]}\rightarrow \beta_t$ can be extended to an isometry between $\mathcal{H}$ and the first chaos $\{\beta(\varphi); \varphi\in \mathcal{H}\}$. We shall write this isometry as $\beta(\varphi)$.

Let us define the following kernel

\begin{equation}\label{KHoperator}
K_H(t,s):=c_Hs^{1/2-H}\int_s^t (u-s)^{H-\frac{3}{2}}u^{H-\frac{1}{2}}du; s < t,
\end{equation}
where $c_H = \Big(\frac{H(2H-1)}{\text{beta}(2-2H,H-\frac{1}{2})}\Big)^{\frac{1}{2}}$ and $\text{beta}$ denotes the Beta function. From (\ref{KHoperator}), we have

$$\frac{\partial K_H}{\partial t}(t,s) = c_H\Big(\frac{t}{s}\Big)^{H-\frac{1}{2}}(t-s)^{H-\frac{3}{2}}.$$
Consider the linear operator $K^*_H:\mathcal{E}\rightarrow L^2([0,T];\mathbb{R})$ defined by

$$(K^*_H \varphi)(s):=\int_s^T \varphi(t)\frac{\partial K_H}{\partial t}(t,s)dt; 0\le s\le T.$$
We observe $(K^*_H \mathds{1}_{[0,t]})(s) = K_H(t,s)\mathds{1}_{[0,t]}(s)$. It is well-known (see e.g \cite{nualart}) that $K_H^*$ can be extended to an isometric isomorphism between $\mathcal{H}$ and $L^2([0,T];\mathbb{R})$. Moreover,

\begin{equation}\label{iso1}
\beta(\varphi) = \int_0^T (K^*_H\varphi)(t)dw_t; \varphi\in \mathcal{H},
\end{equation}
where

\begin{equation}\label{inverseBM}
w_t :=\beta\big((K^*_H)^{-1}(\mathds{1}_{[0,t]})\big)
\end{equation}
is a real-valued Brownian motion. From (\ref{iso1}), we can represent

$$\beta_t = \int_0^t K_H(t,s)dw_s; 0\le t \le T,$$
and (\ref{inverseBM}) implies both $\beta$ and $w$ generate the same filtration. Lastly, we recall that $\mathcal{H}$ is a linear space of distributions of negative
order. In order to obtain a space of functions contained in $\mathcal{H}$, we consider the linear space $|\mathcal{H}|$ as the space of measurable functions $f:[0,T]\rightarrow\mathbb{R}$ such that

\begin{equation}\label{moduloHnorm}
\| f\|^2_{|\mathcal{H}|}:=\alpha_H \int_0^T\int_0^T |f(t)||f(s)||t-s|^{2H-2}dsdt< \infty,
\end{equation}
for a constant $\alpha_H>0$. The space $|\mathcal{H}|$ is a Banach space with the norm (\ref{moduloHnorm}) and isometric to a subspace of $\mathcal{H}$ which is not complete under the inner product (\ref{innerproduct}). Moreover, $\mathcal{E}$ is dense in $|\mathcal{H}|$. The following inclusions hold true

\begin{equation}\label{inclusions}
L^{\frac{1}{H}}([0,T];\mathbb{R})\hookrightarrow |\mathcal{H}|\hookrightarrow \mathcal{H},
\end{equation}
where

\begin{equation}\label{niceinner}
\langle f,g \rangle_\mathcal{H} = \alpha_H \int_0^T\int_0^T |u-v|^{2H-2}f(u) g(v) dudv,
\end{equation}
for $f,g\in L^{\frac{1}{H}}([0,T];\mathbb{R})$. Moreover, there exists a constant $C$ such that

\begin{equation}\label{repniceinner}
\|f\|^2_\mathcal{H} =C\int_0^T |I^{H-\frac{1}{2}}_{T-}f(s)|^2ds,
\end{equation}
where $I^{H-\frac{1}{2}}_{T-}$ is the right-sided fractional integral given by
$$I^{H-\frac{1}{2}}_{T-}f(x):=\frac{1}{\Gamma(H-\frac{1}{2})}\int^T_x f(s)(s-x)^{H-\frac{3}{2}}ds; 0\le x\le T.$$
For further details, we refer the reader to Lemma 1.6.6 and (1.6.14) in \cite{mishura}.




\subsection{Malliavin calculus on Hilbert spaces}

Throughout this article, we fix a self-adjoint, non-negative and trace-class operator $Q:U\rightarrow U$ defined on a separable Hilbert space $U$. Then, there exists an orthonormal basis $\{e_i; i\ge 1\}$ of $U$ and eigenvalues $\{\lambda_i; i\ge 1\}$ such that

$$Q e_i = \lambda_i e_i; i\ge 1,$$
where $\text{trace}~Q = \sum_{k=1}^\infty \lambda_k < + \infty$. We assume that $\lambda_k >0$ for every $k\ge 1$. Let $U_0:= Q^{1/2}(U)$ be the linear space
equipped with the inner product

$$\langle u_0, v_0\rangle_0:= \langle Q^{-1/2}u_0, Q^{-1/2}v_0,\rangle_U; u_0,v_0\in U_0,$$
where $Q^{-1/2}$ is the inverse of $Q^{1/2}$. Then, $(U_0,\langle\cdot,\cdot \rangle_0)$ is a separable Hilbert space with an orthonormal basis $\{\sqrt{\lambda_k}e_k; k\ge 1\}$.




Let $W$ be a $Q$-Brownian motion given by

$$W_t:= \sum_{k\ge 1} \sqrt{\lambda}_k e_k w^k_t; t\ge 0,$$
where $(w^k)_{k\ge 1}$ is a sequence of independent real-valued Brownian motions. Let $(\beta^k)_{k\ge 1}$ be a sequence of independent FBMs, where $\beta^k$ is associated with $w^k$ via (\ref{iso1}), i.e.,

$$\beta^k_t = \int_0^t K_H(t,s)dw^k_s; 0 \le t\le T.$$
We then set

\begin{equation}\label{QFBM}
B_t := \sum_{k=1}^\infty \sqrt{\lambda_k} e_k \beta^k_t; 0\le t
\le T.
\end{equation}
For separable Hilbert spaces $E_1$ and $E_2$, let us denote $\mathcal{L}_2(E_1;E_2)$ as the space of all Hilbert-Schmidt operators from $E_1$ to $E_2$ equipped with the usual inner product. Let $\mathcal{F}$ be the sigma-field generated by $\{B(\varphi); \varphi \in\mathcal{H}\otimes \mathcal{L}_2(U_0, \mathbb{R})\}$ where $B:\mathcal{H}\otimes \mathcal{L}_2(U_0, \mathbb{R})\rightarrow L^2(\Omega,\mathcal{F},\mathbb{P})$ is the linear operator defined by

$$B(\Phi):=\int_0^T \Phi(t) dB_t:=\sum_{k=1}^\infty \int_0^T (K^*_H\Phi^k)(t)dw^k_t; \Phi\in\mathcal{H}\otimes \mathcal{L}_2(U_0, \mathbb{R}),$$
where

$$\Phi^i:= \Phi(\sqrt{\lambda_i}e_i); i\ge 1.$$
We recall the tensor product $\mathcal{H}\otimes \mathcal{L}_2(U_0, \mathbb{R})$ is isomorphic to $\mathcal{L}_2(U_0,\mathcal{H})$. The elements of $\mathcal{H}\otimes \mathcal{L}_2(U_0, \mathbb{R})$ are described by

$$\sum_{m,j=1}^\infty a_{mj}\sqrt{\lambda_m}e_m \otimes h_j,$$
where $(a_{mj})_{m,j}\in \ell^2(\mathbb{N}^2)$, $(h_j)$ is an orthonormal basis for $\mathcal{H}$ and we denote

$$ e \otimes h: y \in U_0 \mapsto\langle e,y \rangle_{U_0} h.$$
It is easy to check that $\mathbb{E}\big[B(\Phi) B(\Psi)\big]= \langle \Phi,\Psi\rangle_{\mathcal{L}_2(U_0, \mathcal{H})}$ for every $\Phi,\Psi\in \mathcal{L}_2(U_0, \mathcal{H})$. In this case, $\Big(\Omega, \mathcal{F},\mathbb{P}; \mathcal{L}_2(U_0, \mathcal{H})\Big)$ is the Gaussian space associated with the isonormal Guassian process $B$.

For Hilbert spaces $E_1$ and $E_2$, let $C^k_p(E_1;E_2)$ be the space of all functions $f:E_1\rightarrow E_2$ such that $f$ and all its derivatives have polynomial growth. Let $\mathcal{P}$ be the set of all cylindrical random variables of the form

\begin{equation}\label{cili1}
F = f(B(\varphi_1),\ldots, B(\varphi_m)),
\end{equation}
where $f\in C^\infty_p(\mathbb{R}^m;\mathbb{R})$ and $\varphi_i \in \mathcal{L}_2(U_0,\mathcal{H}); i=1, \ldots, m$ for some $m\ge 1$. The Malliavin derivative of an element of $F\in \mathcal{P}$ of the form (\ref{cili1}) over the Gaussian space $\Big(\Omega, \mathcal{F},\mathbb{P}; \mathcal{L}_2(U_0, \mathcal{H})\Big)$ is defined by

$$\mathbf{D} F: = \sum_{k=1}^m \frac{\partial}{\partial x_k}f(B(\varphi_1),\ldots, B(\varphi_m)) \varphi_k.$$
We observe

$$\langle \mathbf{D}F, h \rangle_{\mathcal{L}_2(U_0, \mathcal{H})} =\sum_{k=1}^m \frac{\partial}{\partial x_k}f(B(\varphi_1),\ldots, B(\varphi_m)) \langle \varphi_k, h\rangle_{\mathcal{L}_2(U_0, \mathcal{H})}$$
$$ = \frac{d}{d\epsilon}f\big( B(\varphi_1) + \epsilon \langle \varphi_1,h\rangle_{\mathcal{L}_2(U_0, \mathcal{H})} ,\ldots, B(\varphi_m) + \epsilon\langle\varphi_m,h\rangle_{\mathcal{L}_2(U_0, \mathcal{H})} \big)|_{\epsilon=0}.$$

 For a given separable Hilbert space $E$, let $\mathcal{P}(E)$ be the set of all cylindrical $E$-valued random variables of the form

$$F = \sum_{j=1}^n F_j h_j,$$
where $F_j \in \mathcal{P}$ and $h_j\in E$ for $j=1, \ldots, n$ and $n\ge 1$. We then define

$$\mathbf{D} F := \sum_{j=1}^n \mathbf{D} F_j \otimes h_j.$$

A routine exercise yields the following result.

\begin{lemma}
The operator $\mathbf{D}:\mathcal{P}(E)\subset L^p(\Omega; E)\rightarrow L^p\big(\Omega;\mathcal{L}_2(U_0, \mathcal{H})\otimes E\big)$ is closable and densely defined for every $p \ge 1$.
\end{lemma}
For an integer $k\ge 1$ and $p\ge 1$, let $\mathbb{D}^{1,p}(E)$ be the completion of $\mathcal{P}(E)$ w.r.t the semi-norm

$$\| F\|_{\mathbb{D}^{1,p}(E)}:= \Bigg[\mathbb{E}\|F\|^p_E + \mathbb{E}\|\mathbf{D}F\|^p_{\mathcal{L}_2(U_0,\mathcal{H})\otimes E}\Bigg]^{1/p}.$$
Let us now devote our attention to some criteria for checking when a given functional $F:\Omega \rightarrow E$ belongs to the Sobolev spaces $\mathbb{D}^{1,p}(E)$ for $p>1$. In the sequel, $loc$ denotes localization in the sense of \cite{nualart}.

\begin{lemma}\label{weakCH}
For a given $p >1$, assume that $F\in L^p_{loc}(\Omega; E)$ and $\langle F,x\rangle_E \in \mathbb{D}^{1,p}_{loc}(\mathbb{R})$ for every $x\in E$. If there exists $\xi\in L^p_{loc}\big(\Omega;\mathcal{L}_2(U_0;\mathcal{H})\otimes E\big)$ such that

\begin{equation}\label{weakch}
\Big\langle \mathbf{D} \langle F,u\rangle_E, h \Big\rangle_{\mathcal{L}_2(U_0;\mathcal{H})}  = \langle \xi(u), h \rangle_{\mathcal{L}_2(U_0;\mathcal{H})},
\end{equation}
for every $u\in E$ and $h\in\mathcal{L}_2(U_0;\mathcal{H})$, then $F\in \mathbb{D}^{1,p}_{loc}(E)$ and $\mathbf{D}F = \xi$.
\end{lemma}
\begin{proof}
Consider the Gaussian space $\Big(\Omega, \mathcal{F},\mathbb{P}; \mathcal{L}_2(U_0, \mathcal{H})\Big)$, take a localizing sequence $(\Omega_n,F_n)\in \mathcal{F}\times \mathbb{D}^{1,2}(\mathbb{R})$ such that $F_n = \langle F,u\rangle_E$ on $\Omega_n$ and $\Omega_n\uparrow \Omega$ as $n\rightarrow +\infty$. Then, apply Theorem 3.3 given by \cite{veraar}.
\end{proof}
In view of the H\"{o}lder path regularity of the underlying noise, it will be useful to play with Fr\'echet and Malliavin derivatives. In this case, it is convenient to realize $\mathbb{P}$ as a Gaussian probability measure defined on a suitable H\"{o}lder-type separable Banach space equipped with a Cameron-Martin space which supports possibly infinitely many independent FBMs. Let $\mathcal{C}^\infty_0(\mathbb{R}_+)$ be the space of smooth functions $w:[0,\infty)\rightarrow\mathbb{R}$ satisfying $w(0)=0$ and having compact support. Given $\gamma\in(0,1)$ and $\delta\in (0,1)$, we define for every $w\in \mathcal{C}^\infty_0(\mathbb{R}_+)$, the norm

$$\|w\|_{\mathcal{W^{\gamma,\delta}}}:=\sup_{t,s\in \mathbb{R}_+}\frac{|w(t)-w(s)|}{|t-s|^{\gamma}(1+|t| + |s|)^{\delta}}.$$
Let $\mathcal{W}^{\gamma,\delta}$ be the completion of $\mathcal{C}^\infty_0(\mathbb{R}_+)$ w.r.t $\|\cdot\|_{\mathcal{W}^{\gamma,\delta}}$. We also write $\mathcal{W}^{\gamma,\delta}_T$ when we restrict the arguments to the interval $[0,T]$. It should be noted that $\|\cdot\|_{\mathcal{W}^{\gamma,\delta}_T}$ is equivalent to the $\gamma$-H\"{o}lder norm on $[0,T]$ given by

$$|f(0)| + |f|_\gamma,$$
where
$$|f|_\gamma:=\sup_{0\le s < t\le T}\frac{|f(t)-f(s)|}{|t-s|^\gamma},~\frac{1}{2} < \gamma < 1.$$
Moreover, $\mathcal{W}^{\gamma,\delta}_T$ is a separable Banach space. Let $\lambda = (\lambda_i)_{i=1}^\infty$ be the sequence of strictly positive eigenvalues of $Q$. In addition to $\text{trace}~Q  = \sum_{i\ge 1}\lambda_i < \infty$, let us assume $\sum_{i\ge 1}\sqrt{\lambda_i} < \infty$. Let $\mathcal{W}^{\gamma,\delta,\infty}_{\lambda,T}$ be the vector space of functions $g:\mathbb{N}\rightarrow\mathcal{W}^{\gamma,\delta}_T$ such that

$$\|g\|_{\mathcal{W}^{\gamma,\delta,\infty}_{\lambda,T}}:=\sum_{i= 1}^\infty \sqrt{\lambda_i}\|g^i\|_{\mathcal{W}^{\gamma,\delta}_T} < \infty.$$
Clearly, $\mathcal{W}^{\gamma,\delta,\infty}_{\lambda,T}$ is a normed space.
\begin{lemma}
$\mathcal{W}^{\gamma,\delta,\infty}_{\lambda,T}$ is a separable Banach space equipped with the norm $\| \cdot\|_{\mathcal{W}^{\gamma,\delta,\infty}_{\lambda,T}}$.
\end{lemma}
\begin{proof}
Let $\|g_n-g_m\|_{\mathcal{W}^{\gamma,\delta,\infty}_{\lambda,T}}\rightarrow 0$ as $n,m\rightarrow +\infty$. Then, for $\epsilon>0$, there exists $N(\epsilon)$ such that

$$\sum_{i= 1}^\infty \sqrt{\lambda_i}\|g^i_n - g^i_m\|_{\mathcal{W}^{\gamma,\delta}_T} < \epsilon, $$
for every $n,m > N(\epsilon)$. Since $\mathcal{W}^{\gamma,\delta}_T$ is complete, then there exists $g:\mathbb{N}\rightarrow \mathcal{W}^{\gamma,\delta}_T$ defined by $g^i:=\lim_{n\rightarrow \infty} g^i_n$ in $\mathcal{W}^{\gamma,\delta}_T$ for each $i\ge 1$. By construction, we observe that

\begin{eqnarray*}
\sum_{i= 1}^\infty \sqrt{\lambda_i}\|g^i\|_{\mathcal{W}^{\gamma,\delta}_T} &\le& \sum_{i= 1}^\infty \sqrt{\lambda_i}\|g^i - g^i_n\|_{\mathcal{W}^{\gamma,\delta}_T} + \sum_{i= 1}^\infty \sqrt{\lambda_i}\|g^i_n\|_{\mathcal{W}^{\gamma,\delta}_T}\\
& &\\
&\le& \sum_{i= 1}^\infty \sqrt{\lambda_i}\frac{\epsilon}{\sqrt{2^i}} + \sup_{j\ge 1}\|g_j\|_{\mathcal{W}^{\lambda,\gamma,\infty}_{\delta,T}}\\
& &\\
&\le& \Bigg(\sum_{i=1}^\infty \lambda_i\Bigg)^{1/2}\sqrt{2}\epsilon + \sup_{n\ge 1}\| g_n\|_{\mathcal{W}^{\gamma,\delta,\infty}_{\lambda,T}} < \infty.
\end{eqnarray*}
For separability, let $\Big[\oplus_{j=1}^\infty \mathcal{W}^{\gamma,\delta}_T\Big]_{2} = \{f:\mathbb{N}\rightarrow \mathcal{W}^{\gamma,\delta}_T; \|f\|_2< \infty\}$ be the $l_2$-direct sum of the Banach spaces $\mathcal{W}^{\gamma,\delta}_T$, where

$$\|f\|_2 = \Bigg(\sum_{j=1}^\infty \| f^j\|^2_{\mathcal{W}^{\gamma,\delta}_T}\Bigg)^{1/2}.$$
Since $\text{trace}~Q < \infty$, then

\begin{equation}\label{str}
\| \cdot \|_{\mathcal{W}^{\gamma,\delta,\infty}_{\lambda,T}}\le (\text{trace}~Q)^{1/2} \| \cdot\|_2.
\end{equation}
Of course, $\cup_{n\ge 1}\oplus_{j=1}^n\mathcal{W}^{\gamma,\delta}_T\subset \oplus_{j=1}^\infty \mathcal{W}^{\gamma,\delta}_T$ and clearly $\cup_{n\ge 1}\oplus_{j=1}^n\mathcal{W}^{\gamma,\delta}_T$ is a dense subset of $\big[\oplus_{j=1}^\infty \mathcal{W}^{\gamma,\delta}_T\big]_2$. Since $\mathcal{W}^{\gamma,\delta}_{T}$ is separable, the previous argument shows $\Big[\oplus_{j=1}^\infty \mathcal{W}^{\gamma,\delta}_T\Big]_{2}$ is separable and hence (\ref{str}) implies $\mathcal{W}^{\gamma,\delta,\infty}_{\lambda,T}$ is separable as well.
\end{proof}







\begin{lemma}\label{AWSlemma}
If $\gamma\in \big(\frac{1}{2},H\big)$ and $\gamma+\delta \in (H,1)$, then there exists a Gaussian probability measure $\mu^\infty_{\gamma,\delta}$ on $\mathcal{W}^{\gamma,\delta,\infty}_{\lambda,T}$. Therefore, there exists a separable Hilbert space $\mathbf{H}$ continuously imbedded into $\mathcal{W}^{\gamma,\delta,\infty}_{\lambda,T}$ such that $\big(\mathcal{W}^{\gamma,\delta,\infty}_{\lambda,T},\mathbf{H},\mu^\infty_{\gamma,\delta}\big)$ is an abstract Wiener space.
\end{lemma}
\begin{proof}
From Lemma 4.1 in \cite{hairerohashi}, we know there exists a probability measure $\mu_{\gamma,\delta}$ on $\mathcal{W}^{\gamma,\delta}_T$ such that the canonical process is a FBM with Hurst parameter $\frac{1}{2} < H <1$ as long as $\gamma\in \big(\frac{1}{2},H\big)$ and $\gamma+\delta \in (H,1)$. Let $\mathcal{W}^{\gamma,\delta,\infty}_T := \prod_{j\ge 1} \mathcal{W}^{\gamma,\delta}_T$ be the countable product of the Banach spaces $\mathcal{W}_T^{\gamma,\delta}$ equipped with the product topology which makes $\mathcal{W}^{\gamma,\delta,\infty}_T$ as a topological vector space. Let $\mu_{\gamma,\delta}^\infty$ be the product probability measure $\otimes_{j\ge 1}\mu_{\gamma,\delta}$ over $\mathcal{W}^{\gamma,\delta,\infty}_T$ equipped with the usual product sigma-algebra. Then, $\mu^\infty_{\gamma,\delta}$ is a Gaussian probability measure (see e.g Example 2.3.8 in \cite{bogachev}). Moreover, we observe

$$\mu^\infty_{\gamma,\delta}\big(\mathcal{W}^{\gamma,\delta, \infty}_{\lambda,T}\big)=1.$$
Indeed, by construction, we can take a sequence of $\mu_{\gamma,\delta}$-independent FBMs $\beta^i; i\ge 1$. By using the modulus of continuity of FBM, it is well-known that $\mathbb{E}_{\mu_{\gamma,\delta}} \| \beta^i\|_{\mathcal{W}^{\gamma,\delta}_T}  = \mathbb{E}_{\mu_{\gamma,\delta}} \| \beta^1\|_{\mathcal{W}^{\gamma,\delta}_T}< \infty$ for every $i\ge 1$. Therefore

$$\mathbb{E}_{\mu^\infty_{\gamma,\delta}}\sum_{i=1}^\infty \lambda_i \| \beta^i\|_{\mathcal{W}^{\gamma,\delta}_T} =\mathbb{E}_{\mu_{\gamma,\delta}}\| \beta^1\|_{\mathcal{W}^{\gamma,\delta}_T} \sum_{i=1}^\infty \lambda_i < \infty,$$
and this proves that $\mu^\infty_{\gamma,\delta}$ is a Gaussian probability measure on the Banach space $\mathcal{W}^{\gamma,\delta, \infty}_{\lambda,T}$.  As a conclusion, this shows that we have an abstract Wiener space structure for $\mu^\infty_{\gamma,\delta}$.


\end{proof}
In the sequel, with a slight abuse of notation, we define $K^*_H: \mathcal{E}\otimes \mathcal{L}_2(U_0,\mathbb{R}) \rightarrow L^2\big([0,T];\mathcal{L}_2(U_0,\mathbb{R})\big)$ as follows

$$K^*_H(h\otimes \varphi)(s):=\int_s^Th(t)\frac{\partial}{\partial t}K_H(t,s)dt \varphi; \quad h\in \mathcal{E}, \varphi\in \mathcal{L}_2(U_0,\mathbb{R}).$$
Clearly,

$$\langle K^*_H(h_1\otimes \varphi_1), K^*_H(h_2\otimes \varphi_2) \rangle_{L^2([0,T];\mathcal{L}_2(U_0,\mathbb{R}))} = \langle (h_1\otimes \varphi_1), (h_2\otimes \varphi_2) \rangle_{\mathcal{H}\otimes \mathcal{L}_2(U_0,\mathbb{R})},$$
for every $h_1,h_2\in \mathcal{E}$ and $\varphi_1,\varphi_2\in \mathcal{L}_2(U_0,\mathbb{R})$ and hence we can extend $K^*_H$ to an isometric isomorphism from $\mathcal{H}\otimes \mathcal{L}_2(U_0,\mathbb{R})$ to $L^2\big([0,T];\mathcal{L}_2(U_0,\mathbb{R})\big)$. Let us also denote $\mathcal{K}_H:L^2\big([0,T];\mathcal{L}_2(U_0,\mathbb{R})\big)\rightarrow \mathbf{H}$ by

$$\mathcal{K}_H f(t):=\int_0^t K_H(t,s) f_s(\sqrt{\lambda_i}e_i)ds;0\le t\le T, i\ge 1,$$
for $f\in L^2\big([0,T];\mathcal{L}_2(U_0,\mathbb{R})\big)$. Here, $\mathbf{H}:= \text{Range}~\mathcal{K}_H$ is the Hilbert space equipped with the norm
$$\| \mathcal{K}_H f\|^2_{\mathbf{H}} := \int_0^T \|f_s\|^2_{\mathcal{L}_2(U_0,\mathbb{R})}ds=\sum_{i=1}^\infty  \|f(\sqrt{\lambda_i}e_i)\|^2_{L^2([0,T];\mathbb{R})} = \sum_{i=1}^\infty  \|\mathcal{K}_{H,1}f(\sqrt{\lambda_i}e_i)\|^2_{\mathbb{H}},$$
where $\mathbb{H}:=\text{Range}~\mathcal{K}_{H,1}$ and

$$\mathcal{K}_{H,1}g(t):=\int_0^t K_H(t,s)g(s)ds; 0\le t\le T,$$
for $g\in L^2([0,T];\mathbb{R})$. We recall (see Th 3.6 \cite{samko}) there exists a constant $C$ such that

$$\| \mathcal{K}_{H,1} g\|_{\mathcal{W}^{\gamma,\delta}_T}\le C \|g\|_{L^2([0,T];\mathbb{R})},$$
for every $g\in L^2([0,T];\mathbb{R})$. Therefore, Cauchy-Schwartz inequality yields

$$\|\mathcal{K}_H f\|_{\mathcal{W}^{\gamma,\delta,\infty}_{\lambda,T}}\le (\text{trace}~Q)^{1/2}\|\mathcal{K}_H f\|_{\mathbf{H}},$$
for every $f\in L^2\big([0,T];\mathcal{L}_2(U_0,\mathbb{R})\big)$. Let us set $\mathbb{P} = \mu^\infty_{\gamma,\delta}$ and $\Omega = \mathcal{W}^{\gamma,\delta,\infty}_{\lambda,T}$. Summing up the above computations, we conclude $\mathbf{H}$ is the Cameron-Martin space associated with $\mathbb{P}$ in Lemma \ref{AWSlemma}, namely

\begin{equation}\label{CMid}
\int_{\Omega}\exp(i \langle \omega, z\rangle_{\Omega,\Omega^*})\mathbb{P}(d\omega) = e^{-\frac{1}{2}\|z\|^2_{\mathbf{H}}}; z\in \Omega^*,
\end{equation}
where $\Omega^*$ is the topological dual of $\Omega$.


By applying Prop. 4.1.3 in \cite{nualart} (see also \cite{kusuoka}), we arrive at the following result. Let $\mathcal{R}_H: = \mathcal{K}_H \circ K^*_H$ be the injection of $\mathcal{L}_2(U_0;\mathcal{H})$ into $\Omega$. We observe $\mathcal{R}_H: \mathcal{L}_2(U_0;\mathcal{H})\rightarrow \Omega$ is a bounded operator with dense range.

\begin{corollary}\label{FMall}
If a random variable $Y:\Omega\rightarrow\mathbb{R}$ is Fr\'echet differentiable along directions in the Cameron-Martin space $\mathbf{H}$, then

$$h\mapsto Y(\omega + \mathcal{R}_H(h))$$
is Fr\'echet differentiable for each $\omega \in \Omega$. Moreover, $Y\in \mathbb{D}^{1,2}_{loc}(\mathbb{R})$ and

$$\nabla Y(\cdot)(\mathcal{R}_H h) = \langle \mathbf{D} Y, h\rangle_{\mathcal{L}_2(U_0,\mathcal{H})},$$
for every $h\in \mathcal{L}_2(U_0,\mathcal{H})$.
\end{corollary}

\section{Malliavin differentiability of solutions}\label{MalliavinDiffSection}
In this section, we discuss differentiability in the sense of Malliavin calculus (on the probability space defined on Lemma \ref{AWSlemma}) of SPDE mild $\mathbb{F}$-adapted solutions of

\begin{equation}\label{SPDE1}
dX_t = \big(A(X_t) + F(X_t)\big)dt + G(X_t)dB_t,
\end{equation}
in a separable Hilbert space $E$. Here, $\mathbb{F}$ is the filtration generated by an $U$-valued FBM $B$ of the form

$$B_t=\sum_{i=1}^\infty \sqrt{\lambda_i}e_i\beta^i_t; 0\le t\le T.$$
We will assume $\sum_{i=1}^\infty \lambda_i < \infty$ and additional regularity conditions: $\sum_{i=1}^\infty \sqrt{\lambda_i}< \infty$ and $\lambda_i >0$ for all $i\ge 1$. The coefficients $F:E\rightarrow E$ and $G:E\rightarrow \mathcal{L}(U;E)$ will satisfy suitable minimal regularity conditions (see Assumption H1) to ensure well-posedness of (\ref{SPDE1}). Let us define $G_i(x):= G(x) (e_i)$ for an orthonormal basis $(e_i)_{i\ge 1}$ of $U$. Then, we view the solution as

\begin{equation}\label{mildsol}
X_t = S(t)x_0 + \int_0^t S(t-s)F(X_s)ds + \int_0^t S(t-s)G(X_s)dB_s,
\end{equation}
where the $dB$ differential is understood in Young's sense \cite{young,gubinelli1}

$$\int_0^t S(t-s)G(X_s)dB_s = \sum_{i=1}^\infty \sqrt{\lambda_i}\int_0^t S(t-s)G_i(X_s)d\beta^i_s,$$
where the convergence of the sum is understood $\mathbb{P}$-a.s in $E$ in the sense of Lemma \ref{consLemma} below. The solution of (\ref{mildsol}) will take values on $E_\alpha$ for suitable $\alpha >0$.

In order to prove Fr\'echet differentiability, it is crucial to play with linear SPDE solutions living in Banach spaces which are ``sensible" to the H\"{o}lder -type norm of the noise space $\mathcal{W}^{\gamma,\delta}_{\lambda,T}$. For this purpose, we make use of the algebraic/analytic formalism developed by \cite{gubinelli} in the framework of rough paths. Even though we are working with a regular noise $1/2 < H < 1$, the techniques developed by \cite{gubinelli,gubinelli1} allow us to derive better estimates than the classical approach of \cite{young} or fractional calculus given by \cite{nualart1}.

\subsection{Algebraic integration}
For completeness of presentation, let us summarize the basic objects of \cite{gubinelli,gubinelli1} which will be important to us. At first, we fix some notation. For a given normed space $V$ equipped with a norm $\|\cdot\|_V$, $\mathcal{C}_k(V)$ is the set of all continuous functions $g:[0,T]^k\rightarrow V$ such that $g_{t_1\ldots t_k}=0$ whenever $t_i=t_{i+1}$ for some $i\le k-1$. We define $\delta:\mathcal{C}_n(V) \rightarrow \mathcal{C}_{n+1}(V)$ by

$$(\delta F)_{t_1,\ldots t_{n+1}}:=\sum_{j=1}^{n+1}(-1)^j F_{t_1,\ldots \hat{t}_j\ldots t_{n+1}}; F\in \mathcal{C}_n(V),$$
where $\hat{t}_j$ means that this particular argument is omitted. We are mostly going to use the two special cases:
If $F\in \mathcal{C}_1(V)$, then

$$(\delta F)_{ts}= F_t - F_s; (t,s)\in [0,T]^2.$$
If $F \in \mathcal{C}_2(V)$, then

$$(\delta F)_{tsu}= -F_{su} + F_{tu} - F_{ts}; (t,s,u)\in [0,T]^3.$$
We measure the size of the increments by H\"{o}lder regularity defined as follows: For $f\in \mathcal{C}_2(V)$ and $\mu \ge 0$, let us define

$$\|f\|_{\mu,V}:=\sup_{s,t\in [0,T]}\frac{\|f_{st}\|_V}{|t-s|^\mu},$$
and the sets $\mathcal{C}^\mu_2(V):=\{f\in \mathcal{C}_2(V); \|f\|_{\mu,V} < \infty\}$, $\mathcal{C}^\mu_1(V):=\{f\in \mathcal{C}_1(V); \|\delta f\|_{\mu,V} < \infty\}$. In the same way, for $h\in \mathcal{C}_3(V)$, we set

$$\|h\|_{\gamma,\rho,V}:=\sup_{s,u,t\in [0,T]}\frac{\|h_{tus}\|_V}{|t-u|^\rho|s-u|^\gamma},$$

$$\|h\|_{\mu,V}:=\inf\Big\{\sum_{i}\|h_i\|_{\rho_i,\mu-\rho_i,V}; h = \sum_{i}h_i, 0 < \rho_i < \mu  \Big\},$$
where the last infimum is taken over all sequences $\{h_i\in \mathcal{C}_3(V)\}$ such that $h=\sum_{i}h_i$ and for all choices of numbers $\rho_i\in (0,\mu)$. Then, $\|\cdot\|_{\mu,V}$ is a norm on the space $\mathcal{C}_3(V)$, and we set

$$\mathcal{C}^\mu_3(V):=\{h\in \mathcal{C}_3(V); \|h\|_{\mu,V}< \infty\}.$$
Let us denote $\mathcal{Z}\mathcal{C}_k(V):=\mathcal{C}_k(V)\cap\text{Ker}\delta|_{\mathcal{C}_k(V)}$ and $\mathcal{B}\mathcal{C}_k(V):=\mathcal{C}_k(V)\cap\text{Range}~\delta|_{\mathcal{C}_{k-1}(V)}$. We have $\mathcal{Z}\mathcal{C}_{k+1}(V) = \mathcal{B}\mathcal{C}_{k+1}(V)$ for $k\ge 1$.

The convolutional increments will be defined as follows. Let $\mathcal{S}_n = \{(t_1,\ldots, t_n); T \ge t_1 \ge t_2\ge \ldots t_n \ge 0\}$. For a Banach space $V$, $\hat{\mathcal{C}}_n(V)$ denotes the space of continuous functions from $\mathcal{S}_n$ to $V$. We also need a modified version of basic increments distorted by the semigroup as follows: Let $\hat{\delta}: \hat{\mathcal{C}}_{n}(E)\rightarrow \hat{\mathcal{C}}_{n+1}(E)$ given by

 $$(\hat{\delta} F)_{t_1,\ldots t_{n+1}}:=(\delta F)_{t_1,\ldots t_{n+1}} - a_{t_1t_2} F_{t_2\ldots t_n},$$
where $a_{t_1t_2}:= S(t_1-t_2) - \text{Id}$ for $(t_1,t_2)\in \mathcal{S}_2$.

\

\noindent \textit{H\"{o}lder-type space of increments}. We need to define H\"{o}lder-type subspaces of $\hat{\mathcal{C}}_k$ for $1\le k\le 3$ associated with $E_\alpha; \alpha \in \mathbb{R}$. For $\mu \ge 0 $ and $g\in \hat{\mathcal{C}}_2(E_\alpha)$, we define the norm

$$\|g\|_{\mu,\alpha}:=\sup_{t,s\in \mathcal{S}_2}\frac{|g_{ts}|_{\alpha}}{|t-s|^\mu}$$
and the spaces
$$\hat{\mathcal{C}}_2^{\mu,\alpha}:=\{g\in \hat{\mathcal{C}}_2(E_\alpha); \|g\|_{\mu,\alpha} < \infty \},$$

$$\hat{\mathcal{C}}^{\mu,\alpha}_1 :=\{ f\in \hat{\mathcal{C}}_1(E_\alpha); \|\hat{\delta}f\|_{\mu,\alpha} < \infty\},$$

$$\mathcal{C}^{\mu,\alpha}_1 :=\{ f\in \hat{\mathcal{C}}_1(E_\alpha); \|\delta f\|_{\mu,\alpha} < \infty\}.$$


We denote $\hat{\mathcal{C}}^{0,\alpha}_1:=\hat{\mathcal{C}}_1(E_\alpha)$ equipped with the norm

$$\|f\|_{0,\alpha}:=\sup_{0\le t\le T}|f_t|_\alpha.$$

We also equip $\mathcal{C}^{\mu,\alpha}_1$ and $\hat{\mathcal{C}}^{\mu,\alpha}_1$ with the norms given, respectively, by

$$\|f\|_{\mathcal{C}^{\mu,\alpha}_1}:=\|f\|_{0,\alpha} + \|\delta f\|_{\mu,\alpha},$$

$$\|f\|_{\hat{\mathcal{C}}^{\mu,\alpha}_1}:=\|f\|_{0,\alpha} + \|\hat{\delta} f\|_{\mu,\alpha}.$$
We observe that

\begin{equation}\label{embed}
\hat{\mathcal{C}}^{\mu,\mu}_1\hookrightarrow \mathcal{C}^{\mu,0}_1,
\end{equation}
for every $\mu \in (0,1)$ due to the following estimate: For $\lambda \ge \mu$, we have

\begin{equation}\label{interpolationDELTA}
\|\delta f\|_{\mu,0}\le \|\hat{\delta}f\|_{\mu,\lambda} + C|T|^{\lambda-\mu}\|f\|_{0,\lambda},
\end{equation}
for every $f\in \hat{\mathcal{C}}^{\mu,\lambda}_1$ (see Lemma 2.4 in \cite{deya}).

Let us now consider the 3-increment spaces. If $h\in \hat{\mathcal{C}}_3(E_\alpha)$, we define

$$\|h\|_{\eta,\rho,\alpha}:=\sup_{t,u,s\in \mathcal{S}_3}\frac{|h_{tus}|_\alpha}{|t-u|^{\eta}|u-s|^{\rho}},$$

$$\|h\|_{\mu,\alpha}:=\inf\Big\{\sum_{i}\|h_i\|_{\rho_i, \mu-\rho_i,\alpha}; h = \sum_{i} h_i, 0 < \rho_i < \mu\Big\},$$
where the last infimum is taken over all sequences $h_i$ such that $h =\sum_{i} h_i$ and for all choices of the numbers $\rho_i \in (0,\mu)$. One can check $\|\cdot\|_{\mu,\alpha}$ is a norm and we define

$$\hat{\mathcal{C}}^{\mu,\alpha}_3:=\{ h\in \hat{\mathcal{C}}_3(E_\alpha); \|h\|_{\mu,\alpha} < \infty\}.$$
We also need H\"{o}lder-type spaces for operator-valued increments. For $\mu \ge 0$ and $\alpha,\beta\in \mathbb{R}$, we set

$$\hat{\mathcal{C}}^\mu_2 \mathcal{L}^{\beta,\alpha}:= \hat{\mathcal{C}}^\mu_2 \big(\mathcal{L}(E_\beta;E_\alpha)\big) = \{f:\mathcal{S}_2\rightarrow \mathcal{L}(E_\beta;E_\alpha); \|f\|_{\mu,\beta\rightarrow \alpha}< \infty\},$$
where

$$\|f\|_{\mu,\beta\rightarrow\alpha}:=\sup_{t,s\in \mathcal{S}_2}\frac{\|f_{ts}\|_{\beta\rightarrow \alpha}}{|t-s|^\mu}.$$
In order to work with the convolution sewing map (see \cite{gubinelli1}), we define

$$\mathcal{Z}\hat{\mathcal{C}}^{\mu,\alpha}_j: = \hat{\mathcal{C}}^{\mu,\alpha}_j \cap \text{ker}~\hat{\delta}|_{\hat{\mathcal{C}}_j}; j=2,3. $$
We recall $\text{Range}~\hat{\delta}|_{\hat{\mathcal{C}}_j} = \text{Ker}~\hat{\delta}|_{\hat{\mathcal{C}}_{j+1}}; j\ge 1$. Let us define $\mathcal{E}^{\mu,\alpha}_2:= \cap_{\epsilon\le \mu \wedge 1^-} \hat{\mathcal{C}}^{\mu-\epsilon,\alpha+\epsilon}_2$, where $\epsilon\le \mu \wedge 1^-$ means $\epsilon \in [0,\mu] \cap [0,1).$

\

\noindent \textit{Infinite-dimensional regularized noise}: We define

\begin{equation}\label{regnoise}
\mathbf{X}^{x,i}_{ts}:=S(t-s)(\delta x^i)_{ts}\sqrt{\lambda_i}; (t,s)\in \mathcal{S}_2,
\end{equation}
for $x = (x^i)_{i\ge 1}\in \mathcal{W}^{\gamma,\delta,\infty}_{\lambda,T}$ and $1/2 < \gamma < H < 1, \gamma+\delta \in (H,1)$. Let us now collect some important properties of the regularized noise.

\begin{lemma}\label{regNOISEreg}
The following properties hold true: $\mathbf{X}^{x,i}\in \hat{\mathcal{C}}^{\gamma}_2 \mathcal{L}^{\beta,\alpha}$ for $i\ge1$ and for every $(\alpha,\beta)\in \mathbb{R}^2$ such that $\beta \ge \alpha$. Moreover, there exists a constant $C$ which depends on $(\alpha ,\beta)$ such that

\begin{equation}\label{XHolder}
\sup_{(t,s)\in \mathcal{S}_2}\frac{\|\mathbf{X}^{x,i}_{ts}\|_{\beta \rightarrow \alpha}}{|t-s|^\gamma}\le C\sqrt{\lambda_i}\|x^i\|_{\mathcal{W}^{\gamma,\delta}_T},
\end{equation}
for every $i\ge 1$. Moreover, the following algebraic condition holds

\begin{equation}\label{algPR}
(\hat{\delta}\mathbf{X}^{x,i})_{tsu} = (\mathbf{X}^{x,i}a)_{tsu}; (t,s,u)\in \mathcal{S}_3,
\end{equation}
where $a_{su} = S(s-u) - \text{Id}; (s,u)\in\mathcal{S}_2$.
\end{lemma}
\begin{proof}
We observe if $\beta \ge \alpha$, then there exists $C_{\alpha,\beta}$ such that $\sup_{0\le r\le T}\|S(r)\|_{\beta\rightarrow\alpha}\le C_{\alpha,\beta} < \infty$. This is obviously true for $\alpha=\beta$. In case, $\beta > \alpha$, we observe if $x\in E_\beta$, then


\begin{eqnarray*}
|S(r)x|_\alpha &=& \|(-A)^\alpha S(r)x\|_E = \| S(r)(-A)^\alpha x\|_E = \|S(r) (-A)^{\alpha-\beta}(-A)^\beta x\|_E\\
& &\\
&\le& \|S(r) (-A)^{\alpha-\beta}\|_{0\rightarrow 0}|x|_\beta,
\end{eqnarray*}
because $(-A)^{\alpha-\beta}$ is a bounded operator on $E$ (see Section 2.6 in \cite{pazy}) whenever $\beta > \alpha$. Therefore, $\|S(r)\|_{\beta\rightarrow \alpha}\le \|S(r) (-A)^{\alpha-\beta}\|_{0\rightarrow 0}\le \|S(r)\|_{0\rightarrow 0}\|(-A)^{\alpha-\beta}\|_{0\rightarrow 0}$ for every $r\in [0,T]$. This proves our first claim. Therefore,

$$\|\mathbf{X}^{x,i}_{ts}\|_{\beta\rightarrow \alpha}\le \|S(t-s)\|_{\beta\rightarrow \alpha}|x^{i}_t-x^i_s|\sqrt{\lambda_i},$$
which implies (\ref{XHolder}). By definition,

\begin{eqnarray}
\nonumber (\hat{\delta} \mathbf{X}^{x,i})_{tsu} &=& \mathbf{X}^{x,i}_{tu} -  \mathbf{X}^{x,i}_{ts} - S(t-s) \mathbf{X}^{x,i}_{su}\\
\nonumber & &\\
\nonumber &=& S(t-u)(x^i_t-x^i_s)\sqrt{\lambda_i} - S(t-s)(x^i_t - x^i_s)\sqrt{\lambda_i}\\
\nonumber & &\\
\nonumber &=& S(t-s) [S(s-u) - \text{Id}](x^i_t-x^i_s)\sqrt{\lambda_i}=\mathbf{X}^{x,i}_{ts}a_{su}=(\mathbf{X}^{x,i}a)_{tsu}; (t,s,u)\in\mathcal{S}_3.
\end{eqnarray}
This shows (\ref{algPR}).
\end{proof}

In the sequel, for a given $\mu > 1$ and $\alpha \in \mathbb{R}$, $\hat{\Lambda}:\mathcal{Z}\hat{\mathcal{C}}^{\mu,\alpha}_3\rightarrow\mathcal{E}^{\mu,\alpha}_2$ is the sewing map as defined by Theorem 3.5 in \cite{gubinelli1}.
\begin{lemma}\label{consLemma}
Let us fix $x = (x^i)_{i\ge 1}\in \mathcal{W}^{\gamma,\delta,\infty}_{\lambda,T}$ where $1/2 < \gamma < H < 1, \gamma+\delta \in (H,1)$. Assume $z = (z^i)_{i\ge 1}$ satisfies $\sup_{i\ge 1}\| z^i\|_{\hat{\mathcal{C}}^{\eta,\beta}_1}< \infty$ for $\eta + \gamma >1$. Then

$$\mathcal{J}_{t_1t_2}(\hat{d}x z):= \sum_{i=1}^\infty \sqrt{\lambda_i} X^{x,i}_{t_1t_2} z^i_{t_2} + \sum_{i=1}^\infty\sqrt{\lambda_i}\hat{\Lambda}\big(X^{x,i} \hat{\delta} z^i\big)_{t_1t_2}$$
satisfies:

\

(i) There exists a constant $C$ such that

$$\| \hat{\delta}\mathcal{J}(\hat{d}x z)\|_{\gamma,\alpha}\le C \| x \|_{\mathcal{W}^{\gamma,\delta,\infty}_{\lambda,T}} \sup_{i\ge 1}\{\|z^i\|_{0,\beta} + \|\hat{\delta} z^i\|_{\eta,\beta}\},$$
for $\alpha\le \beta$.

\

(ii) $$\mathcal{J}_{t_1t_2}(\hat{d}x z) = \sum_{i=1}^\infty \sqrt{\lambda_i}\int_{t_2}^{t_1} S(t_1-u)z^i_u dx^i_u~ \text{in}~E_\alpha,$$
for each $(t_1,t_2)\in \mathcal{S}_2$.
\end{lemma}
\begin{proof}
The proof is a straightforward application of Lemma \ref{regNOISEreg} above and Lemma 3.2, Th. 3.5 and Cor 3.6 in \cite{gubinelli1}. We omit the details.
\end{proof}

\subsection{The It\^o map}
For a given $y_0=\psi\in E$, the It\^o map $x\mapsto y$ is defined as the solution of the equation
$$y_t = S(t-s)y_s + \int_s^t S(t-u)F(y_u)du + \mathcal{J}_{ts}(\hat{d}x G(y)); (t,s)\in \mathcal{S}_2,$$
which can be rewritten in terms of the increment operator $\hat{\delta}$

\begin{equation}\label{SDEequation}
(\hat{\delta}y)_{ts} = \int_s^t S(t-u)F(y_u)du + \mathcal{J}_{ts}(\hat{d}x G(y)); y_0 = \psi.
\end{equation}
Next, we list the basic assumptions needed for the existence and uniqueness of the SPDE solution. Before that, let us check that we may choose the correct set of parameters.

\begin{lemma}\label{choiceX}
For given $\frac{1}{2} < H < 1$ and $1/2 > \kappa > 1/4$, there exist $\tilde{\gamma}, \kappa_0$ satisfying $\tilde{\gamma} > \kappa_0 > \kappa > \frac{1}{4} $ with $\tilde{\gamma} + \kappa > 1$, $\tilde{\gamma} - \kappa \ge \kappa_0$ such that

\begin{equation}\label{choiceX1}
\mathbf{X}^{x,i}\in \hat{\mathcal{C}}^{\tilde{\gamma}}_2\mathcal{L}^{0,-\kappa} \cap  \hat{\mathcal{C}}^{\kappa_0}_2\mathcal{L}^{\kappa,\kappa},
\end{equation}
for every $i\ge 1$.
\end{lemma}
\begin{proof}
From Lemma \ref{regNOISEreg} and the definition of the spaces $\mathcal{W}^{\gamma,\delta}_T$, there exists a constant $C$ (which does not depend on $i\ge 1$) such that

$$\|\mathbf{X}^{x,i}\|_{H-\epsilon,0\rightarrow -\kappa}\le C \sqrt{\lambda_i}\|x^i\|_{\mathcal{W}^{H-\epsilon,\delta}_T},$$

$$\|\mathbf{X}^{x,i}\|_{H-\eta,\kappa\rightarrow \kappa}\le C \sqrt{\lambda_i}\|x^i\|_{\mathcal{W}^{H-\eta,\delta}_T},$$
for every $\kappa >0$, $\epsilon \in (0,H), \eta\in (0,H)$ and $\delta>0$ such that $H-\epsilon + \delta \in (H,1)$ and $H-\eta + \delta \in (H,1)$. For a given $\frac{1}{2}< H < 1$ and $ \frac{1}{2} > \kappa > \frac{1}{4}$, choose $\epsilon = \epsilon(\kappa,H)\in (0,H)$ in such way that

\begin{equation}\label{choosing1}
H-\epsilon + \kappa > 1.
\end{equation}
Choose $\eta = \eta (\epsilon,H)$ in such way that

\begin{equation}\label{choosing2}
\eta > \frac{1}{2} + \epsilon\quad \text{and}~H-\kappa > \eta.
\end{equation}
Of course, (\ref{choosing2}) implies $\frac{1}{2} + \epsilon < \eta < H-\kappa$. Choose $\delta$ accordingly to these conditions. We then set $\tilde{\gamma} = H-\epsilon$, $\kappa_0 = H-\eta$ where $\epsilon$ and $\eta$ satisfy (\ref{choosing1}) and (\ref{choosing2}). Then, by construction $\tilde{\gamma} + \kappa = H-\epsilon + \kappa > 1$ due to (\ref{choosing1}) and $\tilde{\gamma} > \kappa_0 > \kappa > \frac{1}{4}$ due to (\ref{choosing2}). Moreover, $\eta-\epsilon > \frac{1}{2} > \kappa > \frac{1}{4}$ so that

$$\tilde{\gamma} - \kappa_0 > \frac{1}{2}> \kappa > \frac{1}{4}.$$
Finally, we stress the choice of $\epsilon$ and $\eta$ does not depend on the index $i\ge 1$. This concludes the proof.

\end{proof}

Let us assume the following regularity assumptions on $F,G$:

\

\noindent \textbf{Assumption H1:} For $1/2 > \kappa > 1/4$, we assume that $F, G_i:E_{\kappa}\rightarrow E_{\kappa}$ is Lipschitz (uniformly in $i\ge 1$) and they have linear growth: there exists a constant $C$ such that

$$|G_i(x)|_\kappa\le C (1+ |x|_\kappa),\quad |F(x)|_\kappa\le C (1+ |x|_\kappa); x\in E_\kappa,$$
for every $i\ge 1$.
Furthermore, we suppose that $F,G_i$ can also be seen as
maps from $E$ to $E$, and when considered as such, it holds that $F,G_i$ are Lipschitz (uniformly in $i\ge 1$).

\

In the sequel, recall $\hat{\mathcal{C}}_1^{\kappa,\kappa}$ is the subspace of $\hat{\mathcal{C}}_1(E_\kappa)$ such that

$$\|z\|_{\hat{\mathcal{C}}^{\kappa,\kappa}_1}=\|z\|_{0,\kappa} + \|\hat{\delta}z\|_{\kappa,\kappa}< \infty.$$

In what follows, $x\in \mathcal{W}^{\tilde{\gamma},\delta,\infty}_{\lambda,T}$ where $\tilde{\gamma} + \delta \in (H,1), \frac{1}{2} < \tilde{\gamma} < H$,

\begin{equation}\label{Aindices}
\tilde{\gamma} > \kappa_0 > \kappa > \frac{1}{4},
\end{equation}
and $\tilde{\gamma} + \kappa > 1, \tilde{\gamma}-\kappa \ge \kappa_0$. By Lemma \ref{choiceX}, $\mathbf{X}^{x}$ satisfies (\ref{choiceX1}). By using Assumption H1, the following result is a straightforward application of Theorem 4.3 in \cite{gubinelli1}.

\begin{proposition}\label{exisuniq}
Under Assumption H1 and the choice of indices (\ref{Aindices}), for each $\psi\in E_\kappa$ there exists a unique global solution to (\ref{SDEequation}) in $ \hat{\mathcal{C}}_1^{\kappa,\kappa}$.
\end{proposition}
By noticing (see Lemma \ref{AWSlemma}) that $(\beta^i)_{i\ge 1}\in \mathcal{W}^{\tilde{\gamma}, \delta,\infty}_{\lambda,T}$ a.s, Proposition \ref{exisuniq} yields the following result.

\begin{proposition}
Under Assumption H1 and the choice of indices in (\ref{Aindices}), for each initial condition $x_0\in E_\kappa$, there exists a unique adapted process $X$ which is solution to (\ref{SPDE1}).
\end{proposition}

\subsection{Fr\'echet differentiability}
Let us now devote our attention to the Fr\'echet differentiability of the It\^o map
$$\Phi:\mathcal{W}^{\tilde{\gamma},\delta,\infty}_{\lambda,T}\rightarrow \hat{\mathcal{C}}^{\kappa,\kappa}_1\quad x\mapsto y,$$
where $y$ is the mild solution of (\ref{SDEequation}) driven by $x\in \mathcal{W}^{\tilde{\gamma},\delta,\infty}_{\lambda,T}$ and the indices $\tilde{\gamma}, \delta, \kappa_0,\kappa$ satisfy (\ref{Aindices}). Then, the Fr\'echet derivative is a mapping

$$\nabla \Phi:\mathcal{W}^{\tilde{\gamma},\delta,\infty}_{\lambda,T}\rightarrow
\mathcal{L}\big(\mathcal{W}^{\tilde{\gamma},\delta,\infty}_{\lambda,T};\hat{\mathcal{C}}^{\kappa,\kappa}_1\big).$$
The importance of Fr\'echet differentiability lies on the following argument: Once we have Fr\'echet differentiability of the It\^o map $x\mapsto y$, we shall use the Fr\'echet derivative chain rule to infer that $\langle X_t, h\rangle_E$ is Fr\'echet differentiable along the direction of the Cameron-Martin space $\mathbb{H}$ for a given $h\in E$ and $t\in [0,T]$. Hence, Corollary \ref{FMall} implies
$$\langle X_t, h\rangle_E\in \mathbb{D}^{1,2}_{loc}(\mathbb{R}).$$
Then, we must use Lemma \ref{weakCH} and try to conclude a representation. We follow the idea contained in the work of Nualart and Saussereau \cite{nualart3}. At first, we list a set of assumptions on the vector fields which will be important in this section.

\

\noindent \textbf{Assumption A1:} The vector fields, $G_i,F:E_\kappa\rightarrow E_\kappa$ are Fr\'echet differentiable and also differentiable when considering from $E$ to $E$. Moreover,

$$\sup_{i\ge 1}\sup_{x\in E_\kappa} \|\nabla G_i(x)\|_{\kappa\rightarrow \kappa} + \sup_{x\in E_\kappa}\|\nabla F(x)\|_{\kappa\rightarrow \kappa}< \infty,$$
$$\sup_{i\ge 1}\sup_{x\in E} \|\nabla G_i(x)\| + \sup_{x\in E}\|\nabla F(x)\| < \infty.$$

\

\noindent \textbf{Assumption A2:}

$$ \sup_{i\ge 1}\sup_{g\in E}\|\nabla^{(2)} G_i(g)\|_{(2),q\rightarrow q} + \sup_{f\in E}\|\nabla^{(2)} F(f)\|_{(2),\kappa\rightarrow\kappa} < \infty,$$
for $q\in \{0,\kappa\}$ and there exists a constant $C$ such that
$$\sup_{i\ge 1}\|\nabla G_i(f) -\nabla G_i(g)\| + \sup_{i\ge 1}\|\nabla^{(2)}G_i(f) -\nabla^{(2)}G_i(g)\big\|_{(2),0\rightarrow 0}\le C \|f-g\|_E,$$
for every $f,g\in E$.

\


At first, it is necessary to investigate flow properties of linear equations. We start with the following corollary whose proof is entirely analogous to Proposition \ref{exisuniq}, so we omit the details.

\begin{corollary}\label{linearequation1}
Suppose $F,G$ satisfy Assumptions A1-H1 and let us fix $(x,y)\in \mathcal{W}^{\tilde{\gamma},\delta,\infty}_{\lambda,T}\times \hat{\mathcal{C}}^{\kappa,\kappa}_1$ and $t_0\in [0,T]$. Then, for every $\eta \in \hat{\mathcal{C}}^{\kappa,\kappa}_1$,

$$v_t = \eta_t + \int_{t_0}^t S(t-s)\nabla F(y_s)v_sds + \mathcal{J}_{tt_0}(\hat{d}x\nabla G (y)v)$$
admits a unique solution in $v\in \hat{\mathcal{C}}^{\kappa,\kappa}_1$ on the interval $[t_0,T]$.
\end{corollary}

The following lemma plays a key role on the Fr\'echet differentiability of the It\^o map.

\begin{lemma}\label{tecdeltaZ}
Let $[s_0,t_0]$ be a subset of $[0,T]$, let $Z_t = \sum_{i\ge 1}\sqrt{\lambda_i}\int_{s_0}^t S(t-s) z^i_s dx^i_s; s_0\le t\le t_0$ where $x\in \mathcal{W}^{\tilde{\gamma},\delta,\infty}_{\lambda,T}$ and assume $\sup_{i\ge1}\|z^i\|_{0,\eta} + \sup_{i\ge 1}\| \hat{\delta} z^i\|_{\zeta,\eta-\alpha} < \infty$ on the interval $[s_0,t_0]$ for some $\eta \ge 0$ where $0\le \alpha \le \text{min}(\zeta,\eta)$, $0 \le \zeta \le \tilde{\gamma}$ and $\tilde{\gamma} + \zeta >1$.  Then, there exists a constant $C$ which depends on $\eta$ and $\tilde{\gamma}$ such that

\begin{equation}\label{withgamma}
\|\hat{\delta}Z\|_{\tilde{\gamma},\eta}\le C \|x\|_{\mathcal{W}^{\tilde{\gamma},\delta,\infty}_{\lambda,T}}\Big\{\sup_{i\ge1}\|z^i\|_{0,\eta} + |t_0-s_0|^{\zeta-\alpha}\sup_{i\ge 1}\| \hat{\delta} z^i\|_{\zeta,\eta-\alpha}\Big\},
\end{equation}

\begin{equation}\label{withzeta}
\|\hat{\delta}Z\|_{\zeta,\eta}\le C \|x\|_{\mathcal{W}^{\tilde{\gamma},\delta,\infty}_{\lambda,T}}\Big\{|t_0-s_0|^{\tilde{\gamma}-\zeta}\sup_{i\ge1}\|z^i\|_{0,\eta} + |t_0-s_0|^{\tilde{\gamma}-\alpha}\sup_{i\ge 1}\| \hat{\delta} z^i\|_{\zeta,\eta-\alpha}\Big\},
\end{equation}
on the interval $[s_0,t_0]$.
\end{lemma}
\begin{proof}
In the sequel, $C$ is a constant which may differ from line to line. To keep notation simple, without loss of generality, we set $s_0=0, t_0=T$. We observe $(\hat{\delta}Z)_{ts} =\sum_{i\ge 1}\sqrt{\lambda_i}\int_s^t S(t-u) z^i_u dx^i_u $. From the proof of Lemma \ref{consLemma}, we know that

$$\int_s^t S(t-u) z^i_u dx^i_u = \mathbf{X}^{x,i}_{ts} z^i_{s} + \hat{\Lambda}\big( \mathbf{X}^{x,i}\hat{\delta} z^i\big)_{ts}; (t,s)\in \mathcal{S}_2,$$
where $\mathbf{X}^{x,i}\in \hat{\mathcal{C}}^{\tilde{\gamma}}_2\mathcal{L}^{\eta,\eta}$ due to Lemma \ref{regNOISEreg}. Then, checking the proof of Lemma \ref{consLemma}, we have $\mathbf{X}^{x,i}\hat{\delta} z^i\in \mathcal{Z}\hat{\mathcal{C}}^{\zeta+\tilde{\gamma},\theta}_3$ for $\theta\le \eta-\alpha$. Now,

$$\Big|\sum_{i\ge 1}\sqrt{\lambda_i}\int_s^t S(t-u) z^i_u dx^i_u\Big|_{\eta}\le \sum_{i\ge 1}\sqrt{\lambda_i}\Big|\int_s^t S(t-u) z^i_u dx^i_u\Big|_{\eta}\le C \sum_{i\ge 1}\sqrt{\lambda_i}\|\mathbf{X}^{x,i}_{ts}\|_{\eta\rightarrow \eta}|z^i_s|_\eta$$
$$+  \sum_{i\ge 1}\sqrt{\lambda_i}\Big|\hat{\Lambda}\big( \mathbf{X}^{x,i}\hat{\delta} z^i\big)_{ts}\Big|_{\eta}$$
\begin{equation}\label{difinter}
\le C \|x\|_{\mathcal{W}^{\tilde{\gamma},\delta,\infty}_{\lambda,T}} |t-s|^{\tilde{\gamma}}\sup_{i\ge 1}\|z^i\|_{0,\eta} + \sum_{i\ge 1}\sqrt{\lambda_i} \Big|\hat{\Lambda}\big( \mathbf{X}^{x,i}\hat{\delta} z^i\big)_{ts}\Big|_\eta; (t,s)\in\mathcal{S}_2.
\end{equation}
By applying the ``convolution" Sewing lemma (Th 3.5 in \cite{gubinelli1}), there exists a constant $C_{\zeta+ \tilde{\gamma}}$ such that

$$\| \hat{\Lambda}\mathbf{X}^{x,i}\hat{\delta} z^i\|_{\zeta+\tilde{\gamma}-\epsilon,\theta+\epsilon}\le C_{\zeta+\tilde{\gamma},\epsilon}\|\mathbf{X}^{x,i}\hat{\delta} z^i\|_{\zeta+\tilde{\gamma},\theta},$$
for every $\epsilon \in [0,\zeta+ \tilde{\gamma}] \cap [0,1)$. Take $\theta= \eta-\alpha$ and $\epsilon = \alpha$. Then,


\begin{equation}\label{difinter1}
\Big|\hat{\Lambda}\big( \mathbf{X}^{x,i}\hat{\delta} z^i\big)_{ts}\Big|_{\eta}\le C_{\zeta+\tilde{\gamma},\epsilon}\|\mathbf{X}^{x,i}\hat{\delta} z^i\|_{\zeta+\tilde{\gamma},\theta}|t-s|^{\zeta + \tilde{\gamma}-\epsilon}.
\end{equation}
On the other hand, $(\mathbf{X}^{x,i}\hat{\delta} z^i)$ is a 3-increment, where

$$\| \mathbf{X}^{x,i}\hat{\delta} z^i\|_{\zeta + \tilde{\gamma},\eta-\alpha}=\inf\Big\{\sum_{j}\|h_j\|_{\rho_j, \zeta+\tilde{\gamma}-\rho_j,\eta-\alpha};  \mathbf{X}^{x,i}\hat{\delta} z^i= \sum_{j} h_j, 0 < \rho_j <\zeta + \tilde{\gamma} \Big\},$$
and the last infimum is taken over all sequences $h_j$ such that $\mathbf{X}^{x,i}\hat{\delta} z^i =\sum_{j} h_j$ and for all choices of the numbers $\rho_j \in (0,\zeta + \tilde{\gamma})$. Here, we recall for any 3-increment $f$, we have

$$\|f\|_{\rho_j,\zeta+\tilde{\gamma}-\rho_j,\eta-\alpha}=\sup_{t,u,s\in \mathcal{S}_3}\frac{|f_{tus}|_{\eta-\alpha}}{|t-u|^{\rho_j}|u-s|^{\zeta+\tilde{\gamma}-\rho_j}}.$$
Take $h_i= \mathbf{X}^{x,i}\hat{\delta} z^i$ and $\rho_j = \tilde{\gamma}$. By definition, $(\mathbf{X}^{x,i}\hat{\delta} z^i)_{tus} = \mathbf{X}^{x,i}_{tu} \hat{\delta} z^i_{us}$, and then

\begin{eqnarray*}
\|\mathbf{X}^{x,i}\hat{\delta} z^i\|_{\zeta + \tilde{\gamma},\eta-\alpha}&\le& \sup_{t,u,s\in \mathcal{S}_3}\frac{|\mathbf{X}^{x,i}_{tu} \hat{\delta} z^i_{us}|_{\eta-\alpha}}{|t-u|^{\tilde{\gamma}}|u-s|^{\zeta}}\\
& &\\
&\le& \sup_{t,u,s\in \mathcal{S}_3}\frac{|\mathbf{X}^{x,i}_{tu}|_{\eta-\alpha\rightarrow \eta-\alpha} |\hat{\delta} z^i_{us}|_{\eta-\alpha}}{|t-u|^{\tilde{\gamma}}|u-s|^{\zeta}}\\
& &\\
&\le& C \|x^i\|_{\mathcal{W}^{\tilde{\gamma},\delta}_T}\|\hat{\delta} z^i\|_{\zeta,\eta-\alpha}.
\end{eqnarray*}
Then, (\ref{difinter1}) yields

\begin{equation}\label{difinter2}
\sum_{i\ge 1}\sqrt{\lambda_i}\Big|\hat{\Lambda}\big( \mathbf{X}^{x,i}\hat{\delta} z^i\big)_{ts}\Big|_\eta\le C_{\zeta+\tilde{\gamma},\alpha}|t-s|^{\zeta+\tilde{\gamma}-\alpha}\|x\|_{\mathcal{W}^{\tilde{\gamma},\delta,\infty}_{\lambda,T}}\sup_{i\ge 1}\|\hat{\delta} z^i\|_{\zeta,\eta-\alpha}.
\end{equation}

Finally, we shall plug (\ref{difinter2}) into (\ref{difinter}) and we conclude the proof of (\ref{withgamma}). By observing (\ref{difinter2}) and (\ref{difinter}), we conclude (\ref{withzeta}).


\end{proof}

\begin{lemma}
Assume that hypotheses H1-A1-A2 hold true. Let $y$ be the solution of (\ref{SDEequation}) with initial condition $\psi$ and driven by $x\in \mathcal{W}^{\tilde{\gamma},\delta,\infty}_{\lambda,T}$. Then, the mapping

$$L: \mathcal{W}^{\tilde{\gamma},\delta,\infty}_{\lambda,T}\times \hat{\mathcal{C}}^{\kappa,\kappa}_1\rightarrow \hat{\mathcal{C}}^{\kappa,\kappa}_1, $$
defined by

$$(x,y)\mapsto L(x,y)_t:=y_t - S_t\psi - \int_0^t S(t-s) F(y_s)ds  - \mathcal{J}_{t0}\big(\hat{d}(x) G(y)\big)$$
is Fr\'echet differentiable. In particular, for each $(x,y)\in \mathcal{W}^{\tilde{\gamma},\delta,\infty}_{\lambda,T}\times \hat{\mathcal{C}}^{\kappa,\kappa}_1$ and $(q,v)\in \mathcal{W}^{\tilde{\gamma},\delta,\infty}_{\lambda,T}\times \hat{\mathcal{C}}^{\kappa,\kappa}_1$, we have

\begin{equation}\label{partial1}
\nabla_1 L(x,y)(q)_t = -\mathcal{J}_{t0}\big(\hat{d}q G(y)\big),
\end{equation}

\begin{equation}\label{partial2}
\nabla_2 L(x,y)(v)_t = v_t - \int_0^t S(t-s)\nabla F(y_s)v_sds  - \mathcal{J}_{t0}\big(\hat{d}x \nabla G(y)v\big); 0\le t\le T.
\end{equation}
Moreover, for each $x\in \mathcal{W}^{\tilde{\gamma},\delta,\infty}_{\lambda,T}$, the mapping $\nabla_2 L(x,\Phi(x)):\hat{\mathcal{C}}^{\kappa,\kappa}_1\rightarrow \hat{\mathcal{C}}^{\kappa,\kappa}_1$ is an homeomorphism.
\end{lemma}
\begin{proof}
In the sequel, $C$ is a constant which may differ form line to line. By the very definition,

$$L(x+h,y+v)_t - L(x,y)_t =  v_t - \int_0^t S(t-u) \big[ F(y_u + v_u) - F(y_u)\big]du$$
$$-\sum_{i\ge 1}\sqrt{\lambda_i}\int_0^t S(t-u)(G_i(y_u+ v_u))dh^i_u - \sum_{i\ge 1}\sqrt{\lambda_i}\int_0^t S(t-u)(G_i(y_u+ v_u) -G_i(y_u) )dx^i_u.$$
Let us write the increments in terms of the Taylor formula,

$$F(y_u +v_u) - F(y_u)  = \nabla F(y_u)v_u + z_u(y,v),\quad G_i(y_u +v_u) - G_i(y_u)  = \nabla G_i(y_u)v_u + c^i_u(y,v),$$
$$G_i(y_u +v_u) = G_i(y_u) + e^i_u(y,v),$$
where

$$z_u(y,v):=\Bigg(\int_0^1 (1-r)\nabla^{(2)}F(y_u+rv_u)dr\Bigg)(v_u,v_u),\quad c^i_u(y,v):=\Bigg(\int_0^1 (1-r)\nabla^{(2)}G_i(y_u+rv_u)dr\Bigg)(v_u,v_u),$$
$$e^i_u(y,v):= \bigg(\int_0^1 \nabla G_i (y_u + r v_u)dr\bigg)v_u,$$
for $i\ge 1$ and $0\le u\le t$. Therefore,

$$L(x+h,y+v)_t - L(x,y)_t - \nabla_1 L(x,y)(h)_t - \nabla_2 L(x,y)(v)_t = R_1(y,v)_t+R_2(y,v)_t + R_3(y,v)_t, $$
where

$$R_1(y,v)_t:=-\int_0^t S(t-u)z_u(y,v)du,\quad R_2(y,v)_t:=-\sum_{i\ge 1}\sqrt{\lambda_i}\int_0^t S(t-u)c^i_u(y,v)dx^i_u,$$
$$R_3(y,v)_t:=-\sum_{i\ge 1}\sqrt{\lambda_i}\int_0^t S(t-u)e^i_u(y,v)dh^i_u. $$

We need to check

\begin{equation}
\| R_1(y,v) + R_2(y,v) + R_3(y,v)\|_{\hat{\mathcal{C}}^{\kappa,\kappa}_1} = o \big( \| h\|^2_{\mathcal{W}^{\tilde{\gamma},\delta,\infty}_{\lambda,T}} + \|v\|^2_{\hat{\mathcal{C}}^{\kappa,\kappa}_1}\big)^{\frac{1}{2}}.
\end{equation}
The first term is easy. Indeed, if the second order derivative of $F$ is bounded, then the norm of the bilinear form $z_u(y,v)$ can be estimated as follows $\|z_u(y,v)\|_{(2),\kappa\rightarrow \kappa}\le C |v_u|^2_{\kappa}\le C \|v\|^2_{\hat{\mathcal{C}}^{\kappa,\kappa}_1}$. Therefore,

$$
\|R_1(u,v)\|_{\hat{\mathcal{C}}^{\kappa,\kappa}_1}\le C\|v\|^2_{\hat{\mathcal{C}}^{\kappa,\kappa}_1}.
$$
Then,

\begin{equation}
\frac{\|R_1(u,v)\|_{\hat{\mathcal{C}}^{\kappa,\kappa}_1}}{\big( \| h\|^2_{\mathcal{W}^{\tilde{\gamma},\delta,\infty}_{\lambda,T}} + \|v\|^2_{\hat{\mathcal{C}}^{\kappa,\kappa}_1}\big)^{\frac{1}{2}}}\le \frac{\|R_1(u,v)\|_{\hat{\mathcal{C}}^{\kappa,\kappa}_1}}{\big(\|v\|^2_{\hat{\mathcal{C}}^{\kappa,\kappa}_1}\big)^{\frac{1}{2}}} \le  C\|v\|_{\hat{\mathcal{C}}^{\kappa,\kappa}_1}.
\end{equation}
Let us now estimate $R_2(y,v)$. At first, since $R_2(y,v)_0=0$, then

\begin{equation}\label{diffL}
\|R_2(y,v)\|_{\hat{\mathcal{C}}^{\kappa,\kappa}_1}\le (2+T^\kappa)\|\hat{\delta} R_2(y,v)\|_{\kappa,\kappa},
\end{equation}
where $-(\hat{\delta} R_2(y,v))_{ts} = \sum_{i\ge 1}\sqrt{\lambda_i}\int_s^t S(t-u)c_u^i(y,v)dx^i_u = \mathcal{J}_{ts}\big(\hat{d}x c(y,v)\big)$ so that $\|\hat{\delta} R_2(y,v)\|_{\kappa,\kappa} = \|\mathcal{J}\big(\hat{d}x c(y,v)\big)\|_{\kappa,\kappa}$. By Lemma \ref{tecdeltaZ}, there exists a constant $C$ such that

\begin{equation}\label{diffL1}
\|\mathcal{J}\big(\hat{d}x c(y,v)\big)\|_{\kappa,\kappa} \le C \|x\|_{\mathcal{W}^{\gamma,\delta,\infty}_{\lambda,T}}\Big\{\sup_{i\ge1}\|c^i(y,v)\|_{0,\kappa} +  \sup_{i\ge 1}\| \hat{\delta} c^i(y,v)\|_{\kappa,0}\Big\}.
\end{equation}
By definition,

$$(\hat{\delta}c^i(y,v))_{ts} =c^i_t(y,v) - c^i_s(y,v) + c^i_s(y,v) - S(t-s)c^i_s(y,v); (t,s)\in \mathcal{S}_2.$$
By viewing $\nabla^{(2)}G_i:E_\kappa\times E_\kappa \rightarrow E_\kappa$ as a bounded bilinear form where $\kappa >0$, we observe $c^i_s(y,v)\in E_\kappa$ and this little gain of spatial regularity allows us to estimate

\begin{equation}\label{depo1}
\|(\hat{\delta}c^i(y,v))_{ts}\|_{E}\le \| (\delta c^i(y,v))_{ts}\|_{E} + \| \big(S(t-s) - \text{Id}\big)c^i_s(y,v)\|_{E},
\end{equation}
where (see e.g Th 6.13 in \cite{pazy})

\begin{eqnarray}
\nonumber\| \big(S(t-s) - \text{Id}\big)c^i_s(y,v)\|_{E}&\le& C |t-s|^\kappa |c^i_s(y,v)|_{\kappa}\\
\nonumber& &\\
\label{depo2}&\le& C |t-s|^\kappa |v_s|^2_{\kappa}\le C|t-s|^{\kappa}\|v\|^2_{0,\kappa},
\end{eqnarray}
and the estimate (\ref{depo2}) is due to the boundedness $\sup_{i\ge 1}\sup_{a\in E_\kappa}\|\nabla^{(2)}G_i(a)\|_{(2),\kappa\rightarrow\kappa} < \infty$.

For each $i\ge 1$ and $u\in [0,t]$, we observe $\int_0^1 (1-r)\nabla^{(2)}G_i(y_u+rv_u)dr:E \times E \rightarrow E$ is a bounded bilinear form, so that we shall estimate

$$\| c^i_u(y,v) - c^i_{u'}(y,v)\|_{E}\le C\|v_{u'} - v_u\|_{E} \|v_{u'}\|_{E} + C \|v_{u'} - v_u\|_{E} \|v_{u}\|_{E} $$
$$+ \int_0^1 (1-r)\big\|\nabla^{(2)}G_i(y_u+rv_u) -\nabla^{(2)}G_i(y_{u'}+rv_{u'})\big\|_{(2),0\rightarrow 0} dr \|v_{u'}\|^2_{E}.$$
By using the Lipschitz property on the bilinear form $\nabla^{(2)} G_i$, we have

$$\int_0^1 (1-r)\big\|\nabla^{(2)}G_i(y_u+rv_u) -\nabla^{(2)}G_i(y_{u'}+rv_{u'})\big\|_{(2),0\rightarrow 0} dr\le C\int_0^1 (1-r)\| y_u-y_{u'}\|_{E} dr$$
$$+ \int_0^1 (1-r)r\| v_u-v_{u'}\|_{E} dr \le C \| y_u-y_{u'}\|_{E} + C\|v_u-v_{u'}\|_{E}.$$
Now, we observe $\hat{\mathcal{C}}^{\kappa,\kappa}_1 \hookrightarrow \mathcal{C}^{\kappa,0}_1$ (see (\ref{interpolationDELTA})) and $E_\kappa \hookrightarrow E$. Therefore,

\begin{equation}\label{diffL2}
\frac{\| c^i_u(y,v) - c^i_{u'}(y,v)\|_{E}}{|u-u'|^\kappa}\le C2\| v\|^2_{\hat{\mathcal{C}}^{\kappa,\kappa}_1} + C\| v\|^3_{\hat{\mathcal{C}}^{\kappa,\kappa}_1}.
\end{equation}
By assumption, $\sup_{i\ge 1}\sup_{p\in E_\kappa}\|\nabla^2G_i(p)\|_{(2),\kappa\rightarrow\kappa} < \infty$ and hence

\begin{equation}\label{diffL3}
\sup_{i\ge 1}\|c^i(y,v)\|_{0,\kappa}\le C \|v\|^2_{\hat{\mathcal{C}}^{\kappa,\kappa}_1}.
\end{equation}
Plugging (\ref{diffL3}), (\ref{diffL2}), (\ref{depo2}) and (\ref{depo1}) into (\ref{diffL1}), we conclude from (\ref{diffL}) that $\|R_2(y,v)\|_{\hat{\mathcal{C}}^{\kappa,\kappa}_1}\le C \|v\|^2_{\hat{\mathcal{C}}^{\kappa,\kappa}_1}$.

Let us now estimate $R_3(y,v)$. Similar to (\ref{diffL}), from Lemma \ref{tecdeltaZ}, we know there exists a constant $C$ such that

\begin{equation}\label{diffL4}
\|\mathcal{J}\big(\hat{d}x e(y,v)\big)\|_{\kappa,\kappa} \le C \|h\|_{\mathcal{W}^{\gamma,\delta,\infty}_{\lambda,T}}\Big\{\sup_{i\ge1}\|e^i(y,v)\|_{0,\kappa} +  \sup_{i\ge 1}\| \hat{\delta} e^i(y,v)\|_{\kappa,0}\Big\}.
\end{equation}
Clearly, Assumption A1 yields
\begin{equation}\label{diffL5}
\sup_{i\ge 1}\|e^i(y,v)\|_{0,\kappa}\le C \|v\|_{0,\kappa}\le C \|v\|_{\hat{\mathcal{C}}^{\kappa,\kappa}_1}.
\end{equation}
Similar to (\ref{depo1}) and (\ref{depo2}), we observe

\begin{equation}\label{depo3}
\|(\hat{\delta}e^i(y,v))_{ts}\|_{E}\le \| (\delta e^i(y,v))_{ts}\|_{E} + \| \big(S(t-s) - \text{Id}\big)e^i_s(y,v)\|_{E},
\end{equation}
where

\begin{equation}\label{depo4}
\| \big(S(t-s) - \text{Id}\big)e^i_s(y,v)\|_{E}\le C|t-s|^{\kappa}\|v\|_{0,\kappa}; (t,s)\in \mathcal{S}_2.
\end{equation}

The boundedness and the Lipschitz property on $\nabla G_i$ (Assumption A2) allow us to estimate

$$\| e^i_u(y,v) - e^i_{u'}(y,v)\|_{E}\le C \|v_u - v_{u'}\|_{E} + C\|v_{u'}\|_{E}\Big\{ \|y_u - y_{u'}\|_{E} + \|v_u - v_{u'}\|_{E}\Big\}.$$
Then, (\ref{interpolationDELTA}) yields

\begin{equation}\label{diffL6}
\frac{\| e^i_u(y,v) - e^i_{u'}(y,v)\|_{E}}{|u-u'|^\kappa} \le C\|v\|_{\hat{\mathcal{C}}^{\kappa,\kappa}_1} + \|v\|_{\hat{\mathcal{C}}^{\kappa,\kappa}_1}\{\|y\|_{\hat{\mathcal{C}}^{\kappa,\kappa}_1} + \|v\|_{\hat{\mathcal{C}}^{\kappa,\kappa}_1}\}.
\end{equation}
By using (\ref{diffL4}), (\ref{diffL5}), (\ref{depo3}), (\ref{depo4}) and (\ref{diffL6}), we infer

$$\|\mathcal{J}\big(\hat{d}x e(y,v)\big)\|_{\kappa,\kappa} = O\big( \|h\|_{\mathcal{W}^{\tilde{\gamma},\delta,\infty}_{\lambda,T}}\times \|v\|_{\hat{\mathcal{C}}^{\kappa,\kappa}_1} \big).$$

One can easily check $(x,y)\mapsto \nabla_1 L(x,y)\in \mathcal{L}(\mathcal{W}^{\tilde{\gamma},\delta,\infty}_{\lambda,T}; \hat{\mathcal{C}}^{\kappa,\kappa}_1)$ and
$(x,y)\mapsto \nabla_2 L(x,y)\in \mathcal{L}(\hat{\mathcal{C}}^{\kappa,\kappa}_1; \hat{\mathcal{C}}^{\kappa,\kappa}_1)$ are both continuous. Summing up all the above steps, we conclude $L$ is Fr\'echet differentiable and formulas (\ref{partial1}) and (\ref{partial2}) hold true. It remains to check $\nabla_2 L(x,\Phi(x))$ is a $\hat{\mathcal{C}}^{\kappa,\kappa}_1$-homeomorphism. By open mapping theorem, this is an immediate consequence of Corollary \ref{linearequation1} (which proves it is an isomorphism). The continuity can be easily checked so we left the details of this point to the reader.
\end{proof}

By applying implicit function theorem, $x\mapsto \Phi(x)$ is continuously Fr\'echet differentiable and the following formula holds true

\begin{equation}\label{abstractformulaDER}
\nabla \Phi(x) = -\nabla_2 L(x,\Phi(x))^{-1}\circ \nabla_1 L(x,\Phi(x)); x\in \mathcal{W}^{\tilde{\gamma},\delta,\infty}_{\lambda,T}.
 \end{equation}
The inverse operator yields $\nabla_2 L(x,\Phi(x)) \big( \nabla_2 L(x,\Phi(x))^{-1}(v)\big) = v$ so that

$$\nabla_2 L(x,\Phi(x))^{-1}(v)_t = v_t+ \int_0^t S(t-u) \nabla F (\Phi(x)_u)\nabla_2 L(x,\Phi(x))^{-1}(v)_udu$$
$$+\sum_{i\ge 1}\sqrt{\lambda_i}\int_0^t S(t-u) \nabla G_i \big(\Phi(x)_u\big)\nabla_2 L(x,\Phi(x))^{-1}(v)_udx^i_u; 0\le t\le T,$$
for each $v\in\hat{\mathcal{C}}^{\kappa,\kappa}_1$. Therefore, for each $x,h\in \mathcal{W}^{\tilde{\gamma},\delta,\infty}_{\lambda,T}$, $\nabla \Phi (x)(h)$ is the unique solution of

$$\nabla \Phi (x)(h)_t =\sum_{i\ge 1}\sqrt{\lambda_i}\int_0^t S(t-u) G_i \big(\Phi(x)_u\big)dh^i_u + \sum_{i\ge 1}\sqrt{\lambda_i}\int_0^t S(t-u) \nabla G_i \big(\Phi(x)_u\big) \nabla \Phi(x)(h)_u dx^i_u$$
\begin{equation}\label{frecheteq}
+ \int_0^t S(t-u) \nabla F\big(\Phi(x)_u\big) \nabla \Phi(x)(h)_u du; 0\le t\le T.
\end{equation}
Now, by Corollary \ref{linearequation1}, for each $u\in (0,T),x\in \mathcal{W}^{\tilde{\gamma},\delta,\infty}_{\lambda,T}$ and $i\ge 1$, the mapping $t\mapsto \Psi^i_{t,u}(x)$ given by

\begin{eqnarray}
\nonumber\Psi^i_{t,u}(x)&:=&S(t-u)G_i\big(\Phi(x)_u\big) + \sum_{j\ge1 }\sqrt{\lambda_j}\int_u^t S(t-\ell) \nabla G_j \big(\Phi(x)_\ell\big)\Psi^i_{\ell,u}(x)dx^j_\ell\\
\label{auxiliar}&+& \int_u^t S(t-\ell)\nabla F \big( \Phi(x)_\ell\big)\Psi^{i}_{\ell,u}(x)d\ell,
\end{eqnarray}
where $\Psi^i_{t,u}(x)=0$ for $u > t$, it is a well-defined element of $\hat{\mathcal{C}}^{\kappa,\kappa}_1$ over $[u,T]$. Let us denote $\Gamma^i_{x,u,u'}(t):= \Psi^i_{t,u}(x) - \Psi^i_{t,u'}(x)$ for $0\le u' \le u\le t\le T$. It is simple to check that

\begin{eqnarray}
\nonumber \Gamma^i_{x,u,u'}(t) &=& S(t-u)\Big[\Psi^i_{u,u}(x) - \Psi^i_{u,u'}(x)\Big] + \sum_{j\ge 1}\sqrt{\lambda_j}\int_u^t S(t-\ell)\nabla G_j(\Phi(x)_\ell)\Gamma^i_{x,u,u'}(\ell)dx^j_\ell\\
\nonumber & &\\
\nonumber &+&\int_{u}^t S(t-\ell) \nabla F(\Phi(x)_\ell)\Gamma^i_{x,u,u'}(\ell)d\ell.
\end{eqnarray}
The following technical lemma is important to derive an alternative representation for $\nabla \Phi(x)(h)$.

\begin{lemma}\label{menosK}
If Assumptions H1-A1-A2 hold true, then for each $x\in \mathcal{W}^{\tilde{\gamma},\delta,\infty}_{\lambda,T}$, there exists a positive constant $C$ which only depends on $\|x\|_{\mathcal{W}^{\tilde{\gamma},\delta,\infty}_{\lambda,T}}$ and $\|\delta \Phi(x)\|_{\kappa,\kappa}$ such that
$$|\Gamma^i_{x,u,u'}(t)|_\kappa \le C|\Psi^i_{u,u}(x) - \Psi^i_{u,u'}(x)|_\kappa,$$
for every $0 \le u' < u\le t \le T$ and $i\ge 1$.
\end{lemma}
\begin{proof}
Fix $0 \le u' < u \le T$, $i\ge 1$, $0\le \alpha \le \text{min}\{\kappa,\eta\}$ for $\eta\ge 0$. Let us denote $\varphi^i_{x,u,u'} = \big[\Psi^i_{u,u}(x) - \Psi^i_{u,u'}(x)\big]$. In the sequel, $C$ is a constant which may differ from line to line. Of course,

$$\|\hat{\delta}\Gamma^i_{x,u,u'}\|_{\kappa,\eta} \le \|\hat{\delta}S(\cdot-u)\varphi^i_{x,u,u'}\|_{\kappa,\eta} + \sum_{j\ge 1}\sqrt{\lambda_j}\Bigg\|\hat{\delta}\int_u^\cdot S(\cdot-\ell)\nabla G_j(\Phi(x)_\ell)\Gamma^i_{x,u,u'}(\ell)dx^j_\ell \Bigg\|_{\kappa,\eta}$$
$$+\Bigg\|\hat{\delta}\int_{u}^\cdot S(\cdot-\ell) \nabla F(\Phi(x)_\ell)\Gamma^i_{x,u,u'}(\ell)d\ell\Bigg\|_{\kappa,\eta} =: I_1+I_2+I_3.$$
At first, we observe $S(t-u)\varphi^i_{x,u,u'} - S(t-s)S(s-u)\varphi^i_{x,u,u'} = 0$ so that $I_1=0$.

By Lemma \ref{tecdeltaZ} (see (\ref{withzeta})), we observe there exists a constant $C$ such that

$$I_2\le  C \|x\|_{\mathcal{W}^{\tilde{\gamma},\delta,\infty}_{\lambda,T}}\Big\{ \sup_{j\ge 1}\| z^{ij}_{x,u,u'}\|_{0,\eta}|T-u|^{\tilde{\gamma}-\kappa} + |T-u|^{\tilde{\gamma}-\alpha}\sup_{j\ge 1}\|\hat{\delta} z^{ij}_{x,u,u'}\|_{\kappa,\eta-\alpha}\Big\},
$$
where $z^{ij}_{x,u,u'}(\ell) = \nabla G_j\big(\Phi(x)_\ell\big)\Gamma^i_{x,u,u'}(\ell)$. Let us take $\eta=\kappa=\alpha$. We observe

$$|z^{ij}_{x,u,u'}(\ell)|_\kappa \le \|\nabla G_j\big(\Phi(x)_\ell\big)\|_{\kappa\rightarrow \kappa}|\Gamma^i_{x,u,u'}(\ell)|_\kappa,$$
so that the boundedness assumption on the gradient $\nabla G_j$ yields
\begin{equation}\label{menosK2}
\|z^{ij}_{x,u,u'}\|_{0,\kappa}\le C\|\Gamma^i_{x,u,u'}\|_{0,\kappa}\le C \|\Gamma^i_{x,u,u'}\|_{\hat{\mathcal{C}}^{\kappa,\kappa}_1}.
\end{equation}
Triangle inequality yields
\begin{eqnarray}
\nonumber \big\|(\hat{\delta}z^{ij}_{x,u,u'})_{ts}\big\|_E&\le& \big\| [\nabla G_j\big(\Phi(x)_t\big) -\nabla G_j\big(\Phi(x)_s\big)]\big\|_{0\rightarrow 0} \big\|\Gamma^i_{x,u,u'}(t)\big\|_E\\
\nonumber & &\\
\nonumber &+& \|\nabla G_j\big(\Phi(x)_s\big)\|_{0\rightarrow 0}\|(\delta\Gamma^i_{x,u,u'})_{ts} \|_E\\
\nonumber & &\\
\nonumber &+& \big\|[\text{Id} - S(t-s)]\nabla G_j\big(\Phi(x)_s\big)\Gamma^i_{x,u,u'}(s) \big\|_E=:I_4+I_5+I_6,
\end{eqnarray}
where $\nabla G_j\big(\Phi(x)_s\big)\Gamma^i_{x,u,u'}(s)\in E_\kappa$. We observe

\begin{equation}\label{menosK3}
I_6\le C|t-s|^\kappa\|\Gamma^i_{x,u,u'}\|_{0,\kappa}.
\end{equation}
The imbedding (\ref{interpolationDELTA}) yields
\begin{equation}\label{menosK4}
I_5\le C|t-s|^\kappa \{\| \hat{\delta }\Gamma^i_{x,u,u'}\|_{\kappa,\kappa} + \|\Gamma^i_{x,u,u'}\|_{0,\kappa}\} = C|t-s|^\kappa \|\Gamma^i_{x,u,u'}\|_{\hat{\mathcal{C}}^{\kappa,\kappa}_1}.
\end{equation}
We observe

\begin{equation}\label{menosK5}
I_4\le C\|\delta \Phi(x)\|_{\kappa,\kappa}|t-s|^\kappa\|\Gamma^i_{x,u,u'}\|_{0,\kappa}.
\end{equation}
Summing up (\ref{menosK5}), (\ref{menosK4}) and (\ref{menosK3}), we have

\begin{equation} \label{menosK6}
\|\hat{\delta}z^{ij}_{x,u,u'}\|_{\kappa,0}\le C \Big(1+  \|\delta \Phi(x)\|_{\kappa,\kappa}\Big)    \|\Gamma^i_{x,u,u'}\|_{\hat{\mathcal{C}}^{\kappa,\kappa}_1}.
\end{equation}
This shows that

$$
I_2\le C\|x\|_{\mathcal{W}^{\tilde{\gamma},\delta,\infty}_{\lambda,T}}\Big\{\|\Gamma^i_{x,u,u'}\|_{\hat{\mathcal{C}}^{\kappa,\kappa}_1}|T-u|^{\tilde{\gamma}-\kappa} + |T-u|^{\tilde{\gamma}-\alpha}\Big(1+  \|\delta \Phi(x)\|_{\kappa,\kappa}\Big)    \|\Gamma^i_{x,u,u'}\|_{\hat{\mathcal{C}}^{\kappa,\kappa}_1} \Big\}.
$$
We notice that

\begin{eqnarray}
\nonumber I_3 &\le& C\sup_{u\le s < t\le T}\frac{|\int_{s}^t S(t-\ell) \nabla F(\Phi(x)_\ell)\Gamma^i_{x,u,u'}(\ell)d\ell|_\kappa}{|t-s|^\kappa}\\
\nonumber & &\\
\nonumber &=&C\|\Gamma^i_{x,u,u'}\|_{0,\kappa}|T-u|^{1-\kappa}.
\end{eqnarray}
Summing up the above inequalities, we have

\begin{eqnarray}
\nonumber\|\hat{\delta}\Gamma^i_{x,u,u'}\|_{\kappa,\kappa}&\le&   C\|x\|_{\mathcal{W}^{\tilde{\gamma},\delta,\infty}_{\lambda,T}}\Big\{ C|T-u|^{\tilde{\gamma}-\kappa}\Big(1+  \|\delta \Phi(x)\|_{\kappa,\kappa}\Big)    \|\Gamma^i_{x,u,u'}\|_{\hat{\mathcal{C}}^{\kappa,\kappa}_1} \nonumber\Big\}\\
\label{menosK7}& &\\
\nonumber&+& C\|\Gamma^i_{x,u,u'}\|_{0,\kappa}|T-u|^{1-\kappa}.
\end{eqnarray}
Therefore,

\begin{eqnarray}
\nonumber\|\Gamma^i_{x,u,u'}\|_{\hat{\mathcal{C}}^{\kappa,\kappa}_1}&\le& \|S(\cdot-u)\varphi^i_{x,u,u'}\|_{0,\kappa} +   C\|x\|_{\mathcal{W}^{\tilde{\gamma},\delta,\infty}_{\lambda,T}}\Big\{ C|T-u|^{\tilde{\gamma}-\kappa}\Big(1+  \|\delta \Phi(x)\|_{\kappa,\kappa}\Big)    \|\Gamma^i_{x,u,u'}\|_{\hat{\mathcal{C}}^{\kappa,\kappa}_1} \nonumber\Big\}\\
\label{menosK8}& &\\
\nonumber&+& C\|\Gamma^i_{x,u,u'}\|_{0,\kappa}|T-u|^{1-\kappa},
\end{eqnarray}
where $\|S(\cdot-u)\varphi^i_{x,u,u'}\|_{0,\kappa}\le C |\varphi^i_{x,u,u'}|_\kappa$. Finally, by working on a small interval and performing a standard patching argument, the estimate (\ref{menosK8}) allows us to conclude

$$
\|\Gamma^i_{x,u,u'}\|_{\hat{\mathcal{C}}^{\kappa,\kappa}_1}\le C_{x,y,T} |\varphi^i_{x,u,u'}|_\kappa,
$$
where $C_{x,y,T}= g\big(\|x\|_{\mathcal{W}^{\tilde{\gamma},\delta,\infty}_{\lambda,T}},\|\delta \Phi(x)\|_{\kappa,\kappa},T\big)$ for a function $g:\mathbb{R}^3_+\rightarrow \mathbb{R}_+$ growing with its arguments. This implies

$$|\Gamma^i_{x,u,u'}(t)|_\kappa = |\Psi^i_{t,u}(x) - \Psi^i_{t,u'}(x)|_\kappa\le C_{x,y,T}\big|\Psi^i_{u,u}(x) - \Psi^i_{u,u'}(x)\big|_\kappa.$$



\end{proof}

We are now in position to state the main result of this section. Let $\mathcal{C}^{\infty}_{0,\lambda}$ be the subset of $\mathcal{W}^{\tilde{\gamma},\delta,\infty}_{\lambda,T}$ composed by functions $g:\mathbb{N}\rightarrow \mathcal{C}^\infty_0$.

\begin{theorem}
Under Assumptions (H1-A1-A2), the It\^o map $x\mapsto \Phi(x)$ is continuously Fr\'echet differentiable and for each $x,h\in \mathcal{W}^{\tilde{\gamma},\delta,\infty}_{\lambda,T}$, $\nabla \Phi(x)(h)$ is the unique solution of the equation (\ref{frecheteq}). In addition, the following representation formula holds true
\begin{equation}\label{representationFRECHET}
\nabla \Phi(x)(h)_t = \sum_{i\ge 1}\sqrt{\lambda_i}\int_0^t \Psi^{i}_{t,u}(x)dh^i_u\in E_\kappa; 0\le t\le T,
\end{equation}
for each $(x,h)\in \mathcal{W}^{\tilde{\gamma},\delta,\infty}_{\lambda,T}\times \mathcal{C}^\infty_{0,\lambda} $.
\end{theorem}
\begin{proof}
The fact that $x\mapsto \Phi(x)$ is continuously Fr\'echet differentiable and it satisfies (\ref{frecheteq}) are consequences of (\ref{abstractformulaDER}). Obviously,

$$\sum_{i\ge 1}\sqrt{\lambda_i}\int_0^t \Psi^i_{t,u}(x)dh^i_u = \sum_{i\ge 1}\sqrt{\lambda_i}\int_0^t S(t-u) G_i \big(\Phi(x)_u\big)dh^i_u$$
$$+\sum_{i\ge 1}\sqrt{\lambda_i}\int_0^t \int_u^t S(t-\ell)\nabla F\big(\Phi(x)_\ell\big)\Psi^i_{\ell,u}(x) d\ell d h^i_u$$
$$+\sum_{i\ge 1}\sqrt{\lambda_i}\sum_{j\ge 1}\sqrt{\lambda_j}\int_0^t\int_u^t S(t-\ell)\nabla G_j \big(\Phi(x)_\ell\big)\Psi^i_{\ell,u}(x)d x^j_\ell dh^i_u; 0\le t\le T.$$
Let us fix $i\ge 1$ and $x\in \mathcal{W}^{\tilde{\gamma},\delta,\infty}_{\lambda,T}$. By Lemma \ref{menosK} and noticing

\begin{eqnarray}
\nonumber\Psi^i_{u,u}(x)-\Psi^i_{u,u'}(x)&=&G_i\big(\Phi(x)_u\big) - S(u-u')G_i\big(\Phi(x)_{u'}\big) - \sum_{j\ge1 }\sqrt{\lambda_j}\int_{u'}^u S(u-\ell) \nabla G_j \big(\Phi(x)_\ell\big)\Psi^i_{\ell,u'}(x)dx^j_\ell\\
\label{continuousPSI}& &\\
\nonumber&-& \int_{u'}^u S(u-\ell)\nabla F \big( \Phi(x)_\ell\big)\Psi^{i}_{\ell,u'}(x)d\ell; 0\le u' < u \le T; i\ge 1,
\end{eqnarray}
we clearly have $u\mapsto \Psi^i_{t,u}(x)$ is continuous, so that we shall apply Fubini's theorem to get

$$
\int_0^t \int_u^t S(t-\ell)\nabla F\big(\Phi(x)_\ell\big)\Psi^i_{\ell,u}(x) d\ell d h^i_u = \int_0^t \int_0^\ell S(t-\ell)\nabla F\big(\Phi(x)_\ell\big)\Psi^i_{\ell,u}(x) d h^i_u d\ell,
$$

$$
\int_0^t\int_u^t S(t-\ell)\nabla G_j \big(\Phi(x)_\ell\big)\Psi^i_{\ell,u}(x)d x^j_\ell dh^i_u = \int_0^t\int_0^\ell S(t-\ell)\nabla G_j \big(\Phi(x)_\ell\big)\Psi^i_{\ell,u}(x) dh^i_u d x^j_\ell; 0\le t\le T, i\ge 1.
$$
Therefore,

$$\sqrt{\lambda_i}\int_0^t \Psi^i_{t,u}(x)dh^i_u = \sqrt{\lambda_i}\int_0^t S(t-u) G_i \big(\Phi(x)_u\big)dh^i_u$$
$$+\int_0^t S(t-\ell)\nabla F\big(\Phi(x)_\ell\big)\sqrt{\lambda_i} \int_0^\ell\Psi^i_{\ell,u}(x)dh^i_u d\ell$$
$$+\sum_{j\ge 1}\sqrt{\lambda_j}\int_0^t S(t-\ell)\nabla G_j \big(\Phi(x)_\ell\big)\sqrt{\lambda_i}\int_0^\ell\Psi^i_{\ell,u}(x)dh^i_u d x^j_\ell ; 0\le t\le T.$$

At this point, in order to complete the proof of representation (\ref{representationFRECHET}), we only need to check

\begin{equation}\label{supPsi}
\sup_{i\ge 1}\sup_{0\le t\le T}\|\Psi^i_{t,\cdot}\|_{0,\kappa} < \infty.
\end{equation}
Since $\Psi^i$ is the solution of the linear equation (\ref{auxiliar}), a completely similar argument as detailed in the proof of Lemma \ref{menosK} yields

$$|\Psi^i_{t,u}(x)|_\kappa\le C_{x,y,T} \sup_{0\le r\le T}|S(r-u)G_i(\Phi(x)_u)|_\kappa,$$
for each $0\le u\le t\le T$, where $C_{x,y,T}= g\big(\|x\|_{\mathcal{W}^{\tilde{\gamma},\delta,\infty}_{\lambda,T}},\|\delta \Phi(x)\|_{\kappa,\kappa},T\big)$ for a function $g:\mathbb{R}^3_+\rightarrow \mathbb{R}_+$ growing with its arguments. This completes the proof.

\end{proof}

Let us now check Malliavin differentiability. Let us fix $t\in [0,T]$, $g\in E$ and we now look the mapping $ \mathcal{W}^{\tilde{\gamma},\delta,\infty}_{\lambda,T}\ni x\mapsto \langle \Phi(x)_t,g \rangle_E \in \mathbb{R} $. We can represent $\Phi(x)_t = \tau_t (\Phi(x))$, where $\tau_t:\hat{\mathcal{C}}^{\kappa,\kappa}_1\rightarrow E$ is the evaluation map which is a bounded linear operator for every $t\in [0,T]$. Then, the Fr\'echet derivative of $x\mapsto \Phi(x)_t$ is equal to the linear operator

$$\mathcal{W}^{\tilde{\gamma},\delta,\infty}_{\lambda,T}\ni f\mapsto \nabla \Phi (x)(f)_t\in E_\kappa\subset E.$$
Similarly, the Fr\'echet derivative of $x\mapsto \langle \Phi(x)_t,g \rangle_E$ is equal to

$$f\mapsto \langle \nabla \Phi(x)(f)_t, g\rangle_E.$$
We must find an $\mathcal{L}_2(U_0,\mathcal{H})$-valued random element $\omega\mapsto a(\omega)$ such that

$$\langle \nabla \Phi(\cdot)(\mathcal{R}_Hh)_t, g\rangle_E = \langle a(\cdot), h \rangle_{\mathcal{L}_2(U_0,\mathcal{H})}~a.s,$$
for each $h\in \mathcal{L}_2(U_0;\mathcal{H})$. If this is the case, then $a = \mathbf{D}_\cdot \langle X_t, g\rangle_E$ a.s. The following result is a straightforward consequence of the definition of $\mathcal{R}_H$.

\begin{lemma}
If $h\in \mathcal{C}^\infty_0$ and $\varphi\in \mathcal{L}_2(U_0;\mathbb{R})$, then

$$\mathcal{R}_H(h\otimes \varphi)\in \mathcal{C}^{\infty}_{0,\lambda}.$$
\end{lemma}

\begin{corollary}\label{weakderR}
Under the probability space given in Lemma \ref{AWSlemma}, the random variable $\langle X_t,g\rangle_E\in \mathbb{D}^{1,2}_{\text{loc}}(\mathbb{R})$ and $\mathbf{D}\langle X_t,g\rangle_E\in \mathcal{L}_2(U_0;\mathcal{H})$ is the Hilbert-Schmidt linear operator defined by

$$\mathbf{D}\langle X_t,g\rangle_E (\sqrt{\lambda_i}e_i) :=\langle \sqrt{\lambda_i}\Psi^i_{t,\cdot} , g \rangle_E~a.s,$$
for every $t\in [0,T]$ and $g\in E$.
\end{corollary}
\begin{proof}
Let us fix $t\in [0,T]$ and $g\in E$. By Lemma \ref{AWSlemma}, we shall represent $X_t(\omega) = \Phi(\omega)_t; (\omega,t)\in \mathcal{W}^{\tilde{\gamma},\delta,\infty}_{\lambda,T}\times [0,T]$. Since $\mathbf{H}\subset \mathcal{W}^{\gamma,\delta,\infty}_{\lambda,T}$, then
$$f\mapsto \langle X_t(f),g\rangle_E =\langle \Phi(f)_t,g\rangle_E $$
is Fr\'echet differentiable at all vectors $f\in \mathbf{H}$. In this case, Corollary \ref{FMall} yields $\langle X_t, g\rangle_E \in \mathbb{D}^{1,2}_{loc}(\mathbb{R})$ and

$$ \langle \nabla\Phi(\cdot)(\mathcal{R}_H v)_t,g\rangle_E = \langle \mathbf{D} \langle X_t, g\rangle_E, v\rangle_{\mathcal{L}_2(U_0;\mathcal{H})}~\text{locally in}~\Omega,$$
for each $v\in \mathcal{L}_2(U_0;\mathcal{H})$. Let us take $v=(h\otimes \varphi)\in \mathcal{C}^\infty_0\otimes \mathcal{L}_2(U_0.\mathbb{R})$. By using (\ref{representationFRECHET})

\begin{equation}
\Big\langle \nabla \Phi(\mathcal{R}_H v)_t, g\Big\rangle_E = \sum_{i\ge 1}\sqrt{\lambda_i}\Bigg\langle \int_0^t \Psi^{i}_{t,u}(x)d(\mathcal{R}_Hv^i)_u, g\Bigg\rangle_E.
\end{equation}
We observe

$$(\mathcal{R}_H v^i)'_u = \int_0^u\frac{\partial K_H}{\partial u}(u,s)K^*_H(h\otimes \varphi)_s(e_i)ds .$$

Therefore,
$$\int_0^t \Psi^{i}_{t,u}(x)d(\mathcal{R}_Hv^i)_u = \sqrt{\lambda_i}\int_0^t \Psi^{i}_{t,u}(x) \Bigg(\int_0^u\frac{\partial K_H}{\partial u}(u,s)K^*_H(h\otimes \varphi)_s(e_i)ds \Bigg)du $$
$$ = \sqrt{\lambda_i}\int_0^T \Psi^{i}_{t,u}(x) \Bigg(\int_0^u\frac{\partial K_H}{\partial u}(u,s)K^*_H(h\otimes \varphi)_s(e_i)ds \Bigg)du$$
$$ = \sqrt{\lambda_i}\int_0^T K^*_H(h\otimes \varphi)_s(e_i)\Bigg(\int_s^T\frac{\partial K_H}{\partial u}(u,s)\Psi^{i}_{t,u}(x)du \Bigg)ds.$$
Then,

$$\langle \nabla\Phi(\mathcal{R}_H v)_t,g\rangle_E = \sum_{i=1}^\infty \lambda_i\int_0^T K^*_H\big(h\otimes \varphi\big)_s(e_i) K^*_H \Big(\langle \Psi^i_{t,\cdot},g\rangle_E\Big)_s ds$$
$$ = \sum_{i=1}^\infty \lambda_i \Big\langle (h\otimes \varphi)(e_i), \langle \Psi^i_{t,\cdot} , g \rangle_E  \Big\rangle_\mathcal{H} $$
$$ = \sum_{i=1}^\infty \Big\langle (h\otimes \varphi)(\sqrt{\lambda_i}e_i), \langle \sqrt{\lambda_i}\Psi^i_{t,\cdot} , g \rangle_E  \Big\rangle_\mathcal{H},$$
where we observe (recall that this function is continuous (except at one point) for every $x\in \Omega$) $\langle \sqrt{\lambda_i}\Psi^i_{t,\cdot}(x) , g \rangle_E \in L^{\frac{1}{H}}([0,T];\mathbb{R})\subset |\mathcal{H}|$.
The candidate is then the linear operator defined by

\begin{equation}\label{derinner}
\mathbf{D}_\cdot \langle X_t, g\rangle_E (\sqrt{\lambda_i}e_i):=\langle \sqrt{\lambda_i}\Psi^i_{t,\cdot} , g \rangle_E~a.s.
\end{equation}

We observe (\ref{derinner}) provides a well-defined Hilbert-Schmidt operator from $U_0$ to $\mathcal{H}$ because

\begin{align}
\nonumber \sum_{i=1}^\infty \lambda_i \int_0^T \big| K^*_H\big(\langle \Psi^i_{t,\cdot}(\omega),g\rangle_E\big)_s & \big|^2ds \nonumber  ~=~ \sum_{i=1}^\infty \lambda_i \int_0^T \Big|\int_s^T\frac{\partial K_H}{\partial u}(u,s)\langle \Psi^{i}_{t,u}(\omega), g\rangle_E du\Big|^2ds\\
\nonumber &
\nonumber \le \int_0^T \Big(\int_s^T \big|\frac{\partial K_H}{\partial u}(u,s)\big|du\Big)^2ds\|g\|_E^2 \sup_{i\ge 1}\|\Psi^i_{t,\cdot}(\omega)\|^2_{0,\kappa} \text{Trace}~Q < \infty,
\end{align}
for each $\omega\in \Omega$. This concludes the proof.
\end{proof}
We are now able to state the main result of this section.

\begin{theorem}
If Assumptions H1-A1-A2 hold true, then $X_t\in \mathbb{D}^{1,2}_{loc}(E)$ for each $t\in [0,T]$ and the following formula holds
\begin{equation}\label{Malliavindereq}
\mathbf{D}_sX_t = S(t-s)G(X_s) + \int_s^t S(t-r)\nabla F(X_r) \mathbf{D}_sX_rdr + \sum_{i=1}^\infty\sqrt{\lambda_i}\int_s^t S(t-r)\nabla G_i(X_r) \mathbf{D}_sX_rd\beta^i_r,
\end{equation}
where $\mathbf{D}_sX_t = 0$ for $s > t$.
\end{theorem}
\begin{proof}
At first, we observe the postulated object $\mathbf{D} X_t$ takes values on $\mathcal{H}\otimes \mathcal{L}_2(U_0;\mathbb{R})\otimes E \equiv \mathcal{L}_2(U_0;\mathcal{H}\otimes E)$. Let us compute
$$
\Big\langle \mathbf{D} \langle X_t,g\rangle_E, v \Big\rangle_{\mathcal{L}_2(U_0;\mathcal{H})},
$$
for a given $g\in E$ and $v = (\varphi\otimes h)\in\mathcal{L}_2(U_0;\mathcal{H})$. By definition,

\begin{eqnarray}
\nonumber \Big\langle \mathbf{D} \langle X_t,g\rangle_E, v \Big\rangle_{\mathcal{L}_2(U_0;\mathcal{H})}&=&\sum_{i=1}^\infty  \Big \langle \langle \sqrt{\lambda_i}\Psi^i_{t,\cdot} , g \rangle_E, (\varphi\otimes h)(\sqrt{\lambda_i}e_i) \Big\rangle_\mathcal{H}\\
\nonumber & &\\
\nonumber &=& \sum_{i=1}^\infty \varphi(e_i) \lambda_i\Big \langle \langle \Psi^i_{t,\cdot} , g \rangle_E, h \Big\rangle_\mathcal{H}.
\end{eqnarray}

Let us define a Hilbert-Schmidt operator $\Psi_{t,\cdot}(\omega):U_0 \rightarrow L^{\frac{1}{H}}([0,T];E)\hookrightarrow \mathcal{H}\otimes E$ as follows

$$\Psi_{t,\cdot}(\omega)(\sqrt{\lambda_i}e_i):=\sqrt{\lambda_i}\Psi^i_{t,\cdot}(\omega); \omega \in \Omega.$$
By (\ref{inclusions}), there exists a constant $C$ such that

\begin{eqnarray}
\nonumber \sum_{i=1}^\infty \|\Psi_{t,\cdot}(\sqrt{\lambda_i}e_i)\|^2_{\mathcal{H}\otimes E}&\le & C \sum_{i=1}^\infty \|\Psi_{t,\cdot}(\sqrt{\lambda_i}e_i)\|^2_{L^{\frac{1}{H}}([0,T];E)}\\
\nonumber & &\\
\nonumber &\le& C \sum_{i=1}^\infty \lambda_i\|\Psi^i_{t,\cdot}\|^2_{0,\kappa}\le C \sup_{i\ge 1}\|\Psi^i_{t,\cdot}\|_{0,\kappa}.\text{trace}~Q < \infty~a.s.
\end{eqnarray}
We claim that $X_t \in \mathbb{D}^{1,2}_{loc}(E)$ and

\begin{equation}\label{pathwiseMalliavinder}
\mathbf{D}_\cdot X_t = \Psi_{t,\cdot}~a.s.
\end{equation}
Indeed, we observe $\Psi_t$ satisfies

$$\Psi_{t,s}= S(t-s)G(X_s) + \int_s^t S(t-r)\nabla F(X_r) \Psi_{r,s}dr + \sum_{i=1}^\infty\sqrt{\lambda_i}\int_s^t S(t-r)\nabla G_i(X_r) \Psi_{r,s}d\beta^i_r~a.s,$$
where $\Psi_{t,s}=0$ for $t < s$. Moreover,

\begin{eqnarray}
\nonumber \Big\langle \mathbf{D} \langle X_t,g\rangle_E, v \Big\rangle_{\mathcal{L}_2(U_0;\mathcal{H})}&=& \sum_{i=1}^\infty \varphi(e_i) \lambda_i\Big \langle \langle \Psi^i_{t,\cdot} , g \rangle_E, h \Big\rangle_\mathcal{H}\\
\nonumber &=&\sum_{i=1}^\infty\Big \langle \langle \Psi_{t,\cdot}(\sqrt{\lambda_i}e_i) ,g \rangle_E, \varphi(e_i) \sqrt{\lambda_i}h \Big\rangle_\mathcal{H}\\
\nonumber & &\\
\nonumber &=& \langle \mathbf{D}X_t g, v\rangle_{\mathcal{L}_2(U_0;\mathcal{H})}~a.s.
\end{eqnarray}
By applying Lemma \ref{weakCH} and Corollary \ref{weakderR}, we conclude the proof.


\end{proof}
\section{The right inverse of the SPDE Jacobian}\label{jacobianSECTION}
In this section, we investigate the existence of the right inverse of the Jacobian SPDE operator under some algebraic constraints on the vector fields combined with the range of the semigroup. From now on, it will be useful to make clear the dependence on the initial conditions of (\ref{SPDE1}). Let us write $X^{y}$ as the solution of (\ref{SPDE1}) for an initial condition $y\in E_\kappa$. In previous section, we made use of the $\hat{\mathcal{C}}^{\kappa,\kappa}_1$-topology to get differentiability of $X^{x_0}_t$ (in Malliavin's sense) for each initial condition at $x_0\in E_\kappa$. Even though we are interested in establishing the existence of densities $\mathcal{T}( X^{x_0}_t)$ for initial conditions on $\text{dom}~(A^\infty)\subset E_\kappa$, it is important to work with the solution map $E\rightarrow \mathcal{C}^{\alpha,0}_1$ given by

\begin{equation}\label{Emap}
y\mapsto X^y\in \mathcal{C}^{\alpha,0}_1,
\end{equation}
for some $\alpha > 1-H$. One drawback to keep the flow from $E_\kappa$ to $\hat{\mathcal{C}}^{\kappa,\kappa}_1$ is that $X^{x_0}$ does not belong to $\mathcal{C}^{\kappa,\kappa}_1$ and the best we can get is $X^{x_0}\in \mathcal{C}^{\kappa,0}_1$ a.s. For this purpose, we need to impose further regularity assumptions as described in Th 3.2 in \cite{nualart1}, which we list here for the sake of preciseness:

\

\noindent \textbf{Assumption A3:} There exists $\gamma_1, \gamma_2 \in (0,(2H-1)\wedge \frac{3}{4})$ and $c_1$ such that

$$\|S(r)G(x)\|\le \frac{c_1}{r^{\gamma_1}}(1+ \|x\|_E),$$

$$\|S(r) \big(G(x) - G(y)\big)\|\le \frac{c_1}{r^{\gamma_2}}\|x-y\|_E,$$
for every $x,y\in E$. Furthermore, for $\alpha > 1-H$, $\alpha <  \frac{1}{2}(1-\gamma_1)\wedge \frac{1}{2}(1-\gamma_2)$, assume there exist constants $c_2>0, \eta \in [0, 1-\alpha)$ and $\tilde{\beta}\in (\alpha,\frac{1}{2})$ such that

$$\|\nabla_x S(r)F(x)\| + \|\nabla_x S(r) G_i(x)\|\le c_2,$$
$$\|\nabla_x S(r)F(x) - \nabla_y S(r) F(y)\| +\|\nabla_x S(r)G_i(x) - \nabla_y S(r) G_i(y)\|\le \frac{c_2}{r^\eta}\|x-y\|_E, $$
$$\|\nabla_x (S(r)-S(s))F(x)\| + \|\nabla_x \big(S(r)-S(s)\big)G_i(x)\|\le c_2(r-s)^{\tilde{\beta}}s^{-\tilde{\beta}},$$
for every $r\in (0,T], 0 < s< r, x, y \in E$ and $i\ge 1$.

\

Under these conditions, the map (\ref{Emap}) is well-defined (see Th 3.2 in \cite{nualart1}). Moreover, it is not difficult to check the map $E \ni y\mapsto X^y\in \mathcal{C}^{\alpha,0}_1$ is Fr\'echet differentiable. In other words, the Jacobian

$$\mathbf{J}_{0\rightarrow t}(y;v) :=\nabla_y X^y_t(v)$$
is well defined for each $t\in [0,T]$ and $y,v\in E$. The proof of this fact is quite standard and the main arguments do not differ too much from the classical Brownian motion driving case (see e.g Th. 3.9 in \cite{gawarecki}), so we left the details to the reader. Moreover, (see \cite{nualart2}) for a given $\alpha \in \big(1-H, \frac{1}{2}\big)$ satisfying Assumption A3, we shall take $\kappa \in \big(\frac{1}{4}, \frac{1}{2}\big)$ with $\kappa=\alpha+\epsilon$ and $0 < \epsilon < \alpha$ such that

$$\mathcal{C}^{\kappa,0}_1 \subset W^{\alpha,\infty}(0,T;E),$$
where $W^{\alpha,\infty}(0,T;E)$ is the space of all measurable functions $f:[0,T]\rightarrow E$ such that

$$\|f\|_{\alpha,\infty}:= \Big(|f|_{0,0} + \sup_{0\le t\le T}\int_0^t \frac{\|f(t)-f(s)\|_{E}}{|t-s|^{1+\alpha}}\Big)< \infty.$$
Therefore, under Assumptions H1 and A3, the uniqueness of the flow described in Th 3.2 in \cite{nualart1} and (\ref{embed}) imply that all solutions $X^y$ generated by Proposition \ref{exisuniq} coincides with the ones given by \cite{nualart1} for every $y\in E_\kappa$. In addition, by applying Th 3.2 in \cite{nualart1}, $\mathbf{J}_{0\rightarrow t}(y;v)$ satisfies the following linear equation

\begin{equation}\label{jacobianSPDE}
\mathbf{J}_{0\rightarrow t}(y;v) = S(t)v + \int_0^t S(t-s) \nabla F(X^y_s)\mathbf{J}_{0\rightarrow s}(y;v)ds + \sum_{i=1}^\infty \sqrt{\lambda_i} \int_0^t S(t-s) \nabla G_i(X^y_s)\mathbf{J}_{0\rightarrow s}(y;v)d\beta^i_s.
\end{equation}
Of course, $v\mapsto \mathbf{J}_{0\rightarrow t}(y;v)\in \mathcal{L}(E; E)$ for each $t\in [0,T]$ and $y\in E$. Then, we shall see $t\mapsto \mathbf{J}_{0\rightarrow t}(y)$ as an operator-valued process as follows

$$\mathbf{J}_{0\rightarrow t}(y) = S(t) + \int_0^t S(t-s) \nabla F(X^y_s)\mathbf{J}_{0\rightarrow s}(y)ds + \sum_{i=1}^\infty \sqrt{\lambda_i} \int_0^t S(t-s) \nabla G_i(X^y_s)\mathbf{J}_{0\rightarrow s}(y)d\beta^i_s; 0\le t\le T.$$

\begin{remark}
Recall that infinitesimal generators of analytic semigroups are sectorial. Then, it is known (see e.g Corollary 2.1.7 in \cite{lunardi}) that $S(t)$ is one-to-one for every $t\ge 0$. We also observe the left-inverse linear operator $S(-t)$ of $S(t)$ defined on the subspace $S(t)E$ is, in general, unbounded.
\end{remark}

\begin{example}
Let $E = L^2(0,1)$ with Dirichlet boundary conditions. Take the orthonormal basis

$$e_n(x) = \sqrt{2}\text{sin}(\pi n x);0 < x < 1,$$
with eigenvalues $\lambda_n = \pi^2 n$. Then, the heat semigroup generated by the Laplacian $A = \Delta$ is given by

$$S(t)f = \sum_{n=1}^\infty e^{-\lambda_n t}\langle f, e_n\rangle_E e_n, $$
for $f\in E$. This is an analytic semigroup whose left-inverse is equal to

$$S(-t)g = \sum_{n=1}^\infty e^{\lambda_n t}\langle g, e_n\rangle_E e_n,$$
for $g\in S(t)E$.
\end{example}


In order to obtain a right-inverse operator-valued process for the Jacobian, we need to assume the following regularity conditions. In the sequel, we denote $S^-(t):=S(-t); t\ge 0$ where $S(-t)$ stands the left-inverse linear operator on $S(t)E$.

\

\noindent \textbf{Assumption B1}: Let $\alpha > 1-H$ be a constant as defined in Assumption A3. For each path $f\in \mathcal{C}^{\alpha,0}_1$,

$$\sup_{i\ge 1}\Big\{\|S^- \nabla G_i(f)S\|_{0,0\rightarrow 0} + \|\delta S^- \nabla G_i(f)S\|_{\mu,0\rightarrow 0}\Big\} < \infty,$$
for $\mu + \tilde{\gamma} > 1$, where $\frac{1}{2} <  \tilde{\gamma} < H$ satisfies (\ref{Aindices}).

\

\noindent \textbf{Assumption B2}: For each path $f\in \mathcal{C}^{\alpha,0}_1$, $ \|S^-\nabla F(f)S\|_{0,0\rightarrow 0} < \infty$.

\

In Assumptions B1-B2, we assume
\begin{equation}\label{Sminusalg}
\nabla F(w)z\in S(T)E~\text{and}~\nabla G_i(w)z\in S(T)E,
\end{equation}
for every $w,z\in E$ and $i\ge 1$.

\begin{remark}\label{S(T)}
Since $S(T)z = S(t)S(T-t)z$ for every $0\le t\le T$ and $z\in E$, then $S(T)E\subset S(t)E_\beta$ for every $0\le t\le T$ and $\beta\ge 0$.
\end{remark}

\begin{remark}
We stress we implicitly assume in Assumptions B1-B2 that $\nabla F(f_t)S(t)x\in S(t)E$ and $\nabla G_i(f_t) S(t)x\in S(t)E$ for every $t\ge0$, $x\in E$ and $i\ge 1$. This property holds true under (\ref{Sminusalg}) due to Remark \ref{S(T)}. In this case, taking into account that $S$ is a differentiable semigroup, then (see e.g Prop 3.12 in \cite{ito}) we have $\nabla F(w)z\in\cap_{n=1}^\infty\text{dom}(A^n)$ and $\nabla G_i(w)z\in \cap_{n=1}^\infty \text{dom}(A^n)$ for every $w,z\in E$ and $i\ge 1$.
\end{remark}


In the sequel, we freeze an initial condition $y\in E_\kappa$. Let us now investigate the existence of an operator-valued process $\mathbf{J}^+_{0\rightarrow t}(y)$ such that

$$\mathbf{J}_{0\rightarrow t}(y) \mathbf{J}^+_{0\rightarrow t}(y) =\text{Id}~a.s; 0\le t\le T,$$
where $\text{Id}$ is the identity operator on $S(t)E$. We start the analysis with the following equation

\begin{eqnarray}
\nonumber U_t(y) &=& -\int_0^t \big[\text{Id} + U_r(y)\big] S(-r)\nabla F (X^y_r)S(r)dr\\
\label{UeqINT1}& &\\
\nonumber&-&\sum_{i=1}^\infty \sqrt{\lambda_i}\int_0^t \big[\text{Id}+ U_r(y)\big] S(-r)\nabla G_i(X^y_r)S(r)d\beta^i_r~.
\end{eqnarray}

Let $\mathcal{C}_1^{\mu,0\rightarrow 0}$ be the linear space of $\mathcal{L}(E;E)$-valued functions $r\mapsto f_r$ such that

$$\|f\|_{\mathcal{C}_1^{\mu,0\rightarrow 0}}:=\|f\|_{0,0\rightarrow 0} + \|\delta f\|_{\mu,0\rightarrow 0} < \infty.$$

\begin{lemma} \label{fixpointjac}
Under Assumptions B1-B2, there exists a unique adapted solution $U(y)$ of (\ref{UeqINT1}) such that $U(y)\in \mathcal{C}^{\mu,0\rightarrow0}_1$ a.s for $\mu +\tilde{\gamma}>1$ and $0 < \mu < \tilde{\gamma}$.
\end{lemma}
\begin{proof}

For a given $g\in \mathcal{W}^{\tilde{\gamma},\delta,\infty}_{\lambda,T}$ and $w\in \mathcal{C}^{\alpha,0}_1$, let us define  $\Gamma:\mathcal{C}_1^{\mu,0\rightarrow 0}\rightarrow \mathcal{C}_1^{\mu,0\rightarrow 0}$ by

\begin{eqnarray}
\nonumber \Gamma(U)_t&:=& -\int_0^t \big[\text{Id} + U_r\big] S(-r)\nabla F (w_r)S(r)dr\\
\nonumber & &\\
\nonumber &-&\sum_{i=1}^\infty \sqrt{\lambda_i}\int_0^t \big[\text{Id}+ U_r\big] S(-r)\nabla G_i(w_r)S(r)dg^i_r~.
\end{eqnarray}
We claim that $\Gamma$ is a contraction map on a small interval $[0,T]$. Indeed, for $U,V\in \mathcal{C}_1^{\mu,0\rightarrow 0}$, if $q_t = \Gamma(U)_t - \Gamma(V)_t$, then

\begin{eqnarray}
\nonumber q_t &=& \int_0^t \big[V_r  - U_r\big] S(-r)\nabla F (w_r)S(r)dr\\
\nonumber & &\\
\nonumber &+& \sum_{i=1}^\infty \sqrt{\lambda_i}\int_0^t \big[V_r- U_r\big] S(-r)\nabla G_i(w_r)S(r)dg^i_r=:q^1_t +\sum_{i=1}^\infty q^{2,i}_t.
\end{eqnarray}

Assumption B2 implies the existence of a constant $C_F$ such that

\begin{eqnarray}
\nonumber\|q^1\|_{0,0\rightarrow 0}=\sup_{0\le t\le T}\|q^1_t\|&\le& \int_0^T\Big\|\big[V_r  - U_r\big] S(-r)\nabla F (w_r)S(r)\Big\|_{0\rightarrow 0}dr\\
\label{jacex1}& &\\
\nonumber&\le &C_{F}T \|U-V\|_{0,0\rightarrow 0},
\end{eqnarray}

\begin{eqnarray}
\nonumber \frac{\|q^1_t-q^1_s\|}{|t-s|^\mu}\le  C_F \|U-V\|_{0,0\rightarrow 0}|t-s|^{1-\mu}\le C_FT^{1-\mu}\|U-V\|_{0,0\rightarrow 0}.
\end{eqnarray}
Then,

\begin{equation}\label{jacex2}
\|\delta q^1\|_{\mu,0\rightarrow 0}\le C_FT^{1-\mu}\|U-V\|_{0,0\rightarrow 0}.
\end{equation}
Young-Loeve's inequality yields

\begin{eqnarray}
\nonumber\Big\|\sum_{i=1}^\infty (q^{2,i}_t-q^{2,i}_s)\Big\| &\le& \frac{1}{2^{\mu+\tilde{\gamma}} - 2} \sum_{i=1}^\infty\big\| \delta [V-U] S^-\nabla G_i(w)S\big\|_{\mu,0\rightarrow 0}\big\|g^i\big\|_{\mathcal{W}^{\tilde{\gamma},\delta}_{T}}|t-s|^{\mu+\tilde{\gamma}}\sqrt{\lambda_i}\\
\nonumber& &\\
\nonumber&+& \sum_{i=1}^{\infty}\big\|[V_s-U_s]S(-s)\nabla G_i(w_s)S(s)\big\| |(\delta g^i_{ts})|\sqrt{\lambda_i}\\
\nonumber& &\\
\nonumber&\le& \frac{1}{2^{\mu+\tilde{\gamma}} - 2}  \sum_{i=1}^\infty\big\| \delta [V-U] S^-\nabla G_i(w)S\big\|_{\mu,0\rightarrow 0}\big\|g^i\big\|_{\mathcal{W}^{\tilde{\gamma},\delta}_{T}}|t-s|^{\mu+\tilde{\gamma}}\sqrt{\lambda_i}\\
\label{jacex3}& &\\
\nonumber&+& \sum_{i=1}^\infty\big\|[V_s-U_s]S(-s)\nabla G_i(w_s)S(s)\big\|\|g^i\|_{\mathcal{W}^{\tilde{\gamma},\delta}_{T}}|t-s|^{\tilde{\gamma}}\sqrt{\lambda_i},
\end{eqnarray}
where by linearity, we have

\begin{eqnarray}
\nonumber\big\| \delta [V-U] S^-\nabla G_i(w)S\big\|_{\mu,0\rightarrow 0}&\le& \|S^-\nabla G_i(w)S\|_{0,0\rightarrow 0}\|\delta (V-U)\|_{\mu,0\rightarrow 0}\\
\nonumber& &\\
\nonumber&+& \|V-U\|_{0,0\rightarrow 0}\|\delta S^-\nabla G_i(w_\cdot)S\|_{\mu,0\rightarrow 0}\\
\label{jacex4}& &\\
\nonumber&\le& C_G\|V-U\|_{\mathcal{C}^{\mu,0\rightarrow 0}_1},
\end{eqnarray}
for a constant $C_G$ coming from Assumption B1. Summing up (\ref{jacex3}) and (\ref{jacex4}), we have

\begin{equation}\label{jacex5}
\Big\|\sum_{i=1}^\infty \delta q^{2,i}\Big\|_{\mu,0\rightarrow 0}\le C_G\|V-U\|_{\mathcal{C}^{\mu,0\rightarrow 0}_1}T^{\tilde{\gamma}}\|g\|_{\mathcal{W}^{\tilde{\gamma},\delta,\infty}_{\lambda,T}} + C_G\|V-U\|_{\mathcal{C}^{\mu,0\rightarrow 0}_1}\|g\|_{\mathcal{W}^{\tilde{\gamma},\delta,\infty}_{\lambda,T}}T^{\tilde{\gamma}-\mu},
\end{equation}
where we recall $\tilde{\gamma}> \mu$. In addition, (\ref{jacex3}) yields

\begin{equation}\label{jacex6}
\Big\|\sum_{i=1}^\infty q^{2,i}\Big\|_{0,0\rightarrow 0}\le  C_G\|V-U\|_{\mathcal{C}^{\mu,0\rightarrow 0}_1}T^{\mu+\tilde{\gamma}}\|g\|_{\mathcal{W}^{\tilde{\gamma},\delta,\infty}_{\lambda,T}} +  C_G\|V-U\|_{\mathcal{C}^{\mu,\kappa\rightarrow \kappa}_1}\|g\|_{\mathcal{W}^{\tilde{\gamma},\delta,\infty}_{\lambda,T}}T^{\tilde{\gamma}}.
\end{equation}

Summing up (\ref{jacex1}), (\ref{jacex2}), (\ref{jacex5}) and (\ref{jacex6}), we conclude

\begin{equation}\label{jacex7}
\|\Gamma(U)-\Gamma(V)\|_{\mathcal{C}^{\mu,0\rightarrow 0}_1}\le \Big[C_F(T^{1-\mu} + T) + (C_G\|g\|_{\mathcal{W}^{\tilde{\gamma},\delta,\infty}_{\lambda,T}})(2T^{\tilde{\gamma}}+T^{\tilde{\gamma}-\mu} + T^{\mu+\tilde{\gamma}})\Big]\|U-V\|_{\mathcal{C}^{\mu,0\rightarrow 0}_1}.
\end{equation}
By making $T$ small in (\ref{jacex7}), we conclude there exists a unique fixed point for $\Gamma$ on small interval $[0,\bar{T}]$ whose size does not depend on the initial condition. The construction of a global unique solution from the solution in $[0,\bar{T}]$ is standard and it is left to the reader for sake of conciseness. This pathwise argument clearly provides a unique adapted process $U$ realizing (\ref{UeqINT1}).
\end{proof}

Now, we set $R_t(y) = U_t(y) + \text{Id}$ and we observe that

\begin{eqnarray}
\nonumber R_t(y) &=& \text{Id} - \int_0^t R_s(y)S(-s)\nabla F(X^y_s)S(s)ds \\ \label{equest1} & &\\
\nonumber &-& \sum_{i=1}^\infty \sqrt{\lambda_i}\int_0^t R_s(y)S(-s)\nabla G_i(X^y_s)S(s)d\beta^i_s; 0\le t\le T.
\end{eqnarray}
We arrive at the following result which will play a key role in representing the Malliavin matrix.
\begin{proposition}
 If Assumptions H1-A1-A2-A3-B1-B2 hold, then for each initial condition $y\in E_k$, the Jacobian $\mathbf{J}_{0\rightarrow t}(y)$ admits a right-inverse adapted process $\mathbf{J}^+_{0\rightarrow t}(y)$ which satisfies

\begin{eqnarray}
\nonumber \mathbf{J}^+_{0\rightarrow t}(y) &=& S(-t) - \int_0^t\mathbf{J}^+_{0\rightarrow s}(y)\nabla F(X^y_s)S(s-t)ds\\
\label{Jplus} & &\\
\nonumber &-&\sum_{i=1}^\infty\sqrt{\lambda_i}\int_0^t \mathbf{J}^+_{0\rightarrow s}(y)\nabla G_i (X^y_s)S(s-t)d\beta^i_s; 0\le t\le T.
\end{eqnarray}
\end{proposition}
\begin{proof}
The candidate is $\mathbf{J}^+_{0\rightarrow t}(y):= R_t(y) S^-(t)$ defined on $S(t)E$. At first, we observe

$$S(s)S(-t)=S(s-t)~\text{on}~S(t)E\subset S(t-s)E,$$
for every $s < t$. Then, (\ref{Jplus}) is well-defined in view of Assumptions B1-B2. Let us check it is the right-inverse. Let
\begin{eqnarray}
\nonumber V_t(y) &=& \int_0^t S(-s)\nabla F (X^y_s)S(s) \big[\text{Id} + V_s(y)\big] ds\\
\label{UeqINT2}& &\\
\nonumber&+&\sum_{i=1}^\infty \sqrt{\lambda_i}\int_0^t S(-s)\nabla G_i(X^y_s)S(s) \big[\text{Id}+ V_s(y)\big] d\beta^i_s;0\le t\le T.
\end{eqnarray}
By following a similar proof of Lemma \ref{fixpointjac}, we can safely state there exists a unique adapted solution $V(y)$ of (\ref{UeqINT2}) such that $V(y)\in \mathcal{C}^{\mu,0\rightarrow 0}_1$ a.s for $\mu < \tilde{\gamma}$ and $\mu + \tilde{\gamma}>1$. Let us define $P_t(y) = V_t(y) + \text{Id}$ and notice that $S(t)S(-s) = S(t-s)$ on $S(s)E$ for every $t>s\ge 0$. Then,
\begin{eqnarray}
\nonumber P_t(y) &=& \text{Id} + \int_0^t S(-s)\nabla F(X^y_s)S(s)  P_s(y) ds \\ \label{equest2} & &\\
\nonumber &+& \sum_{i=1}^\infty \sqrt{\lambda_i}\int_0^t S(-s)\nabla G_i(X^y_s)S(s) P_s(y) d\beta^i_s~,
\end{eqnarray}
and therefore $\mathbf{J}_{0\rightarrow t}(y) = S(t)P_t(y)$. Equations (\ref{equest1}), (\ref{equest2}) and integration by parts in Hilbert spaces yield

\begin{eqnarray}
\nonumber \big\langle P_t(y) R_t(y) w, w^\prime \big\rangle_{E} = \big\langle R_t(y) w, P^*_t(y) w^\prime \big\rangle_{E} &=& \big\langle w, w^\prime \big\rangle_{E} + \int_0^t \big\langle dR_s(y) w, P^*_s(y) w^\prime \big\rangle_{E}\\
\nonumber &+& \int_0^t \big\langle R_s(y) w, dP^*_s(y) w^\prime \big\rangle_{E},
\end{eqnarray}
for each $w, w^\prime \in E$, where $P^*$ is the adjoint. To keep notation simple, we set $I_1 = \int_0^t \big\langle dR_s(y) w, P^*_s(y) w^\prime \big\rangle_{E}$ and $I_2 = \int_0^t \big\langle R_s(y) w, dP^*_s(y) w^\prime \big\rangle_{E}$. We observe

\begin{eqnarray*}
I_1 &=&-\int_0^t \big\langle P_s(y)R_s(y) S(-s)\nabla F(X^y_s)S(s) w, w^\prime \big\rangle _{E} ds\\
& &\\
&-&\sum_{i=1}^\infty \sqrt{\lambda_i} \int_0^t \big\langle P_s(y)R_s(y) S(-s)\nabla G_i(X^y_s)S(s) w, w^\prime \big\rangle_{E} d\beta^i_s.
\end{eqnarray*}
In addition, Assumption B1 allows us to represent

\begin{eqnarray*}
I_2&=& \int_0^t \big\langle S(-s)\nabla F(X^y_s)S(s)P_s(y) R_s(y) w, w^\prime \big\rangle_{E} ds\\
& &\\
&+&\sum_{i=1}^\infty \sqrt{\lambda_i} \int_0^t \big\langle S(-s)\nabla F(X^y_s)S(s)P_s(y) R_s(y) w, w^\prime \big\rangle_{E} d\beta^i_s.
\end{eqnarray*}

This shows that
\begin{eqnarray}
\nonumber P_t(y)R_t(y) &=& \text{Id} + \int_0^t S(-s)\nabla F(X^y_s)S(s)  (P_s(y)R_s(y) ds \\ \nonumber& &\\
\nonumber &+& \sum_{i=1}^\infty \sqrt{\lambda_i}\int_0^t S(-s)\nabla G_i(X^y_s)S(s) P_s(y)R_s(y) d\beta^i_s\\
\label{equestprod} & &\\
\nonumber &-& \int_0^t P_s(y)R_s(y)S(-s)\nabla F(X^y_s)S(s)ds \\ \nonumber& &\\ \nonumber &-& \sum_{i=1}^\infty \sqrt{\lambda_i}\int_0^t P_s(y)R_s(y)S(-s)\nabla G_i(X^y_s)S(s)d\beta^i_s.
\end{eqnarray}
We now observe there exists a unique solution of (\ref{equestprod}). To see this, let $Q_t(y) = P_t(y)R_t(y) - \text{Id}$ and from (\ref{equestprod}), we have
\begin{eqnarray}
\nonumber Q_t(y) &=& \int_0^t S(-s)\nabla F(X^y_s)S(s)  Q_s(y) ds + \sum_{i=1}^\infty \sqrt{\lambda_i}\int_0^t S(-s)\nabla G_i(X^y_s)S(s) Q_s(y) d\beta^i_s\\
\label{equestprod1}& &\\
\nonumber &-& \int_0^t Q_s(y)S(-s)\nabla F(X^y_s)S(s)ds - \sum_{i=1}^\infty \sqrt{\lambda_i}\int_0^t Q_s(y)S(-s)\nabla G_i(X^y_s)S(s)d\beta^i_s~.
\end{eqnarray}
The same argument of the proof of Lemma \ref{fixpointjac} yields the existence of a unique solution of equation (\ref{equestprod1}). This obviously implies that (\ref{equestprod}) admits only one solution. Since $\text{Id}$ solves (\ref{equestprod}), we do have $P_t(y)R_t(y) = \text{Id}$ for every $t \in [0,T]$ and we conclude
$\mathbf{J}_{0\rightarrow t}(y) \mathbf{J}^+_{0\rightarrow t}(y) = S(t)P_t(y)R_t(y) S^-(t) = \text{Id}~a.s.$
\end{proof}

\section{Existence of densities under H\"{o}rmander's bracket condition}\label{ExistenceSection}
In this section, we examine the existence of the densities for random variables of the form $\mathcal{T}(X^{x_0}_t)$ for a bounded linear operator $\mathcal{T}:E\rightarrow \mathbb{R}^d$ for a given $t\in (0,T]$. Throughout this section, we fix a set of parameters $\kappa,\kappa_0,\tilde{\gamma},\delta,\lambda$ as described in (\ref{Aindices}). In order to state a H\"{o}rmander's bracket condition, we need to work with smooth vector fields $F,G_i; i\ge 1$. Let

$$\text{dom}(A^n):=\{h\in E; h\in \text{dom}(A^{n-1})~\text{and}~A^{n-1}h\in \text{dom}(A)\},$$
$$\|h\|^2_{\text{dom}(A^n)}:=\sum_{i=0}^n \|A^ih\|^2_E,$$
$$\text{dom}(A^\infty):=\cap_{n=1}^\infty\text{dom}(A^n).$$

We observe $\text{dom}(A^\infty)$ is a Fr\'echet space equipped with the family of seminorms $\|\cdot \|_{\text{dom}(A^n)}; n\ge 0$. In the sequel, for each $t\in [0,T]$, we equip $S(t)E$ with the following inner product

\begin{equation}\label{Sinner}
\langle S(t)x, S(t)y \rangle_{S(t)E}:=\langle x, y \rangle_{E};~x,y\in E.
\end{equation}
Notice that this is a well-defined inner product due to the injectivity of the semigroup. One can easily check $S(t)E$ is a separable Hilbert space equipped with the norm associated with (\ref{Sinner}). Moreover, for each $x_0\in E_\kappa$ and $t\in [0,T]$, $\mathbf{J}^+_{0\rightarrow t}(x_0):S(t)E\rightarrow E$ admits an adjoint as a bounded linear operator from $E$ to $S(t)E$. Indeed, let $\mathbf{J}^{+,*}_{0\rightarrow t}(x_0):E\rightarrow S(t)E$ be the linear operator defined by

$$y\mapsto \mathbf{J}^{+,*}_{0\rightarrow t}(x_0)y:=S(t)R^*_t(x_0)y.$$
Then,

\begin{eqnarray}
\nonumber \big\langle \mathbf{J}^+_{0\rightarrow t}(x_0)S(t)x,y\big\rangle_{E} &=& \big\langle R_t(x_0) S(-t)S(t)x,y\big\rangle_{E}\\
\nonumber & &\\
\nonumber &=&\big\langle x, R^*_t(x_0)y  \big\rangle_{E} =  \big\langle S(t)x, \mathbf{J}^{+,*}_{0\rightarrow t}(x_0)y  \big\rangle_{S(t)E},
\end{eqnarray}
where $\|\mathbf{J}^{+,*}_{0\rightarrow t}(x_0)y\|_{S(t)E} = \|R^*_t(x_0)y\|_E \le \|R^*_t(x_0)\| \|y\|_E$. This proves our claim. We observe $R^*_t(x_0) = \text{Id} + U^*_t(x_0)$, where

\begin{eqnarray}
\nonumber U^*_t(x_0)&=&-\int_0^t \big( S(-r)\nabla F(X^{x_0}_r)S(r)\big)^* \big(\text{Id} + U^*_r(x_0)\big)dr\\
\nonumber & &\\
\nonumber&-& \sum_{i=1}^\infty\sqrt{\lambda_i}\int_0^t \big( S(-r)\nabla G_i(X^{x_0}_r)S(r)\big)^* \big(\text{Id} + U^*_r(x_0)\big)d\beta^i_r,
\end{eqnarray}
so that

\begin{eqnarray}
\nonumber R^*_t(x_0)&=& \text{Id}-\int_0^t \big( S(-r)\nabla F(X^{x_0}_r)S(r)\big)^* R^*_r(x_0)dr\\
\label{adjointR}& &\\
\nonumber&-& \sum_{i=1}^\infty\sqrt{\lambda_i}\int_0^t \big( S(-r)\nabla G_i(X^{x_0}_r)S(r)\big)^* R^*_r(x_0)d\beta^i_r.
\end{eqnarray}
In other words,

\begin{eqnarray}
\nonumber \mathbf{J}^{+,*}_{0\rightarrow t}(x_0)&=& S(t) -\int_0^t S(t)\big( S(-r)\nabla F(X^{x_0}_r)S(r)\big)^* R^*_r(x_0)dr\\
\nonumber& &\\
\nonumber&-& \sum_{i=1}^\infty\sqrt{\lambda_i}\int_0^t S(t)\big( S(-r)\nabla G_i(X^{x_0}_r)S(r)\big)^* R^*_r(x_0)d\beta^i_r.
\end{eqnarray}

\begin{definition}
A vector field $V$ on an open subset $U\subset M$ of a Fr\'echet space $M$ is a smooth map $V:U\rightarrow M$.
\end{definition}

Let us recall the concept of Lie brackets between two vector fields $V_1,V_2:\text{dom}(A^\infty)\rightarrow \text{dom}(A^\infty)$

\begin{equation}\label{liebracket}
[V_1,V_2](r):= \nabla V_2(r) V_1(r) - \nabla V_1(r) V_2(r),
\end{equation}
for each $r\in \text{dom}(A^\infty)$. We observe $[V_1,V_2]:\text{dom}(A^\infty)\rightarrow\text{dom}(A^\infty)$ is a well-defined vector field whenever $V_1,V_2$ are vector fields on $\text{dom}(A^\infty)$. Moreover, $\frac{1}{4} < \kappa < 1$ implies $\text{dom}(A)\subset \text{dom}(-A)^{\kappa} $, so that $\text{dom}(A^\infty)\subset  E_\kappa$.


\

\noindent \textbf{Assumption C1:} $G:E\rightarrow \mathcal{L}_2\big(U_0; S(T)E\big)$ satisfies:

\

\noindent (i) $x\mapsto G_i(x)$ is an $S(T)\text{dom}(A)$-valued continuous mapping for each $i\ge 1$. Moreover,

\

\noindent (ii) $$\mathbb{E}\int_0^T \|G(X^{x_0}_r)\|^2_{\mathcal{L}_2(U_0,S(T)E)}dr < \infty.$$

\

\noindent \textbf{Assumption C2:} $F,G_i:E\rightarrow \text{dom}(A^\infty)$ are smooth mappings with bounded derivatives for every $i\ge 1$ with the property that

$$\sup_{\ell\ge 1}\sup_{y\in E}\|\nabla^n G_\ell(y)\|_{(n),E\rightarrow \text{dom}(A^m)} < \infty,$$
for every $n,m\ge 1$. There exists a constant $C$ such that
$$\|G_\ell(y)\|_{\text{dom}(A)}\le C(1+\|y\|_{\text{dom}(A)}),~y\in \text{dom}(A),$$
for every $\ell \ge 1$. Moreover, $F,G_i;i\ge 1:\text{dom}(A^k)\rightarrow \text{dom}(A^k)$ are $C^\infty$-bounded for every $k,i\ge 1$.


\

\noindent \textbf{Assumption C3:} For every $n,p\ge 1$, $\nabla^{n}G_p(x)v\in S(T)\text{dom}(A)$ and $\nabla^n F(x)v\in S(T)\text{dom}(A)$ for every  $x\in \text{dom}(A)$ and $v\in \text{dom}^n(A)$.

\

Under Assumption C2, if we assume that $x_0\in \text{dom}(A^\infty$), then we can construct a solution process with $\alpha$-H\"{o}lder continuous trajectories in $\text{dom}(A^\infty)$. This is true because the Picard approximation procedure converges in every Hilbert space $\text{dom}(A^m)$, and the topology of $\text{dom}(A^\infty)$ is the projective limit of the ones on $\text{dom}(A^m)$. We summarize this fact into the following remark.

\begin{remark}\label{strongSOL}
Under Assumption C2, for each initial condition $x_0\in \text{dom}(A)$, (\ref{SPDE1}) has a unique strong solution. If $x_{0}\in \text{dom}(A^\infty)$, then we can construct a solution of (\ref{SPDE1}) taking values on $\text{dom}(A^\infty)$ and such that

$$\|\delta X^{x_0}\|_{\alpha,\text{dom}(A^m)} < \infty,$$
for every $m\ge 1$.
\end{remark}

\begin{remark}
Assumption C3 plays a rule in constructing the argument towards the existence of densities which requires

$$[G_0,V](X^{x_0}_t)\in S(t)E$$
in order to belong to the domain of $\mathbf{J}^+_{0\rightarrow t}(x_0)$ for every $V\in \mathcal{V}_m;m\ge 0$ (see (\ref{horc1})), where $G_0$ is the vector field given by (\ref{stratonovich}).
\end{remark}

The following elementary remark is useful.
\begin{lemma}\label{elem1}
If $V:E\rightarrow \text{dom}(A^\infty)$ is a smooth mapping with bounded derivatives, then

$$\sup_{y\in E}\|\nabla^n V(y)\|_{(n),0\rightarrow 0} < \infty.$$
\end{lemma}
\begin{proof}
The $n$-th Fr\'echet derivative of $V$ viewed as a map from $E$ to $\text{dom}(A)$ is given by $\nabla^n V:E \rightarrow \mathcal{L}_{n}\big(E^n; \text{dom}(A)\big)$, where

$$\|\nabla^n V(x)(h_1, \ldots, h_n)\|_{\text{dom}(A)}\le  \|\nabla ^ n V (x)\|_{(n),E\rightarrow \text{dom}(A)}\|h_1\|_E \times \ldots\times \|h_n\|_E.$$
Then,

\begin{eqnarray*}
\|\nabla^n V(x)(h_1, \ldots, h_n)\|_E&\le& \|\nabla^n V(x)(h_1, \ldots, h_n)\|_{\text{dom}(A)}\le \|\nabla ^ n V (x)\|_{(n),E\rightarrow \text{dom}(A)}\|h_1\|_E\times \ldots\times \|h_n\|_E\\
& &\\
&\le& \sup_{y\in E}\|\nabla ^ n V (y)\|_{(n),E\rightarrow \text{dom}(A)}\|h_1\|_E \times \ldots\times \|h_n\|_E,
\end{eqnarray*}
and hence $\|\nabla ^ n V (x)\|_{(n),0\rightarrow 0}\le \sup_{y\in E}\|\nabla ^ n V (y)\|_{(n),E\rightarrow \text{dom}(A)} < \infty$ for every $x\in E$.
\end{proof}

Let us now investigate the existence of densities for the SPDE (\ref{SPDE1}). We start with some preliminary results.

\begin{lemma}\label{LemmaMallderX}
Under Assumptions H1-A1-A2-A3-B1-B2-C1-C2, for each $x_0\in \text{dom}(A)$, we have

\begin{equation}\label{MallderX}
\mathbf{D}_rX^{x_0}_t = \mathbf{J}_{0\rightarrow t}(x_0)\mathbf{J}_{0\rightarrow r}^{+}(x_0)G(X^{x_0}_r)~a.s,
\end{equation}
for every $r < t$. Therefore,

\begin{equation}\label{chain}
\mathbf{D}_r \mathcal{T}(X^{x_0}_t)  = \mathcal{T}\big(\mathbf{J}_{0\rightarrow t}(x_0)\mathbf{J}_{0\rightarrow r}^{+}(x_0)G(X^{x_0}_r)\big)~a.s,
\end{equation}
for every $r< t$.
\end{lemma}
\begin{proof}
On one hand, Remark \ref{strongSOL} and (\ref{Malliavindereq}) yields

\begin{equation}\label{Malliavindereq1}
\mathbf{D}_rX^{x_0}_t = G(X^{x_0}_r) + \int_r^t \nabla F(X^{x_0}_\ell)\mathbf{D}_rX^{x_0}_\ell d\ell+ \sum_{i=1}^\infty \int_r^t \nabla G_i(X^{x_0}_\ell)\mathbf{D}_rX^{x_0}_\ell d\beta^i_\ell,
\end{equation}
for $0\le r< t$. On the other hand, Assumption C2 implies that (\ref{jacobianSPDE}) has a strong solution for $y=x_0\in \text{dom}(A)$ and for each $v=G_j(X^{x_0}_r)$. Having said that, let us fix $0\le r < t$ and a positive integer $j\ge 1$. The fact that $G_j(E) \subset S(T)E$ and Remark \ref{S(T)} yield

$$ G_j(X^{x_0}_r) + \int_r^t \nabla F(X^{x_0}_\ell)\mathbf{J}_{0\rightarrow \ell}(x_0)\mathbf{J}_{0\rightarrow r}^{+}(x_0)G_j(X^{x_0}_r) d\ell+ \sum_{i=1}^\infty \int_r^t \nabla G_i(X^{x_0}_\ell)\mathbf{J}_{0\rightarrow \ell}(x_0)\mathbf{J}_{0\rightarrow r}^{+}(x_0)G_j(X^{x_0}_r)d\beta^i_\ell$$
$$ = G_j(X^{x_0}_r) + \Bigg(\int_r^t \nabla F(X^{x_0}_\ell) \mathbf{J}_{0\rightarrow \ell}(x_0) d\ell+ \sum_{i=1}^\infty \int_r^t \nabla G_i(X^{x_0}_\ell)\mathbf{J}_{0\rightarrow \ell}(x_0)d\beta^i_\ell\Bigg)\mathbf{J}_{0\rightarrow r}^{+}(x_0)G_j(X^{x_0}_r)$$
$$ = G_j(X^{x_0}_r) + \Big(\mathbf{J}_{0\rightarrow t}(x_0) - \mathbf{J}_{0\rightarrow r}(x_0)\Big)\mathbf{J}_{0\rightarrow r}^{+}(x_0)G_j(X^{x_0}_r) = \mathbf{J}_{0\rightarrow t}(x_0)\mathbf{J}^{+}_{0\rightarrow r}(x_0)G_j(X^{x_0}_r)~a.s.$$
By invoking (\ref{pathwiseMalliavinder}), (\ref{continuousPSI}), Lemma \ref{menosK}, (\ref{Malliavindereq1}) and Assumption C1(i), we know that both $(r,t)\mapsto \mathbf{D}_r X^{x_0}_t$ and $(r,t)\mapsto \mathbf{J}_{0\rightarrow t}(x_0)\mathbf{J}^{+}_{0\rightarrow r}(x_0)G_j(X^{x_0}_r) $ are jointly continuous a.s on the simplex $\{(r,t); 0\le r\le t \le T\}$. This fact combined with the uniqueness of the SPDE solution of (\ref{Malliavindereq1}) (for each fixed $r$) implies that they are indistinguishable

$$\big(\mathbf{D}_\cdot X^{x_0}_\cdot\big)(\sqrt{\lambda_j}e_j) =\mathbf{J}_{0\rightarrow \cdot}(x_0)\mathbf{J}^{+}_{0\rightarrow \cdot}(x_0)G_j(X^{x_0}_\cdot)~a.s,$$
for each $j\ge 1$. Assumption C1 (ii) implies

$$r\mapsto \mathbf{J}_{0\rightarrow t}(x_0) \mathbf{J}^+_{0\rightarrow r}(x_0)G(X^{x_0}_r)\in \mathcal{L}_2\big(U_0;L^{\frac{1}{H}}([0,T];E)\big)\subset \mathcal{L}_2(U_0;\mathcal{H}\otimes E)~a.s,$$
for every $t\in (0,T]$. Summing up the above arguments, we shall conclude (\ref{MallderX}) holds true. The chain rule yields representation (\ref{chain}).
\end{proof}


In what follows, let us denote

\begin{equation}\label{gammamatrix}
\gamma_t:= \Big(\big\langle \mathbf{D}\mathcal{T}_i(X_t^{x_0}),  \mathbf{D}\mathcal{T}_j(X_t^{x_0})\big\rangle_{\mathcal{L}_2(U_0;\mathcal{H})}\Big)_{1\le i,j\le d},
\end{equation}
where $\mathcal{T} = (\mathcal{T}_1, \ldots \mathcal{T}_d):E \rightarrow \mathbb{R}^d$. In order to investigate non-degeneracy of the Malliavin derivative, it is convenient to work with a \textit{reduced} Malliavin operator. Let us define the self-adjoint linear operator $\mathcal{C}_t:E\rightarrow E$ by the following quadratic form

\begin{eqnarray}
\nonumber\langle \mathcal{C}_t y, y \rangle_{E}&:=&\alpha_H \sum_{\ell=1}^\infty \int_0^t \int_0^t\big\langle \mathbf{J}^+_{0\rightarrow u}(x_0) G_\ell(X^{x_0}_u),y\big\rangle_{E}\big\langle \mathbf{J}^+_{0\rightarrow v}(x_0) G_\ell(X^{x_0}_v),y\big\rangle_{E}|u-v|^{2H-2}dudv\\
\label{Ctoperator}& &\\
\nonumber&=&\sum_{\ell=1}^\infty \Big\| \big\langle \mathbf{J}^+_{0\rightarrow \cdot}(x_0) G_\ell(X^{x_0}_\cdot),y\big\rangle_{E}  \Big\|^2_{\mathcal{H}} = \sum_{\ell=1}^\infty \Big\| \big\langle G_\ell(X^{x_0}_\cdot),\mathbf{J}^{+,*}_{0\rightarrow \cdot}(x_0)y\big\rangle_{S(\cdot)E}  \Big\|^2_{\mathcal{H}},
\end{eqnarray}
for $y\in E$ and $0 < t\le T$. In (\ref{Ctoperator}), the norm in $\mathcal{H}$ is computed over $[0,t]$. We observe $\mathcal{C}_t$ is a well-defined bounded linear operator due to Assumption C1 (ii) and $\frac{1}{H}< 2$.

By applying Lemma \ref{LemmaMallderX} and (\ref{niceinner}), we arrive at the following representation.

\begin{lemma}\label{Malliavinmatrix}
Under Assumptions H1-A2-A2-A3-B1-B2-C1-C2, for each $x_0\in \text{dom}(A)$, we have

$$\gamma_t = \big(\mathcal{T}\circ\mathbf{J}_{0\rightarrow t}(x_0)\big)~ \mathcal{C}_t~\big(\mathcal{T}\circ\mathbf{J}_{0\rightarrow t}(x_0)\big)^*. $$
\end{lemma}



Let us define

\begin{equation}\label{stratonovich}
G_0(x):= A(x) + F(x); x\in \text{dom}(A^\infty).
\end{equation}
Given the SPDE (\ref{SPDE1}), let $\mathcal{V}_k$ be a collection of vector fields given by

\begin{equation}\label{horc1}
\mathcal{V}_0 = \{G_i; i\ge 1\},\quad \mathcal{V}_{k+1}:=\mathcal{V}_k \cup \big\{ [G_j, U]; U\in \mathcal{V}_k~\text{and}~j\ge 0 \big\}.
\end{equation}
We also define the vector spaces $\mathcal{V}_k(x_0):=\text{span}\{V(x_0); V\in \mathcal{V}_k\}$ and we set

$$\mathcal{D}(x_0):=\cup_{k\ge 1}\mathcal{V}_k(x_0),$$
for each $x_0\in \text{dom}(A^\infty)$.

\

Note that under Assumption C2, all the Lie brackets in (\ref{horc1}) are well-defined as vector fields $\text{dom}(A^\infty)\rightarrow \text{dom}(A^\infty)$.

\begin{proposition}\label{Itobrackets1}
If Assumptions H1-A1-A2-A3-B1-B2-C1-C2-C3 hold true, then for each $x_0\in \text{dom}(A^\infty)$, we have

\begin{eqnarray}
\nonumber \mathbf{J}^+_{0\rightarrow t}(x_0)V(X^{x_0}_t)&=& V(x_0) + \int_0^t \mathbf{J}_{0\rightarrow s}^+(x_0)[G_0,V](X^{x_0}_s)ds\\
\label{iter} & &\\
\nonumber &+& \sum_{\ell=1}^\infty \sqrt{\lambda_\ell}\int_0^t \mathbf{J}^+_{0\rightarrow s}(x_0)[G_\ell,V](X^{x_0}_s)d\beta^\ell_s;~0\le t\le T,
\end{eqnarray}
where $V\in \mathcal{V}_n$ for $n=0,1,2, \ldots~.$
\end{proposition}
\begin{proof}
At first, we take $V\in \mathcal{V}_0$. Assumptions C2-C3 yield $V(X^{x_0}_\cdot)\in S(T)E_\kappa$, $[G_0,V](X^{x_0}_\cdot)\in S(T)E$, and $[G_\ell,V](X^{x_0}_\cdot)\in S(T)E$ a.s. Moreover, change of variables for Young integrals yields
\begin{equation}\label{changeV1}
V(X^{x_0}_t) = V(x_0)+\int_0^t \nabla V(X^{x_0}_s)G_0(X^{x_0}_s)ds + \sum_{\ell=1}^\infty \sqrt{\lambda_\ell}\int_0^t \nabla V(X^{x_0}_s) G_\ell(X^{x_0}_s)d\beta^\ell_s,
\end{equation}
where $G_0(X^{x_0}_s) = A(X^{x_0}_s) + F(X^{x_0}_s); 0\le s\le T$. We observe Young-Loeve's inequality and A1-A2-A3 allow us to state the Young integral in (\ref{changeV1}) is well-defined. Recall the Lie bracket $[G_0,V](X^{x_0}_s) = \nabla V(X^{x_0}_s) G_0(X^{x_0}_s) - \nabla G_0(X^{x_0}_s) V(X^{x_0}_s)$, so that we can actually rewrite

$$V(X^{x_0}_t) = V(x_0)+\int_0^t \Big(\nabla G_0(X^{x_0}_s)V(X^{x_0}_s) + [G_0,V](X^{x_0}_s)\Big)ds + \sum_{\ell=1}^\infty \sqrt{\lambda_\ell}\int_0^t \nabla V(X^{x_0}_s) G_\ell(X^{x_0}_s)d\beta^\ell_s,$$
where $\nabla G_0(X^{x_0}_s)V(X^{x_0}_s) = A(V(X^{x_0}_s)) + \nabla F(X^{x_0}_s)V(X^{x_0}_s); 0\le s\le T$. This implies that $V(X^{x_0})$ can be written as the mild solution of

\begin{eqnarray}
\nonumber V(X^{x_0}_t)&=& S(t)V(x_0) + \int_0^t S(t-s)\big(\nabla F(X^{x_0}_s)V(X^{x_0}_s) + [G_0,V](X^{x_0}_s)\big)ds\\
\nonumber & &\\
\nonumber &+&\sum_{\ell=1}^\infty \sqrt{\lambda_\ell}\int_0^t S(t-s)\nabla V(X^{x_0}_s)G_\ell (X^{x_0}_s)d\beta^\ell_s,
\end{eqnarray}
so that

\begin{eqnarray}
\nonumber S(-t)V(X^{x_0}_t)&=& V(x_0) + \int_0^t S(-s)\big(\nabla F(X^{x_0}_s)V(X^{x_0}_s) + [G_0,V](X^{x_0}_s)\big)ds\\
\label{hiddenV}& &\\
\nonumber &+&\sum_{\ell=1}^\infty \sqrt{\lambda_\ell}\int_0^t S(-s)\nabla V(X^{x_0}_s)G_\ell (X^{x_0}_s)d\beta^\ell_s;0\le t\le T.
\end{eqnarray}

The adjoint operator $\mathbf{J}^{+,*}_{0\rightarrow t}(x_0)$ yields
$$\langle \mathbf{J}^+_{0\rightarrow t}(x_0) V(X^{x_0}_t),y\rangle_{E} = \langle V(X^{x_0}_t),\mathbf{J}^{+,*}_{0\rightarrow t}(x_0)y\rangle_{S(t)E} = \langle S(-t)V(X^{x_0}_t),R^*_t(x_0)y\rangle_{E} $$
for a given $y\in E$. Hence, integration by parts yields

\begin{eqnarray}
\nonumber \langle \mathbf{J}^+_{0\rightarrow t}(x_0) V(X^{x_0}_t),y\rangle_{E}  &=& \langle  V(x_0),y\rangle_{E} +\int_0^t \langle dS(-s)V(X^{x_0}_s),R^*_s(x_0)y \rangle_{E}\\
\nonumber & &\\
\nonumber &+& \int_0^t \langle S(-s)V(X^{x_0}_s),dR^*_s(x_0)y \rangle_{E};0\le t\le T.
\end{eqnarray}
By combining (\ref{hiddenV}) and (\ref{adjointR}), we conclude that (\ref{iter}) holds true for $V\in \mathcal{V}_0$. Now, we take $V =[ G_i,G_p ]$ or $V  = [G_0,G_p]$ for $p,i=1, 2, \ldots~.$ In this case, C2-C3 yield $V(X^{x_0}_\cdot)\in S(T)E$, $[G_0,V](X^{x_0}_\cdot)\in S(T)E$, and $[G_\ell,V](X^{x_0}_\cdot)\in S(T)E$. From the above argument for vector fields in $\mathcal{V}_0$, we learn that in order to prove (\ref{iter}), it is sufficient to ensure that the Young integral in the right-hand side of (\ref{changeV1}) is well-defined, i.e.,

\begin{equation}\label{deltaV}
\sup_{\ell\ge 1}\|\delta \nabla V(X^{x_0}_\cdot)G_\ell(X^{x_0}_\cdot)\|_{\alpha,0}<\infty~a.s.
\end{equation}
At first, we observe if $W:\text{dom}(A^\infty)\rightarrow \text{dom}(A^\infty)$ is smooth, then

\begin{eqnarray}
\nonumber\nabla [G_0, W](x)(h) &=& \nabla^2 W(x)(h,Ax) + \nabla W(x)A(h) + \nabla^2 W(x)(h,F(x)) + \nabla W(x) \nabla F(x)h\\
\label{lieG0}& &\\
\nonumber&-&A\nabla W(x)h - \nabla^2 F(x)(h,W(x)) - \nabla F(x) \nabla W(x) h; h\in \text{dom}(A^\infty),
\end{eqnarray}

\begin{equation}\label{lieGp}
\nabla [G_p, W](x)(h) = \nabla^2 W(x)(h,G_p(x)) + \nabla W(x)\nabla G_p(x)(h) - \nabla^2 G_p(x)(h,W(x)) - \nabla G_p(x) \nabla W(x)(h),
\end{equation}
for $h\in \text{dom}(A^\infty)$ and $p\ge 1$. If $V = [G_0,G_p]$, we observe

$$\nabla V(X^{x_0}_t)G_\ell(X^{x_0}_t) = -A \nabla G_p(X^{x_0}_t)G_\ell(X^{x_0}_t)- \nabla^2 F(X^{x_0}_t)(G_p(X^{x_0}_t),G_\ell(X^{x_0}_t))$$
 $$- \nabla F(X^{x_0}_t)\nabla G_p(X^{x_0}_t)G_\ell(X^{x_0}_t) + \nabla^2G_p(X^{x_0}_t)(AX^{x_0}_t, G_\ell(X^{x_0}_t)) + \nabla G_p(X^{x_0}_t)AG_\ell(X^{x_0}_t)$$
$$ + \nabla^2 G_p(X^{x_0}_t)(F(X^{x_0}_t), G_\ell(X^{x_0}_t)) + \nabla G_p(X^{x_0}_t)\nabla F(X^{x_0}_t)G_\ell(X^{x_0}_t)$$
$$=:\sum_{i=1}^7 I_{i,p,\ell}(t).$$
Since $F,G_i:E \rightarrow \text{dom}(A^\infty)$ has bounded derivatives of all orders (by Assumption C2), we shall use Lemma \ref{elem1} to get

\begin{eqnarray*}
\|I_{1,p,\ell}(t) - I_{1,p,\ell}(s)\|_E&\le& \|\nabla G_p(X^{x_0}_t)G_\ell(X^{x_0}_t) - \nabla G_p(X^{x_0}_t)G_\ell(X^{x_0}_s)\|_{\text{dom}(A)}\\
& &\\
&+& \|A \nabla G_p(X^{x_0}_t)G_\ell(X^{x_0}_s) - A \nabla G_p(X^{x_0}_s)G_\ell(X^{x_0}_s)\|_E\\
& &\\
&\le& \sup_{y\in E}\|\nabla G_p(y)\|_{E\rightarrow \text{dom}(A)} \|G_\ell(X^{x_0}_t)  - G_\ell(X^{x_0}_s)\|_E\\
& &\\
&+& \|\nabla G_p(X^{x_0}_t) - \nabla G_p(X^{x_0}_s)\|_{E\rightarrow \text{dom}(A)} \|G_\ell(X^{x_0}_s)\|_E\\
& &\\
&\le& C \sup_{y\in E}\|\nabla G_p(y)\|_{E\rightarrow \text{dom}(A)} \|\delta X^{x_0}_{ts}\|_E\\
& &\\
&+&C \sup_{y\in E}\|\nabla^2 G_p(y)\|_{E\rightarrow \text{dom}(A)} \|\delta X^{x_0}_{ts}\|_E (1+\|X^{x_0}\|_{0,0}),
\end{eqnarray*}

\

$$
\|I_{2,p,\ell}(t)-I_{2,p,\ell}(s)\|_E\le \|\nabla^2 F(X^{x_0}_t)(G_p(X^{x_0}_t),G_\ell(X^{x_0}_t)) - \nabla^2 F(X^{x_0}_s)(G_p(X^{x_0}_t),G_\ell(X^{x_0}_t))\|_E
$$
$$+\|\nabla^2 F(X^{x_0}_s)(G_p(X^{x_0}_t),G_\ell(X^{x_0}_t)) - \nabla^2 F(X^{x_0}_s)(G_p(X^{x_0}_s),G_\ell(X^{x_0}_t))\|_E$$
$$+\|\nabla^2 F(X^{x_0}_s)(G_p(X^{x_0}_s),G_\ell(X^{x_0}_t)) -\nabla^2 F(X^{x_0}_s)(G_p(X^{x_0}_s),G_\ell(X^{x_0}_s))\|_E $$
$$\le \|\nabla^2 F(X^{x_0}_t) - \nabla^2 F(X^{x_0}_s)\|_{(2),0\rightarrow 0} \|G_p(X^{x_0}_t)\|_E \|G_\ell(X^{x_0}_t)\|_E$$
$$+\|\nabla^2 F(X^{x_0}_s)\|_{(2),0\rightarrow 0}\|G_p(X^{x_0}_t) - G_p(X^{x_0}_s)\|_E\|G_\ell(X^{x_0}_t)\|_E$$
$$+\|\nabla^2 F(X^{x_0}_s)\|_{(2),0\rightarrow 0} \|G_\ell(X^{x_0}_t) - G_\ell(X^{x_0}_s)\|_E\|G_p(X^{x_0}_s)\|_E$$
$$\le C \delta X^{x_0}_{ts}\|_E(1+ \|X^{x_0}\|_{0,0})^2$$
$$+2C\sup_{y\in E}\|\nabla^2 F(y)\|_{(2),0\rightarrow 0}(1+ \|X^{x_0}\|_{0,0}) \|\delta X^{x_0}_{ts}\|_E,$$

\

$$
\|I_{3,p,\ell}(t)-I_{3,p,\ell}(s)\|\le\|\nabla F(X^{x_0}_t)\nabla G_p(X^{x_0}_t)G_\ell(X^{x_0}_t) - \nabla F(X^{x_0}_s)\nabla G_p(X^{x_0}_t)G_\ell(X^{x_0}_t)\|_E$$
$$+\|\nabla F(X^{x_0}_s)\nabla G_p(X^{x_0}_t)G_\ell(X^{x_0}_t) - \nabla F(X^{x_0}_s)\nabla G_p(X^{x_0}_s)G_\ell(X^{x_0}_t)\|_E$$
$$+\|\nabla F(X^{x_0}_s)\nabla G_p(X^{x_0}_s)G_\ell(X^{x_0}_t) - \nabla F(X^{x_0}_s)\nabla G_p(X^{x_0}_s)G_\ell(X^{x_0}_s)\|_E$$
$$\le C \sup_{y\in E}|\nabla^2 F(y)|_{(2),0\rightarrow 0}\|\delta X^{x_0}_{ts} \|_E (1+ \| X^{x_0}\|_{0,0})$$
$$+ C\sup_{y\in E}\|\nabla F(y)\| \sup_{y\in E}\| \nabla^2G_p(y)\|_{(2),0\rightarrow 0}\|\delta X^{x_0}_{ts}\|_E(1+\|X^{x_0}\|_{0,0}),$$

\

$$\|I_{4,p,\ell}(t)-I_{4,p,\ell}(s)\|_E\le \|\nabla^2G_p(X^{x_0}_t)(AX^{x_0}_t, G_\ell(X^{x_0}_t)) - \nabla^2G_p(X^{x_0}_s)(AX^{x_0}_t, G_\ell(X^{x_0}_t))\|_E$$
$$+\|\nabla^2G_p(X^{x_0}_s)(AX^{x_0}_t, G_\ell(X^{x_0}_t)) -\nabla^2G_p(X^{x_0}_s)(AX^{x_0}_s, G_\ell(X^{x_0}_t))  \|_E$$
$$+\|\nabla^2G_p(X^{x_0}_s)(AX^{x_0}_s, G_\ell(X^{x_0}_t))  - \nabla^2G_p(X^{x_0}_s)(AX^{x_0}_s, G_\ell(X^{x_0}_s)) \|_E$$
$$\le\sup_{y\in E} \|\nabla^3 G_p(y)\|_{(3),0\rightarrow 0}\|\delta X^{x_0}_{ts}\|_E\|X^{x_0}\|_{0,\text{dom}(A)}(1+\|X^{x_0}\|_{0,0}) $$
$$+\sup_{y\in E}\|\nabla^2 G_p(y)\|_{(2),0\rightarrow 0}\|\delta X^{x_0}_{ts}\|_{\text{dom}(A)}(1+ \|X^{x_0}\|_{0,0})$$
$$+ \sup_{y\in E}\|\nabla^2 G_p(y)\|_{(2),0\rightarrow 0} \|\delta X^{x_0}_{ts}\|_E(1+ \|X^{x_0}\|_{0,\text{dom}(A)}),$$

\

$$\|I_{5,p,\ell}(t) -I_{5,p,\ell}(t)\|_E \le \|\nabla G_p(X^{x_0}_t)AG_\ell(X^{x_0}_t) -\nabla G_p(X^{x_0}_s)AG_\ell(X^{x_0}_t)\|_E$$
$$+\|\nabla G_p(X^{x_0}_s)AG_\ell(X^{x_0}_t) - \nabla G_p(X^{x_0}_s)AG_\ell(X^{x_0}_s)\|_E$$
$$\le C\sup_{y\in E}\|\nabla^2 G_p(y)\|_{(2),0\rightarrow 0} \|\delta X^{x_0}_{ts}\|_E\|G_\ell(X^{x_0}_t)\|_{\text{dom}(A)}$$
$$+C\sup_{y\in E}\|\nabla G_p(y)\|\sup_{y\in E}\|\nabla G_\ell(y)\|_{E\rightarrow\text{dom}(A)} \|\delta X^{x_0}_{ts}\|_E,$$
$$\le C\sup_{y\in E}\|\nabla^2 G_p(y)\|_{(2),0\rightarrow 0} \|\delta X^{x_0}_{ts}\|_E(1+ \|X^{x_0}\|_{0,\text{dom}(A)})$$
$$+C\sup_{y\in E}\|\nabla G_p(y)\|\sup_{y\in E}\|\nabla G_\ell(y)\|_{E\rightarrow\text{dom}(A)}\|\delta X^{x_0}_{ts}\|_E, $$

\

$$\|I_{6,p,\ell}(t) - I_{6,p,\ell}(s)\|_E \le \| \nabla^2 G_p(X^{x_0}_t)(F(X^{x_0}_t), G_\ell(X^{x_0}_t)) -\nabla^2 G_p(X^{x_0}_s)(F(X^{x_0}_t), G_\ell(X^{x_0}_t))\|_E$$
$$+\|\nabla^2 G_p(X^{x_0}_s)(F(X^{x_0}_t), G_\ell(X^{x_0}_t)) - \nabla^2 G_p(X^{x_0}_s)(F(X^{x_0}_s), G_\ell(X^{x_0}_s))\|_E$$
$$\le C\sup_{y\in E}\| \nabla^3 G_p(y)\|_{(3),0\rightarrow 0} \|\delta X^{x_0}_{ts}\|_E(1+\|X^{x_0}\|_{0,0})^2$$
$$+ C\sup_{y\in E}|\nabla^2 G_p(y)\|_{(2),0\rightarrow 0} \|\delta X^{x_0}_{ts}\|^2_E, $$

$$\|I_{7,p,\ell}(t) - I_{7,p,\ell}(s)\|_E\le \|\nabla G_p(X^{x_0}_t)\nabla F(X^{x_0}_t)G_\ell(X^{x_0}_t) - \nabla G_p(X^{x_0}_s)\nabla F(X^{x_0}_t)G_\ell(X^{x_0}_t)\|_E$$
$$+\|\nabla G_p(X^{x_0}_s)\nabla F(X^{x_0}_t)G_\ell(X^{x_0}_t) - \nabla G_p(X^{x_0}_s)\nabla F(X^{x_0}_s)G_\ell(X^{x_0}_s)\|_E$$
$$\le C\sup_{y\in E}\| \nabla^2 G_p(y)\|_{(2),0\rightarrow 0} \|\delta X^{x_0}_{ts}\|_E\sup_{y\in E}\|\nabla F(y)\| (1+\|X^{x_0}\|_{0,0})$$
$$+C\sup_{y\in E}\|\nabla^2 F(y)\|_{(2),0\rightarrow 0}\|\delta X^{x_0}_{ts} \|_E(1+\|X^{x_0}\|_{0,0})$$
$$+C\sup_{y\in E}\|\nabla F(y)\| \|\delta X^{x_0}_{ts}\|_E.$$

This shows that (\ref{deltaV}) holds true for vector fields of the $[G_0,G_p]; p=1, 2, \ldots~.$ A similar computation also shows (\ref{deltaV}) for vector fields of the form $[G_j,G_p]; j,p=1, 2, \ldots~.$ This shows that (\ref{iter}) holds for vectors fields $V\in \mathcal{V}_1$. By using (\ref{lieG0}) and (\ref{lieGp}) and iterating the argument, we recover (\ref{deltaV}) for vector fields $V\in \mathcal{V}_n; n\ge 0$ and hence we conclude the proof.
\end{proof}
\subsection{Doob-Meyer-type decomposition}
Let us now turn our attention to a Doob-Meyer decomposition in the framework of integral equations involving a trace-class FBM. This will play a key step in the proof of the existence of density of Theorem \ref{ExistenceResultTHESIS}. We recall the parameters $\tilde{\gamma},\delta,\lambda$ are fixed according to (\ref{Aindices}). In a rather general situation, Friz and Schekar \cite{friz} have developed the concept of \textit{true roughness} which plays a key role in determining the uniqueness of the Gubinelli's derivative in rough path theory. For sake of completeness, we recall the following concepts borrowed from \cite{gubinelli} and adapted to our infinite-dimensional setting. For a given $g\in\mathcal{W}^{\tilde{\gamma},\delta,\infty}_{\lambda,T}$, we write

$$\mathbf{G}_t=\sum_{j=1}^\infty \sqrt{\lambda_j}e_j g^j_t; 0\le t\le T.$$
Of course, $\mathbf{G}\in \mathcal{C}_1^{\tilde{\gamma}}(U)$ for every $g\in\mathcal{W}^{\tilde{\gamma},\delta,\infty}_{\lambda,T}$.
\begin{definition}
Given a path $\mathbf{G}\in \mathcal{C}^{\tilde{\gamma}}_1(U)$, we say that $Y\in \mathcal{C}^{\tilde{\gamma}}_1(\mathbb{R})$ is controlled by $\mathbf{G}$ if there exists $Y'\in \mathcal{C}^{\tilde{\gamma}}_1(U^*)$ so that the remainder term given implicitly through the relation

$$\delta Y_{ts} = Y'_s\delta \mathbf{G}_{ts} + R^Y_{ts}, s < t,$$
satisfies $\|R^Y\|_{2\tilde{\gamma},\mathbb{R}} < \infty$.
 \end{definition}


In our context, we restrict the analysis to the following class of derivatives. Let $C^{\beta,\infty}_1$ be the set of all sequences of real-valued functions on $[0,T]$, $(f_i)_{i=1}^\infty$ such that $\sup_{i\ge 1}\| \delta f_i\|_{\beta} < \infty$ for $0 < \beta \le 1$. Let $Y':[0,T]\rightarrow U^*$ be a $U^*$-valued path such that $(Y'^i)_{i=1}^\infty\in C^{\tilde{\gamma},\infty}_1 $ where $Y'^i=Y'(e_i); i\ge 1$ . We then observe if

\begin{equation}\label{controlledprocesses}
\delta Y_{ts} = Y'_s\delta \mathbf{G}_{ts} + R^Y_{ts},
\end{equation}
then, $\delta Y_{ts}=\int_s^t Y'_rd\mathbf{G}_r=\sum_{i=1}^\infty\sqrt{\lambda_i}\int_s^t Y'^i_rd g^i_r$ is a well defined Young integral, where the remainder is characterized by

$$
R^{Y}_{ts} = \sum_{j=1}^\infty \sqrt{\lambda_j}\int_s^t \big(Y'^{j}_r - Y'^{j}_s\big) d g^j_r,
$$
and $\|R^Z\|_{2\tilde{\gamma},\mathbb{R}} < \infty$ due to Young-Loeve inequality. The class of all pairs $(Y,Y')$ of the form (\ref{controlledprocesses}) constitutes a subset of controlled paths which we denote it by $\mathscr{D}^{2\tilde{\gamma}}_{\mathbf{G}}([0,T];U^*)$. Next, we recall the following concept of truly rough (see \cite{friz,frizhairer}).

\begin{definition}
For a fixed $s\in (0,T]$, we call a $\frac{1}{\tilde{\gamma}}$-rough path $\mathbf{G}:[0,T]\rightarrow U$, ''\textbf{rough at time}~s'' if

$$\forall v^*\in U^*~\text{non-null}~:~\limsup_{t\downarrow s}\frac{|\langle v^*,\delta \mathbf{G}_{ts}\rangle |}{|t-s|^{2\tilde{\gamma}}}=+\infty.$$
If $\mathbf{G}$ is rough on some dense subset of $[0,T]$, then we call it \textbf{truly rough}.
\end{definition}
\begin{lemma}
The $U$-valued trace-class FBM given by (\ref{QFBM}) is truly rough.
\end{lemma}
\begin{proof}
The proof follows the same lines of Example 2 in \cite{friz} together with the law of iterated logarithm for Gaussian processes as described by Th 7.2.15 in \cite{marcus}. We left the details to the reader.
\end{proof}
The following result is given by Th. 6.5 in Friz and Hairer \cite{frizhairer}.
\begin{theorem}\label{doobmeyer}
Assume that $\mathbf{G}$ is a truly rough path. Let $(Y,Y')$ and $(\tilde{Y},\tilde{Y}')$ be controlled paths in $\mathscr{D}^{2\tilde{\gamma}}_{\mathbf{G}}([0,T];U^*)$ and let $N,\tilde{N}$ be a pair of real-valued continuous paths. Assume that

$$\int_0^\cdot Y d\mathbf{G} + \int_0^\cdot Ndt =\int_0^\cdot \tilde{Y} d\mathbf{G} + \int_0^\cdot \tilde{N}dt$$
on $[0,T]$. Then, $\big(Y,Y'\big) = \big(\tilde{Y},\tilde{Y}'\big)$ and $N=\tilde{N}$ on $[0,T]$.
\end{theorem}

\subsection{Proof of Theorem \ref{ExistenceResultTHESIS}}
We are now in position to prove the main result of this paper.
\begin{proof}
Fix $x_0\in \text{dom}(A^\infty)$ and $t\in (0,T]$. By Lemma \ref{Malliavinmatrix}, we have
$$\gamma_t = \big(\mathcal{T}\circ\mathbf{J}_{0\rightarrow t}(x_0)\big)~ \mathcal{C}_t~\big(\mathcal{T}\circ\mathbf{J}_{0\rightarrow t}(x_0)\big)^*, $$
so that it is sufficient to prove that $\gamma_t$ is positive definite a.s. For  this purpose, we start by noticing that

$$\langle \gamma_t x, x\rangle_{\mathbb{R}^d} = \Big\langle \mathcal{C}_t\big(\mathcal{T}\circ\mathbf{J}_{0\rightarrow t}(x_0)\big)^* x, \big(\mathcal{T}\circ\mathbf{J}_{0\rightarrow t}(x_0)\big)^* x \Big\rangle_{E}; x\in \mathbb{R}^d.$$
We observe that $\big(\mathcal{T}\circ\mathbf{J}_{0\rightarrow t}(x_0)\big)^*$ is one-to-one. By assumption, $\text{Ker}\mathcal{T}^*=\{\mathbf{0}\}$ and clearly $\text{Ker} \mathbf{J}^{*}_{0\rightarrow t}(x_0)=\{\mathbf{0}\}$. Indeed, if $y\in \text{ker}\mathbf{J}^{*}_{0\rightarrow t}(x_0)$, then for every $x\in E$

\begin{eqnarray*}
\langle y, S(t)x\rangle_E &=& \langle y, \mathbf{J}_{0\rightarrow t}(x_0) \mathbf{J}^+_{0\rightarrow t}(x_0)S(t)x\rangle_E\\
& &\\
&=&\langle \mathbf{J}^{+,*}_{0\rightarrow t}(x_0)\mathbf{J}^*_{0\rightarrow t}(x_0)y,S(t)x\rangle_{S(t)E}=0.
\end{eqnarray*}
This implies $y\in \big(S(t)E\big)^\perp = \{\mathbf{0}\}$ (the orthogonal complement in $E$). Therefore, it is sufficient to check

\begin{equation}
\mathcal{C}_t~\text{is positive definite a.s.}
\end{equation}

Similar to the classical Brownian motion case, we argue by contradiction. Let us suppose there exists $\varphi_0\neq \mathbf{0}$ such that

\begin{equation}\label{contr1}
\mathbb{P}\big\{\langle \mathcal{C}_t \varphi_0, \varphi_0\rangle_{E}=0 \big\}>0.
\end{equation}
Take $\varphi\in E$. By (\ref{Ctoperator}), we have

\begin{equation}\label{exist1}
\langle \mathcal{C}_t\varphi,\varphi\rangle_{E} =  \alpha_H \sum_{\ell=1}^\infty \int_0^t \int_0^t\big\langle \mathbf{J}^+_{0\rightarrow u}(x_0) G_\ell(X^{x_0}_u),\varphi\big\rangle_{E}\big\langle \mathbf{J}^+_{0\rightarrow v}(x_0) G_\ell(X^{x_0}_v),\varphi\big\rangle_{E}|u-v|^{2H-2}dudv.
\end{equation}
Let us define

$$K_s = \overline{\text{span}\{\mathbf{J}^+_{0\rightarrow r}(x_0)G_\ell(X^{x_0}_r); 0\le r\le s,\ell\in \mathbb{N}\}}; 0 < s\le T,$$
and we set $K_{0+}=\cap_{s>0}K_s$. The Brownian filtration $\mathbb{F}$ allows us to make use of the Blumental zero-one law to infer that $K_{0+}$ is deterministic\footnote{We say that a random subset $A\subset E_\kappa$ is deterministic a.s when all random elements $a\in A$ are constant a.s} a.s. Let $N>0$ be a natural number and let $N_s$ be the (possibly infinite) dimension of the quotient space $\frac{K_s}{K_{0+}}$.  Consider the non-decreasing adapted process $\big\{\text{min}\{N,N_s\}, 0 < s \le T\big\}$ and the stopping time

$$S=\inf\big\{0 < s\le T; \text{min}~\{N,N_s\}>0\big\}.$$
One should notice that $S>0$ a.s. If $S=0$ on a set $A$ of positive probability, then for every $\epsilon >0$ there exists $0 < s\le T$ such that

$$\epsilon > s >0~\text{and}~\text{min}\{N_s,N\}>0,$$
on $A$. This means that we should have $N_s>0$ for every $s\in (0,T]$ on $A$. This implies that with a positive probability the dimension of $\frac{K_{0+}}{K_{0+}}$ is strictly positive which is a contradiction.

We now claim that $K_{0+}$ is a proper subset of $E$. Otherwise, $K_{0+} = E$ which implies $K_s = E$ for every $0 < s\le T$. In this case, if $\varphi \in E$ is such that $\langle \mathcal{C}_t \varphi,\varphi\rangle_{E}=0$ with positive probability, then  $\big\langle \mathbf{J}^+_{0\rightarrow r}(x_0) G_\ell(X^{x_0}_r),\varphi\big\rangle_{E}=0$ for every $r\in [0,t]$ and $\ell\in \mathbb{N}$ with positive probability which in turn would imply that $\varphi\in K_s^\perp = E^\perp$ so that $\varphi=\mathbf{0}$. This contradicts (\ref{contr1}). Now we are able to select a non-null $\varphi\in E^*$ such that $K_{0+}\subset \text{Ker}\varphi$. At first, we observe $\varphi(K_s)=0$ for every $0\le s < S$ so that

\begin{equation}\label{exists2}
\langle \mathbf{J}^+_{0\rightarrow s}G_\ell(X^{x_0}_s), \varphi\rangle_{E}=0~\forall \ell\ge 1~\text{and}~0\le s< S.
\end{equation}
We claim
\begin{equation}\label{exists3}
\langle \mathbf{J}^+_{0\rightarrow s}(x_0) V(X^{x_0}_s),\varphi \rangle_{E}=0~\text{for every}~0\le s < S, V\in \mathcal{V}_k, k\ge 0,
\end{equation}
where we observe $V$ in (\ref{exists3}) takes values on $S(T)E$. We show (\ref{exists3}) by induction. For $k=0$, (\ref{exists2}) implies (\ref{exists3}). Let us assume (\ref{exists3}) holds for $k-1$. Let $V\in \mathcal{V}_{k-1}$. Then, we have

\begin{eqnarray}
\nonumber 0&=&\langle \mathbf{J}^+_{0\rightarrow s}(x_0) V(X^{x_0}_s),\varphi \rangle_{E}\\
\nonumber & &\\
\nonumber &=&\langle V(x_0),\varphi\rangle_{E}+\int_0^s \langle \mathbf{J}^+_{0\rightarrow r}(x_0)[G_0,V](X^{x_0}_r),\varphi\rangle_{E}dr\\
\nonumber & &\\
\nonumber &+& \sum_{\ell=1}^\infty \sqrt{\lambda_\ell}\int_0^s \langle \mathbf{J}^+_{0\rightarrow r}(x_0)[G_\ell,V](X^{x_0}_r),\varphi \rangle_{E}d\beta^\ell_r,
\end{eqnarray}
where $\langle V(x_0),\varphi\rangle_{E}=0$ by the induction hypothesis. By Theorem \ref{doobmeyer}, we must have

$$\langle \mathbf{J}^+_{0\rightarrow r}(x_0)[G_0,V](X^{x_0}_r),\varphi\rangle_{E}=\langle \mathbf{J}^+_{0\rightarrow r}(x_0)[G_\ell,V](X^{x_0}_r),\varphi \rangle_{E}=0,$$
for every $0\le r\le s$ and $0\le s < S$ and $\ell\ge 1 $. This proves (\ref{exists3}). Clearly, (\ref{exists3}) implies

\begin{equation}\label{exist4}
\varphi(\mathcal{V}_k(x_0))=0~\text{for every non-negative integer}~k,
\end{equation}
and hence the H\"{o}rmander's bracket condition implies $\varphi=\mathbf{0}$. By Th 2.1.1 in \cite{nualart}, we then conclude the proof.
\end{proof}
\begin{remark}\label{oneexample}
The assumption that $S(t)E$ is dense in $E$ seems a bit restrictive but it covers a rather general class of examples. For instance, if $(A, \text{dom}~(A))$ is a densely defined self-adjoint operator such that

$$\sup_{x\in \text{dom}(A) \setminus \{0\}}\frac{\langle x,Ax\rangle_E}{\|x\|^2_E} < \infty,$$
then $(A,\text{dom}~A)$ is the generator of a self-adjoint analytic semigroup (see Th 7.3.4 and Example 7.4.5 in \cite{buhler}). Since analytic semigroups are one-to-one,  $S^*(t)$ is one-to-one for every $t\ge 0$ and hence, $S(t)E$ is dense in $E$ for every $t\ge 0$. The heat semigroup on $L^2$ has dense range (see \cite{fabre}). More generally, assume there exists a separable Hilbert space $W$ densely and continuously embedded
into $E$ with compact imbedding. Assume that

\begin{itemize}
  \item $A:W\rightarrow W^*$ is continuous and its restriction to $W$, $A_E:\text{dom}(A_E)\rightarrow E$ where $\text{dom}(A_E) = \{u\in W; Au\in E\}$ and $A_Eu=Au; u\in \text{dom}(A_E)$, is a self-adjoint operator.
  \item There exists $\lambda\in \mathbb{R}$ and $\eta>0$ such that

  $$\big( Au,u\big)_{W,W^*} + \lambda \|u\|^2_E\ge \eta \|u\|_W^2, $$
for each $u\in W$.
\end{itemize}
Then, $S(t)E$ is dense in $E$ for every $t\in [0,T]$. See e.g \cite{bejenaru} for further details.
\end{remark}






\end{document}